\documentclass[11pt,a4paper]{amsart}
\usepackage{amsmath}
\usepackage{amssymb}
\usepackage{amsfonts}
\usepackage{amscd}
\usepackage{amsthm}
\usepackage{xypic}
\usepackage{enumerate}
\usepackage{fancyhdr}
\usepackage{hyperref}

\input xy
\xyoption{all}

\newcommand{\A}{\mathbb{A}}

\newcommand{\C}{\mathbb{C}}

\newcommand{\K}{\mathbb{K}}

\renewcommand{\P}{\mathbb{P}}
\newcommand{\Q}{\mathbb{Q}}
\newcommand{\R}{\mathbb{R}}
\newcommand{\Z}{\mathbb{Z}}

\newcommand{\mmD}{\mathcal{D}}

\newcommand{\mmK}{\mathcal{K}}

\newcommand{\mmP}{\mathcal{P}}

\newcommand{\mmZ}{\mathcal{Z}}

\newcommand{\mcH}{\mathcal{H}}

\newcommand{\whZ}{\widehat{Z}}

\newcommand{\whmZ}{\widehat{\mmZ}}

\DeclareMathOperator*{\Div}{div}

\DeclareMathOperator*{\Hom}{Hom}\DeclareMathOperator*{\im}{im}

\DeclareMathOperator*{\Id}{Id}
\DeclareMathOperator{\amap}{a}

\numberwithin{equation}{section}
\newtheorem{theo}[equation]{Theorem}
\newtheorem{prop}[equation]{Proposition}
\newtheorem{cor}[equation]{Corollary}
\newtheorem{lema}[equation]{Lemma}

\theoremstyle{definition}
\newtheorem{df}[equation]{Definition}
\newtheorem{obs}[equation]{Remark}

\newenvironment{enumerate*}[1][{}]{\begin{itemize}\renewcommand{\labelitemi}{#1}}{\end{itemize}}

\title{Higher arithmetic Chow groups}
\author{J. I. Burgos Gil}
\address{Instituto de Ciencias Matem\'aticas CSIC, Spain}
\email{jiburgosgil@gmail.com}

\author{E .Feliu}
\address{Facultat de Matem\`atiques, Universitat de Barcelona, Spain}
\email{efeliu@ub.edu}

\date{\today}
\thanks{This work was partially supported by the project MTM2006-14234-C02-01}
\newif\ifprivate
\privatetrue

\begin{document}

\begin{abstract}
We give a new construction of \emph{higher arithmetic Chow groups}
for quasi-projective arithmetic varieties over a field.
Our definition agrees with the higher arithmetic Chow groups defined
by Goncharov for projective arithmetic varieties over a field.
These groups are the analogue, in the Arakelov context, of the higher
algebraic Chow groups defined by Bloch. The degree zero group agrees
with the arithmetic Chow groups of Burgos. Our new construction is
shown to be a contravariant functor and is endowed with a product
structure, which is commutative and associative.

\end{abstract}

\maketitle

\section*{Introduction}
Let $X$ be an arithmetic variety, i.e. a regular scheme which is
flat and quasi-projective over an arithmetic ring. In
\cite{GilletSouleIHES}, Gillet and Soul{\'e} defined the
\emph{arithmetic Chow groups} of $X$, denoted as
$\widehat{CH}^p(X)$, whose elements are classes of pairs
$(Z,g_Z)$, with $Z$ a codimension $p$ subvariety of $X$ and $g_Z$
a Green current for $Z$. Later, in \cite{Burgos2}, the first author
gave an
alternative definition for the arithmetic Chow groups, involving
the Deligne complex of differential forms with logarithmic
singularities along infinity, $\mmD^{*}_{\log}(X,p)$, that
computes real Deligne-Beilinson cohomology,
$H_{\mmD}^{*}(X,\R(p))$. When $X$ is proper, the two definitions are
related by a natural isomorphism that takes into account the different
normalization of both definitions. In this paper, we follow the latter
definition.

It is shown in \cite{Burgos2} that the following properties are
satisfied by $\widehat{CH}^{p}(X)$:
\begin{itemize}
\item The groups $\widehat{CH}^p(X)$ fit into an exact sequence:
\begin{equation*}\label{introchowseq}\tag{1}
CH^{p-1,p}(X) \xrightarrow{\rho} \mmD^{2p-1}_{\log}(X,p)/\im d_{\mmD}
\xrightarrow{a} \widehat{CH}^p(X) \xrightarrow{\zeta} CH^p(X)
\rightarrow 0,
\end{equation*}
 where $CH^{p-1,p}(X)$ is the term $E_2^{p-1,-p}(X)$ of the Quillen
spectral sequence (see \cite{Quillen}, \S 7) and  $\rho$ is the
Beilinson regulator.
\item There is a
pairing
$$\widehat{CH}^{p}(X) \otimes \widehat{CH}^{q}(X)  \xrightarrow{{\cdot}}
 \widehat{CH}^{p+q}(X)_{\Q}$$
turning $\bigoplus_{p\geq 0}\widehat{CH}^p(X)_{\Q}$ into a
commutative graded unitary $\Q$-algebra. \item If $f:X\rightarrow Y$
is a morphism, there exists a
pull-back morphism
$$f^*: \widehat{CH}^p(Y) \rightarrow \widehat{CH}^{p}(X).$$
\end{itemize}

Assume that $X$ is proper and defined over an arithmetic field.
Then the arithmetic Chow groups have been extended to higher
degrees by Goncharov, in \cite{Goncharov}. These groups are
denoted by $\widehat{CH}^p(X,n)$ and are constructed in order to
extend the exact sequence \eqref{introchowseq} to a long exact
sequence of the form
\begin{align*}
\cdots \rightarrow \widehat{CH}^{p}(X,n) \xrightarrow{\zeta}
  CH^{p}(X,n) \xrightarrow{\rho}  H_{\mmD}^{2p-n}(X,\R(p))
 \xrightarrow{a} \widehat{CH}^{p}(X,n-1)  \rightarrow \cdots \\ \cdots \rightarrow CH^p(X,1)
   \xrightarrow{\rho}
\mmD^{2p-1}_{\log}(X,p) / \im d_{\mmD} \xrightarrow{a}
\widehat{CH}^p(X) \xrightarrow{\zeta} CH^p(X)  \rightarrow 0. \end{align*}

Explicitly, Goncharov defined a regulator morphism
$$ Z^p(X,*) \xrightarrow{\mmP}\  \mmD^{2p-*}_{D}(X,p),$$
where
\begin{itemize}
 \item $Z^p(X,*)$ is the chain complex given by Bloch in
   \cite{Bloch1}, whose homology groups are, by definition,
$CH^p(X,*)$.
\item $\mmD^{*}_{D}(X,\ast)$ is the Deligne complex of currents.
\end{itemize}
Then the \emph{higher arithmetic Chow groups} of a regular complex
variety $X$ are defined as $\widehat{CH}^{p}(X,n):=H_n(s(\mmP'))$, the
homology groups of the simple of the
induced morphism
$$\mmP': Z^p(X,*) \xrightarrow{\mmP}\
\mmD^{2p-*}_{D}(X,p)/\mmD^{2p}(X,p).$$

For $n=0$, these groups agree with the ones given by Gillet and
Soul{\'e}. However, this construction leaves the following
questions open:
\begin{enumerate}[(1)]
\item Does the composition of the isomorphism $K_n(X)_{\Q}\cong
\bigoplus_{p\geq 0} CH^p(X,n)_{\Q}$ with the morphism induced by
$\mmP$ agree with the Beilinson regulator? \item Can one define a product
structure on $\bigoplus_{p,n} \widehat{CH}^{p}(X,n)$? \item Are
there well-defined pull-back morphisms?
\end{enumerate}
The use of the complex of currents in the definition of $\mmP$ is the
main obstacle encountered when trying to answer these questions, since
this complex does not behave well under pull-back or
products. Moreover, the usual techniques for the comparison of
regulators
apply to morphisms defined for the class of quasi-projective
varieties, which is not the case of $\mmP$.

In this paper we develop a higher arithmetic intersection theory by
giving a new definition of the higher arithmetic Chow groups, based on
a representative of the Beilinson regulator at the chain complex
level. Our strategy has been to use the Deligne complex of
differential forms instead of the Deligne complex of currents in the
construction of the representative of the Beilinson
regulator. The obtained regulator turns out to be a minor modification
of the regulator described by Bloch in \cite{Bloch4}.

The present definition of higher arithmetic Chow groups is valid for
quasi-projective arithmetic varieties over a field, pull-back
morphisms are well-defined and can be given a commutative and
associative product structure. Therefore, this construction overcomes
the open questions left by Goncharov's construction.

In a paper under preparation, the authors, jointly with Takeda, prove
that this definition agrees with Goncharov's definition when the
arithmetic variety is projective. Moreover, by a direct comparison of
our regulator with $\mmP$, it is also proved that the regulator
defined by Goncharov induces the Beilinson regulator. In this way, the
open questions (1)-(3) are answered positively. Moreover, the question
of the covariance of the higher arithmetic Chow groups with respect to
proper morphisms will also be treated elsewhere.

Note that since the theory of higher algebraic Chow groups given by Bloch,
$CH^{p}(X,n)$, is only fully established for schemes over a field,
we have to restrict ourselves to arithmetic varieties over a
field. Therefore, the following question remains open:
\begin{enumerate}
\item Can we extend
the definition to arithmetic varieties over an arithmetic ring?
\end{enumerate}

\medskip

Let us now briefly describe the constructions  presented in this
paper. First, for the construction of the higher Chow groups, instead
of using the simplicial complex defined by Bloch  in
\cite{Bloch1}, we use its cubical
analog, defined by
Levine in \cite{Levine1}, due to its suitability for describing
the product structure on $CH^*(X,*)$. Thus $Z^p(X,n)_{0}$ will denote
the normalized chain complex associated to a cubical abelian group.
Let $X$ be a complex algebraic manifold. For every $p\geq 0$, we
define two cochain complexes, $\mmD_{\A,\mmZ^p}^*(X,p)_0$ and
$\mmD_{\A}^*(X,p)_0$, constructed out of differential forms on
$X\times \square^n$ with logarithmic singularities along infinity
($\square=\P^1\setminus\{1\}$). For every $p\geq 0$, the following
isomorphisms are satisfied:
$$\begin{array}{rclcr} H^{2p-n}(\mmD_{\A,\mmZ^p}^*(X,p)_0) & \cong &
  CH^p(X,n)_{\R},& \qquad & n\geq 0,
   \\
H^{r}(\mmD_{\A}^*(X,p)_0) & \cong & H_{\mmD}^{r}(X,\R(p)), & \qquad &
r\leq 2p,
\end{array}$$
where the first isomorphism is obtained by a explicit
quasi-isomorphism
\begin{displaymath}
  \mmD_{\A,\mmZ^p}^{2p-*}(X,p)_0\longrightarrow
  Z^{p}(X,\ast)_{0}\otimes{\mathbb{R}}
\end{displaymath}
(see $\S$\ref{diffformsaffine} and $\S$\ref{diffformschow}).

We show that there is a natural chain morphism (see $\S$\ref{regulator3})
$$\mmD_{\A,\mmZ^p}^{2p-*}(X,p)_0 \xrightarrow{\rho} \mmD_{\A}^{2p-*}(X,p)_0 $$
which induces, after composition with the isomorphism
$$K_{n}(X)_{\Q}\cong \bigoplus_{p\geq 0} CH^p(X,n)_{\Q}$$ described
by Bloch in \cite{Bloch1}, the \emph{Beilinson regulator} ({\bf
Theorem \ref{beichow}}):
$$K_{n}(X)_{\Q} \cong \bigoplus_{p\geq 0} CH^p(X,n)_{\Q} \xrightarrow{\rho} \bigoplus_{p\geq 0}
H_{\mmD}^{2p-n}(X,\R(p)).$$

In the second part of this paper we use the morphism $\rho$ to
define the higher arithmetic Chow group $\widehat{CH}^p(X,n)$, for
any arithmetic variety $X$ over a field. The formalism underlying
our definition is the theory of diagrams of complexes and their
associated simple complexes, developed by Beilinson in
\cite{Beilinson}. Let $X_{\Sigma }$ denote the complex manifold
associated with $X$ and let $\sigma $ be the involution that acts as
complex conjugation on the space and on the coefficients. As usual
$\sigma $ as superindex will mean the fixed part under $\sigma $.
Then one considers the diagram of chain
complexes
{\small $$
\widehat{\mmZ}^p(X,*)_0=\left(\begin{array}{c}\xymatrix@C=0.5pt{
& Z^p(X_{\Sigma },*)^{\sigma }_0\otimes \mathbb{R} & &
\mmD_{\A}^{2p-*}(X_{\Sigma },p)^{\sigma }_0 \\
Z^{p}(X,*)_{0} \ar[ur]^{f_1} & & \ar[ul]_{g_1}^{\sim}
\mmD^{2p-*}_{\A,\mathcal{Z}^p}(X_{\Sigma },p)^{\sigma }_0 \ar[ur]^{\rho} &&
Z\mmD^{2p}(X_{\Sigma },p)^{\sigma } \ar[ul]_{i} }
\end{array}\right)$$}
where $ Z\mmD^{2p}(X_{\Sigma },p)^{\sigma }$ is the group of closed elements of
$D^{2p}(X_{\Sigma },p)^{\sigma }$ considered as a complex concentrated in degree $0$.
Then, the higher arithmetic Chow groups of $X$ are given by the
homology groups of the simple of the diagram
$\widehat{\mmZ}^p(X,*)_0$ ({\bf Definition \ref{definitionchow}}):
$$\widehat{CH}^p(X,n):= H_n(s(\widehat{\mmZ}^p(X,*)_0)). $$

The following properties are shown:
\begin{itemize}
\item {\bf Theorem \ref{chowagreement}}: Let $\widehat{CH}^p(X)$
denote the arithmetic Chow group defined in \cite{Burgos2}. Then, there is
a natural isomorphism
$$
\widehat{CH}^p(X)   \xrightarrow{\cong}  \widehat{CH}^p(X,0). $$
It follows that  if $X$ is proper, $\widehat{CH}^p(X,0)$ agrees with
the arithmetic Chow group
defined by Gillet and Soul{\'e} in \cite{GilletSouleIHES}.
\item {\bf Proposition \ref{zeta}}: There is a long exact sequence
\begin{align*}
\cdots \rightarrow \widehat{CH}^{p}(X,n) \xrightarrow{\zeta}
  CH^{p}(X,n) \xrightarrow{\rho}  H_{\mmD}^{2p-n}(X_{\Sigma },\R(p))^{\sigma }
 \xrightarrow{a} \widehat{CH}^{p}(X,n-1)  \rightarrow \cdots \\ \cdots
 \rightarrow CH^p(X,1)
   \xrightarrow{\rho}
\mmD^{2p-1}_{\log}(X_{\Sigma },p)^{\sigma } / \im d_{\mmD} \xrightarrow{a}
\widehat{CH}^p(X) \xrightarrow{\zeta} CH^p(X)  \rightarrow 0, \end{align*}
with $\rho$ the Beilinson regulator.
\item {\bf Proposition \ref{chowpullbackth}} (\emph{Pull-back}):
Let $f:X\rightarrow Y$ be a morphism between two arithmetic
varieties over a field. Then, there is a pull-back morphism
$$\widehat{CH}^p(Y,n) \xrightarrow{f^*} \widehat{CH}^p(X,n), $$ for every
$p$ and $n$, compatible with the pull-back maps on the groups
$CH^p(X,n)$ and $H_{\mmD}^{2p-n}(X,\R(p))$. \item {\bf Corollary
\ref{homotopyinvariancecor}} (\emph{Homotopy invariance}): Let
$\pi: X\times \A^m \rightarrow X$ be the projection on $X$. Then,
the pull-back map
$$\pi^*: \widehat{CH}^p(X,n)  \rightarrow \widehat{CH}^p(X\times \A^m,n),\quad n\geq 1   $$
is an isomorphism.
 \item {\bf Theorem \ref{chowproductth}} (\emph{Product}): There exists a product
on
$$\widehat{CH}^*(X,*):=\bigoplus_{p\geq 0,n\geq 0} \widehat{CH}^p(X,n), $$ which is associative, graded commutative with
respect to the degree $n$.
\end{itemize}

\medskip

The paper is organized as follows.
The first section is a preliminary section. It is devoted to fix the notation and state the main facts
used in the rest of the paper. It includes general results on homological algebra, diagrams of complexes, cubical abelian groups and Deligne-Beilinson cohomology. In the second section we recall the
definition of the higher Chow groups of Bloch and introduce the complexes of differential forms
being the source and target of the regulator map. We proceed in the next section to the definition
of the regulator $\rho$ and we prove that it agrees with Beilinson's
regulator. In sections 4 and 5, we
develop the theory of higher arithmetic Chow groups. Section 4 is
devoted to the definition and basic properties of the higher arithmetic Chow groups
 and to the comparison with the arithmetic Chow group for $n=0$. Finally, in section 5 we
define the product structure on $\widehat{CH}^*(X,*)$ and prove
that it is commutative and associative.

\emph{Acknowledgments.} During the elaboration of this paper, the second author spent an academic year in the University of Regensburg
with a pre-doc grant from the European Network ``Arithmetic
Algebraic Geometry''. She wants to thank all the members of the
Arithmetic Geometry group, specially U. Jannsen and K. K{\"u}nneman.
We would also like to acknowledge M. Levine and H. Gillet for many useful conversations on the subject of this paper.

\section{Preliminaries}

\subsection{Notation on (co)chain complexes}
We use the standard conventions on (co)chain complexes. By a (co)chain
complex we mean a (co)chain complex over the category of abelian
groups.

The \emph{cochain complex associated to a chain complex} $A_*$ is simply denoted by
$A^*$ and the \emph{chain complex associated to a cochain
complex} $A^*$ is denoted by $A_*$. The \emph{translation of a cochain
complex} $(A^*,d_A)$ by an integer $m$ is denoted by $A[m]^*$.
Recall that $A[m]^{n}=A^{m+n}$ and the differential of $A[m]^*$ is
$(-1)^m d_A$. If $(A_*,d_A)$ is a \emph{chain complex}, then the
translation of $A_*$ by an integer $m$ is denoted by $A[m]_*$. In this
case the differential is also $(-1)^m d_A$ but $A[m]_{n}=A_{n-m}$.

The \emph{simple complex} associated to an iterated chain complex
$A_*$ is denoted by $s(A)_*$ and the analogous notation is used for
the simple complex associated to an iterated cochain complex (see
\cite{BurgosKuhnKramer} $\S$2 for definitions).

The \emph{simple of a cochain map} $A^*\xrightarrow{f} B^*$ is the
cochain complex $(s(f)^*,d_s)$ with
$s(f)^n=A^n\oplus B^{n-1}$, and differential $d_s(a,b)=(d_A a,f(a)-d_Bb)$.
Note that this complex is the cone of $-f$ shifted by 1. There is an
associated long exact sequence
\begin{equation}\label{longsimple}
\dots \rightarrow H^n(s(f)^*) \rightarrow H^n(A^*) \xrightarrow{f}
H^n(B^*) \rightarrow H^{n+1}(s(f)^*) \rightarrow \cdots
\end{equation}
If $f$ is surjective, there is a quasi-isomorphism
\begin{equation}\label{simplekernel}
  \ker f  \xrightarrow{i}  s(-f)^*   \qquad  x \mapsto  (x,0),
\end{equation}
and if $f$ is injective, there is a quasi-isomorphism
\begin{equation}\label{simplequotient}
  s(f)[1]^* \xrightarrow{\pi} B^*/A^*   \qquad  (a,b) \mapsto  [b].
\end{equation}
Analogously, equivalent results and quasi-isomorphisms can be stated for chain complexes.

Following Deligne \cite{DeligneHodgeII}, given a cochain complex $A^*$
and an integer $n$, we denote by  $\tau_{\leq n}A^*$ the
\emph{canonical truncation of $A^*$ at degree $n$}.

\subsection{The simple of a diagram of complexes}\label{diagrambei4}
We describe here Beilinson's ideas on the simple complexes associated to a diagram
of complexes (see \cite{Beilinson}).
A diagram of chain complexes is a diagram of the form
{\small
\begin{equation}\label{diagrambei2}
\mathcal{D}_*=\left(\begin{array}{c}\xymatrix{
& B^1_* & & B^2_*  &  \dots & \dots   & B^n_*  \\
A^1_* \ar[ur]^{f_1} & & \ar[ul]_{g_1}  A^2_* \ar[ur]^{f_2}  & \dots & \dots &  A^{n}_*  \ar[ur]^{f_n} & & A^{n+1}_* \ar[ul]_{g_n}
}
\end{array}\right).\end{equation}}

Consider the induced chain morphisms
\begin{equation}
  \label{eq:1}
\bigoplus_{i=1}^{n+1} A^i_*
\xrightarrow{\varphi,\varphi_{1},\varphi_{2}}  \bigoplus_{i=1}^{n}
B^i_*,\qquad
\begin{aligned}
  \varphi_{1}(a_i)&=f_i(a_i) \quad \textrm{if }a_i\in A^i_*,\\
  \varphi_{2}(a_i)&=g_{i-1}(a_i) \quad \textrm{if }a_i\in A^i_*,\\
  \varphi(a_i)&=(\varphi_{1}-\varphi_{2})(a_{i})=(f_i-g_{i-1})(a_i)
  \quad \textrm{if }a_i\in A^i_*.
\end{aligned}
\end{equation}
(where we set $f_{n+1}=g_{0}=0$).
The \emph{simple complex associated to the diagram} $\mmD_*$ is defined to be the simple of the morphism $\varphi$:
\begin{equation}\label{diagrambei7}
 s(\mmD)_*:= s(\varphi)_*.
\end{equation}

\subsection{Morphisms of diagrams}\label{diagrambei5} A \emph{morphism between two diagrams} $\mmD_*$ and $\mmD'_*$
consists of a collection of morphisms
$$A^i_* \xrightarrow{h_i^A} A_{*}^{'i},\qquad B^i_* \xrightarrow{h_i^B} B^{'i}_{*},  $$
commuting with the morphisms $f_i$ and $g_i$, for all $i$. Any  morphism of diagrams $\mmD_*\xrightarrow{h} \mmD'_*$ induces
a morphism on the associated simple complexes $s(\mmD)_*\xrightarrow{s(h)} s(\mmD')_*.$
Observe that if, for every $i$, $h_i^A$ and
$h_i^B$ are quasi-isomorphisms, then $s(h)$ is also a
quasi-isomorphism.

\subsection{Product structure on the simple of a
diagram} \label{productdiagram}
Let $\mmD_*$ and $\mmD'_*$ be two diagrams as \eqref{diagrambei2}.
Consider the diagram obtained by the tensor product of complexes:
{\small
\begin{equation}\label{diagrambei3} (\mathcal{D}\otimes
\mmD')_*=\left(\begin{array}{c}\xymatrix@C=-5pt@R=30pt{
& B^1_*\otimes B^{'1}_* & & B^2_*\otimes B^{'2}_* &  \dots & \dots   & B^n_* \otimes B^{'n}_*    \\
A^1_*\otimes A_*^{'1} \ar[ur]^{f_1\otimes f_1'} & &
\ar[ul]_{g_1\otimes g_1'} A^2_*\otimes A^{'2}_*
\ar[ur]_{f_2\otimes f_2'}  & \dots & \dots &  A^{n}_*\otimes A^{'n}_*  \ar[ur]^{f_n\otimes f_n'} & & A^{n+1}_*\otimes A^{'n+1}_* \ar[ul]_{g_n\otimes g_n'}  }
\end{array}\right).\end{equation}}
In \cite{Beilinson},  Beilinson defined, for every $\beta\in \Z$, a morphism
$$s(\mmD)_*\otimes s(\mmD')_* \xrightarrow{\star_{\beta}} s(\mmD\otimes \mmD')_*$$
 as follows. For $a\in A,a'\in A',b\in B$ and $b'\in
B'$, set:
\begin{eqnarray*}
a\star_{\beta} a'&=& a\otimes a', \\ b \star_{\beta } a'& = &
b\otimes ((1-\beta)\varphi_1(a')+\beta \varphi_2(a')), \\
a\star_{\beta} b' & =&
(-1)^{\deg a}  (\beta \varphi_1(a)+(1-\beta)\varphi_2(a))\otimes b', \\
b\star_{\beta} b' &=& 0,
\end{eqnarray*}
where the tensor product between elements in different spaces is
defined to be zero.

If $B_*,C_*$ are chain complexes, let
$$\sigma: s(B_*\otimes C_*) \rightarrow s(C_*\otimes B_*) $$ be
the map sending $b\otimes c \in B_n\otimes C_m$ to $(-1)^{nm}c\otimes
b \in C_m\otimes B_n$.

\begin{lema}[Beilinson]\label{star}
\begin{enumerate}[(i)]
\item The map $\star_{\beta}$ is a morphism of complexes. \item
For every $\beta,\beta'\in \Z$, $\star_{\beta}$ is homotopic to
$\star_{\beta'}$. \item There is a commutative diagram
$$\xymatrix{
s(\mmD)_*\otimes s(\mmD')_{*} \ar[r]^{\star_{\beta}} \ar[d]_{\sigma} &
s(\mmD \otimes \mmD')_* \ar[d]^{\sigma} \\
s(\mmD')_*\otimes s(\mmD)_{*} \ar[r]^{\star_{1-\beta}} & s(\mmD'
\otimes \mmD)_*. }$$  \item The products $\star_{0}$ and
$\star_1$ are associative.
\end{enumerate}
\end{lema}

\subsection{A specific type of diagrams}
In this work we will use diagrams of the following form:
{\small
\begin{equation}\label{ourdiagrambei2}
\mathcal{D}_*=\left(\begin{array}{c}\xymatrix{
& B^1_* & & B^2_*   \\
A^1_* \ar[ur]^{f_1} & & \ar[ul]_{g_1}^{\sim}  A^2_* \ar[ur]^{f_2}
}
\end{array}\right),\end{equation}}
with $g_1$ a quasi-isomorphism. For this type of diagrams, since $g_1$ is a quasi-isomorphism, we obtain a long exact sequence
equivalent to the long exact sequence related to the simple of a
morphism. Since a diagram like this induces a map
$A^1_*\longrightarrow B^2_*$ in the derived category, we obtain

\begin{lema}\label{sucllargachow}
Let $\mmD_*$ be a diagram like \eqref{ourdiagrambei2}. Then there is
a well-defined morphism
$$H_n(A^1_*)  \xrightarrow{\rho}  H_n(B^2_*), \qquad
\ [a_1]  \mapsto  f_2g_1^{-1}f_1[a_1]. $$
Moreover, there is a long exact sequence
\begin{equation}\label{diagramlong} \cdots \rightarrow
H_n(s(\mmD)_*) \rightarrow H_n(A^1_*) \xrightarrow{\rho}
  H_n(B^2_*) \rightarrow H_{n-1}(s(\mmD)_*)\rightarrow \cdots
  \end{equation}
\end{lema}

Consider now a diagram of the form
{\small
\begin{equation}\label{ourdiagrambei3}
\mathcal{D}_*=\left(\begin{array}{c}\xymatrix{
& B^1_* & & B^2_*   \\
A^1_* \ar[ur]^{f_1} & & \ar[ul]_{g_1}^{\sim}  A^2_* \ar[ur]^{f_2} & &
A^3_* \ar@{>->}[ul]_{g_2}
}
\end{array}\right),\end{equation}}
with $g_1$ a quasi-isomorphism and $g_2$ a monomorphism.

\begin{lema}\label{quasiisodiagrams} Let $\mmD$ be a diagram as \eqref{ourdiagrambei3} and let $\mmD'$ be the diagram
{\small
\begin{equation}
\mathcal{D'}_*=\left(\begin{array}{c}\xymatrix{
& B^1_* & & B^2_*/A^3_*   \\
A^1_* \ar[ur]^{f_1} & & \ar[ul]_{g_1}^{\sim}  A^2_* \ar[ur]^{f_2}
}
\end{array}\right),\end{equation}}
Then, there is a quasi-isomorphism  between the simple complexes associated to $\mmD$ and to $\mmD'$:
$$s(\mmD)_* \xrightarrow{\sim} s(\mmD')_*.$$
\end{lema}
\begin{proof}
It follows directly from the definition that the simple complex associated to $\mmD_*$ is quasi-isomorphic to the simple
associated to the diagram
{\small
\begin{equation}
\mathcal{D''}_*=\left(\begin{array}{c}\xymatrix{
& B^1_* & & s(A^3_*\xrightarrow{g_{2}} B^2_*)[1]   \\
A^1_* \ar[ur]^{f_1} & & \ar[ul]_{g_1}^{\sim}  A^2_*. \ar[ur]^{f_2} }
\end{array}\right),\end{equation}}
Then, the quasi-isomorphism given in \eqref{simplequotient}
induces a quasi-isomorphism
$$s(\mmD')_* \xrightarrow{\sim} s(\mmD'')_*$$
as desired.
\end{proof}

\begin{cor}\label{sucllargachow2}
  For any diagram of the form \eqref{ourdiagrambei3}, there is a long exact sequence
  \begin{equation}\label{diagramlong2} \cdots \rightarrow
H_n(s(\mmD)_*) \rightarrow H_n(A^1_*) \xrightarrow{\rho}
  H_{n-1}(s(g_2)) \rightarrow H_{n-1}(s(\mmD)_*)\rightarrow \cdots
  \end{equation}
  \end{cor}
  \begin{proof}
    It follows from the previous lemma together with Proposition \ref{sucllargachow}.
  \end{proof}

\subsection{Cubical abelian groups and chain complexes}\label{cubical}
Let $C_{\cdot}=\{C_n\}_{n\geq 0}$ be a cubical abelian group with face maps
$\delta_i^j: C_n\rightarrow C_{n-1}$, for $i=1,\dots,n$ and $j=0,1$,
and degeneracy maps $\sigma_i:
C_n\rightarrow C_{n+1}$, for $i=1,\dots,n+1$. Let $D_n\subset C_n$ be
the subgroup of \emph{degenerate elements} of $C_n$, and let
$\widetilde C_{n}=C_{n}/D_{n}$.

Let $C_*$ denote the \emph{associated chain complex}, that is, the
chain complex whose $n$-th graded
piece is $C_n$ and whose differential is given by
$\delta= \sum_{i=1}^n \sum_{j=0,1}(-1)^{i+j}\delta_i^j.$ Thus
$D_{\ast}$ is a subcomplex and $\widetilde C_{\ast}$ is a quotient
complex.
We fix the \emph{normalized chain complex} associated to $C_{\cdot}$,
$NC_*$, to be the chain
complex whose $n$-th graded group is
$NC_n := \bigcap_{i=1}^{n} \ker \delta_i^1,$ and whose
differential is $ \delta=\sum_{i=1}^n (-1)^{i}\delta_i^{0}.$
It is well-known that there is a decomposition of chain complexes
$C_* \cong NC_* \oplus D_* $ giving an isomorphism
$NC_* \cong \widetilde C_{\ast}.$

For certain cubical abelian groups, the normalized chain complex
can be further simplified, up to homotopy equivalence, by considering the
elements which belong to the kernel of all faces but $\delta_1^0$.

\begin{df}\label{normalized2}
Let $C_{\cdot}$ be a cubical abelian  group. Let $N_0C_*$ be the
complex defined by
\begin{equation}
N_0C_n= \bigcap_{i=1}^{n} \ker \delta_i^1 \cap \bigcap_{i=2}^{n}
\ker \delta_i^0,\quad \textrm{and differential }\
\delta=-\delta_1^0.\end{equation}
\end{df}

The proof of the next proposition is analogous to the proof of
Theorem 4.4.2 in \cite{Bloch3}. The result is proved there only
for the cubical abelian group defining  the higher Chow complex
(see $\S$\ref{higherchow} below). We give here the abstract version of
the statement,
valid for a certain type of cubical abelian groups.

\begin{prop}\label{normalized3}
Let $C_{\cdot}$ be a cubical abelian group. Assume that it comes
equipped with a collection of maps
$$h_j: C_n \rightarrow C_{n+1},\qquad j=1,\dots,n, $$
such that, for any $l=0,1$, the following identities are satisfied:
\begin{eqnarray}
\delta_j^1h_j &=& \delta_{j+1}^1h_j =s_j\delta_j^1, \notag \\
\delta_j^0
h_j &=& \delta_{j+1}^0 h_j=id, \label{eq:3}\\
\notag \delta_i^lh_j &=&
\left\{\begin{array}{ll} h_{j-1}\delta_i^l & i<j, \\
h_j \delta_{i-1}^l & i>j+1.
\end{array} \right.
\end{eqnarray}
Then, the inclusion of complexes
$$i: N_0C_* \hookrightarrow NC_* $$
is a homotopy equivalence.
\end{prop}

\begin{proof}
Let $g_j: NC_n \rightarrow NC_{n+1}$ be defined as
$g_j=(-1)^{n-j}h_{n-j}$ if $0\leq j\leq n-1$ and $g_j=0$ otherwise. Then there
is a well-defined morphism of chain complexes
$$H_j=(\Id + \delta g_{j} + g_{j} \delta): NC_* \rightarrow NC_*. $$
This morphism is homotopically equivalent to the identity.

Let $x\in NC_n$ and $0\le j \le n-1$. Then,
\begin{eqnarray*}
\delta h_{n-j}(x) &=& \sum_{i=1}^{n+1}(-1)^{i}\delta_i^0
h_{n-j}(x)\\ &=& \sum_{i=1}^{n-j-1}(-1)^ih_{n-j-1}\delta_i^0(x) +
\sum_{i=n-j+2}^{n+1}(-1)^i h_{n-j}\delta_{i-1}^0(x),
\\ h_{n-j-1}\delta(x) &=& \sum_{i=1}^{n}(-1)^i h_{n-j-1} \delta_i^0(x).
\end{eqnarray*}
Hence,
\begin{eqnarray*}
\delta
g_{j}(x)+g_j\delta(x)=
(-1)^{n-j}\sum_{i=n-j+1}^{n}(-1)^{i-1}h_{n-j}\delta_i^0(x)
+(-1)^{n-j-1}\sum_{i=n-j}^n
(-1)^{i}h_{n-j-1}\delta_i^0(x).
\end{eqnarray*}
We consider the decreasing filtration $G^{\ast}$ of $NC_{\ast}$, given by
\begin{equation}
  \label{eq:2}
  G^{j}NC_{n}=\{x\in NC_{n}\mid \delta _{i}^{0}(x)=0, i > \max(n-j,1)\}.
\end{equation}
Then $G^{0}NC_{\ast}=NC_{\ast}$ and for $j\ge n-1$,
$G^{j}NC_{n}=N_{0}C_{n}$.
If $x\in G^{j+1}NC_{\ast}$, then
$\delta g_{j}(x)+g_j\delta(x)=0$ and thus, $H_j(x)=x$. Moreover, if
$x\in G^{j}NC_{\ast}$, then $H_{j}(x)\in G^{j+1}NC_{\ast}$. Thus,
$H_{j}$ is the projector from $G^{j}NC_{\ast}$ to $G^{j+1}NC_{\ast}$.

Thus, the morphism
$\varphi: NC_\ast \rightarrow N_{0}C_\ast  $ given, on $NC_{n}$, by
$\varphi:=H_{n-2} \circ \cdots \circ H_{0} $ forms a chain morphism
homotopically equivalent to the identity. Moreover $\varphi$ is the
projector from $NC\ast$ to
$N_0C_*$. Hence, $\varphi \circ i$ is the identity of $N_0C_*$
while $i\circ \varphi$ is homotopically equivalent to the identity of
$NC_*$.
\end{proof}

\begin{obs} To every cubical abelian group $C_{\cdot}$
there are associated four chain complexes: $C_*,NC_*,N_0C_*$ and
$\widetilde{C}_*$. In some situations it will be necessary to
consider the cochain complexes associated to these chain
complexes. In this case we will write, respectively,
$C^*,NC^*,N_0C^*$ and $\widetilde{C}^*.$
\end{obs}

\subsection{Cubical cochain complexes}
Let $X^*_{\cdot}$ be a cubical cochain complex. Then, for every $m$,
the cochain complexes $NX_m^*, N_0X_m^*$ and $\widetilde{X}_m^*$
are defined.

\begin{prop}\label{cubchain}
Let $X^*_{\cdot},Y^*_{\cdot}$ be two cubical cochain complexes and
let $f:X^*_{\cdot}\rightarrow Y^*_{\cdot}$ be a morphism. Assume
that for every $m$, the cochain morphism
$$X^*_m\xrightarrow{f_m} Y^{*}_m $$ is a quasi-isomorphism.
Then, the induced morphisms
$$ NX^{*}_m \xrightarrow{f_m}  NY^{*}_m \quad \textrm{and}\quad
\widetilde{X}^{*}_m
\xrightarrow{f_m}  \widetilde{Y}^{*}_m $$
are quasi-isomorphisms.
\end{prop}
\begin{proof}
The proposition follows from the decompositions
\begin{eqnarray*}
 H^r(X^*_{m}) & = &  H^r(NX^*_{m}) \oplus  H^r(DX^*_{m}), \\
 H^r(Y^*_{m}) & = & H^r(NY^*_{m}) \oplus  H^r(DY^*_{m}),
 \end{eqnarray*}
and the fact that $f_m$ induces cochain maps
$$NX^*_m \xrightarrow{f_m} NY^*_m,\qquad DX_m^*\xrightarrow{f_m} DY^*_m.$$
\end{proof}

\begin{prop}\label{cubchain2}
Let $X^*_{\cdot}$ be a cubical cochain complex. Then the natural
morphism
$$ H^r(NX^*_{n}) \xrightarrow{f} NH^r(X^*_{n})$$
is an isomorphism for all $n\geq 0$.
\end{prop}
\begin{proof} The  cohomology groups $ H^r(X^*_{\cdot})$ have a cubical abelian group structure.
Hence there is a decomposition
$$ H^r(X^*_{\cdot})= NH^r(X^*_{\cdot}) \oplus  DH^r(X^*_{\cdot}).$$
In addition, there is a decomposition $X^*_n=NX^*_n\oplus DX^*_n.$
Therefore
$$ H^r(X^*_{\cdot})= H^r(NX^*_{\cdot}) \oplus  H^r(DX^*_{\cdot}).$$
The lemma follows from the fact that the identity morphism in
$H^r(X^*_{\cdot})$ maps $NH^r(X^*_{\cdot})$ to $H^r(NX^*_{\cdot})$
and $DH^r(X^*_{\cdot})$ to $ H^r(DX^*_{\cdot}).$
\end{proof}

\subsection{Deligne-Beilinson cohomology }\label{dbcohomology}\label{deligne}
In this paper we use the definitions and conventions on Deligne-Beilinson cohomology
given in \cite{Burgos2} and \cite{BurgosKuhnKramer}, chapter 5.

One denotes $\R(p)=(2\pi i)^p \cdot \R \subset \C$.
Let $X$ be a complex algebraic manifold and denote by $E_{\log,\R}^*(X)(p)$
the complex of \emph{real differential forms with logarithmic
  singularities along infinity}, twisted
by $p$. Let $(\mathcal{D}_{\log}^*(X,p),d_{\mmD})$ be the Deligne
complex of differential forms with logarithmic singularities, as
described in \cite{Burgos2}. It computes real Deligne-Beilinson
cohomology of $X$, that is,
$$H^n(\mmD^*_{\log}(X,p)) = H_{\mmD}^n(X,\R(p)).$$
This complex is functorial on $X$.

The  product structure in Deligne-Beilinson cohomology can be
described by a cochain morphism on the Deligne
complex (see \cite{Burgos2}):
\begin{eqnarray*}
 \mmD^n(X,p) \otimes \mmD^m(X,q) & \xrightarrow{\bullet} & \mmD^{n+m}(X,p+q) \\
x \otimes y & \mapsto & x\bullet y.
\end{eqnarray*}
This product satisfies the expected relations:
\begin{enumerate}
\item Graded commutativity: \quad  $x\bullet y = (-1)^{nm}y\bullet
x.$
\item Leibniz rule: \quad $d_{\mmD}(x\bullet y)=
d_{\mmD}x\bullet y +(-1)^n x\bullet d_{\mmD}y.$
\end{enumerate}

\begin{prop}\label{commutassoc}
The Deligne product $\bullet$  is associative up to a natural
homotopy, i.e. there exists
$$h: \mmD^{r}(X,p)\otimes \mmD^{s}(X,q)\otimes \mmD^{t}(X,l)
\rightarrow \mmD^{r+s+t}(X,p+q+l) $$
such that
$$d_{\mmD}h(\omega_1 \otimes \omega_2\otimes \omega_3) +
hd_{\mmD}(\omega_1 \otimes \omega_2\otimes \omega_3)
= (\omega_1 \bullet \omega_2)\bullet \omega_3 - \omega_1 \bullet(
\omega_2\bullet \omega_3).  $$
Moreover, if $\omega_1\in \mmD^{2p}(X,p)$,
$\omega_2\in \mmD^{2q}(X,q)$ and $\omega_3\in \mmD^{2l}(X,l)$
satisfy $d_{\mmD}\omega_i=0$ for all $i$, then
\begin{equation}\label{commutassoc2}
h(\omega_1 \otimes \omega_2\otimes \omega_3)=0.
\end{equation}
\end{prop}
\begin{proof}
This is \cite{Burgos2}, Theorem 3.3.
\end{proof}

\subsection{Cohomology with supports}\label{supports} Let $Z$ be a
closed subvariety of a complex algebraic manifold $X$. Consider the
complex
$\mathcal{D}_{\log}^*(X\setminus Z,p), $
i.e. the Deligne complex of differential forms in $X\setminus Z$
with logarithmic singularities along $Z$ and infinity.

\begin{df}
The \emph{ Deligne complex with supports in $Z$} is
defined to be
$$\mmD_{\log,Z}^*(X,p)=s(\mathcal{D}_{\log}^*(X,p)\rightarrow
\mathcal{D}_{\log}^*(X\setminus Z,p)). $$
The \emph{Deligne-Beilinson cohomology with supports
in $Z$} is defined as the cohomology groups of the Deligne complex
with supports in $Z$:
$$H_{\mmD,Z}^n(X,\R(p)):= H^n(\mmD_{\log,Z}^*(X,p)).$$
\end{df}

 \begin{lema}\label{delignesupports}
Let $Z,W$  be two closed subvarieties of a complex algebraic
manifold $X$. Then there is a short exact sequence of Deligne
complexes,
$$0\rightarrow \mmD^*_{\log}(X\setminus Z\cap W,p)
\xrightarrow{i} \mmD^*_{\log}(X\setminus Z,p)\oplus
\mmD^*_{\log}(X\setminus W,p)\xrightarrow{j}
\mmD^*_{\log}(X\setminus Z\cup W,p)\rightarrow 0, $$ where
$i(\alpha)=(\alpha,\alpha)$ and $j(\alpha,\beta)=-\alpha+\beta$.
\end{lema}
\begin{proof}
It follows from \cite{BurgosKuhnKramer2}, Theorem 3.6.
\end{proof}

In addition, Deligne-Beilinson cohomology with supports satisfies
a \emph{semipurity property}.  Namely, let $Z$ be a codimension
$p$ subvariety of an equidimensional complex manifold $X$, and let
$Z_1,\dots,Z_r$ be its codimension $p$ irreducible components. Then
\begin{equation}\label{semipurity}
H_{\mmD,Z}^n(X,\R(p))=\left\{\begin{array}{ll} 0 & n<2p,  \\
\bigoplus_{i=1}^r \R[Z_i] & n=2p. \end{array}\right.
\end{equation}

For the next proposition, let $\delta_{Z}$ denote the current
integration along an irreducible variety $Z$. In the sequel we will
use the conventions of  \cite{BurgosKuhnKramer} \S 5.4 with respect to
the current associated to a locally integrable form and to the current
$\delta _{Z}$.

\begin{prop}\label{cycleclass2} Let $X$ be an equidimensional complex
  algebraic manifold and $Z$ a
codimension $p$ irreducible subvariety of $X$. Let
$j:X\rightarrow \overline{X}$ be a smooth compactification of $X$
(with a normal crossing divisor as its complement) and
$\overline{Z}$ the closure of $Z$ in $\overline{X}$.   The
isomorphism
$$cl: \R[Z]\xrightarrow{\cong} H_{\mmD,Z}^{2p}(X,\R(p))$$ sends
$[Z]$ to $[(j^*w,j^*g)]$, for any $[(w,g)]\in
H_{\mmD,\overline{Z}}^{2p}(\overline{X},\R(p))$ satisfying the
relation of currents  in $\overline{X}$
\begin{equation}\label{cycleclass}
-2\partial\bar{\partial}[g]=[w]-\delta_{\overline{Z}}.
\end{equation}
\end{prop}
\begin{proof}
See \cite{BurgosKuhnKramer}, Proposition 5.58.
\end{proof}

In particular, assume that $Z=\Div(f)$ is a principal divisor,
where $f$ is a rational  function on $X$. Then $[Z]$ is
represented by the couple
$$(0,-\frac{1}{2}\log(f\bar{f}))\in H_{\mmD,Z}^{2p}(X,\R(p)).$$

The definition of the cohomology with support in a subvariety can
be extended to the definition of the cohomology with support in a
set of subvarieties of $X$. We explain here the case used in the
sequel. Let $\mathcal{Z}^p$ be a subset of the set of codimension
$p$ closed subvarieties of $X$, that is closed under finite
unions. The inclusion of subsets turns
$\mathcal{Z}^p$ into a directed ordered set. We
define the complex
\begin{equation}
\mathcal{D}_{\log}^*(X\setminus
\mmZ^p,p):=\lim_{\substack{\rightarrow \\ Z\in
\mathcal{Z}^p}}\mathcal{D}_{\log}^*(X\setminus Z,p),
\end{equation}
which is provided with an injective map
$$\mathcal{D}_{\log}^*(X,p) \xrightarrow{i}
\mathcal{D}_{\log}^*(X\setminus \mmZ^p,p). $$ As above, we define
$$\mmD_{\log,\mathcal{Z}^p}^*(X,p):=s(i)^*$$
and  the Deligne-Beilinson cohomology with supports in $\mmZ^p$ as
$$H_{\mmD,\mathcal{Z}^p}^n(X,\R(p)):= H^n(\mmD_{\log,\mathcal{Z}^p}^*(X,p)).$$

\subsection{Real varieties}\label{realvarieties}
A \emph{real variety}  $X$ consists of a couple
$(X_{\C},F_{\infty})$, with $X_{\C}$ a complex algebraic manifold
and $F_{\infty}$ an antilinear involution of $X_{\C}$.

If $X=(X_{\C},F_{\infty})$ is a real variety, we will denote by
$\sigma $ the involution of $\mmD^n_{\log}(X_{\C},p)$ given by
\begin{displaymath}
  \sigma (\eta)=\overline {F_{\infty}^{\ast}\eta}.
\end{displaymath}
Then the \emph{real
Deligne-Beilinson cohomology of $X$} is defined by
$$H_{\mmD}^{n}(X,\R(p)) := H_{\mmD}^n(X_{\C},\R(p))^{\sigma }, $$
where the superindex $\sigma $ means the fixed part under $\sigma $.

The real cohomology of $X$ is expressed as the cohomology of the
real Deligne complex
$$\mmD^n_{\log}(X,p):=
\mmD^n_{\log}(X_{\C},p)^{\sigma },$$
i.e. there is an isomorphism
$$H^n_{\mmD}(X,\R(p))\cong H^n(\mmD^n_{\log}(X,p),d_{\mmD}).$$

\subsection{Truncated Deligne complex}
In the rest of the work, we will consider the Deligne complex
(canonically) truncated at degree $2p$. For simplicity we will denote
it by
$$\tau \mmD_{\log}^*(X,p)=\tau_{\leq 2p} \mmD_{\log}^*(X,p).$$
The truncated Deligne complex with supports in a variety $Z$ is
denoted by $\tau \mmD_{\log,Z}^*(X,p)=\tau_{\leq 2p}
\mmD_{\log,Z}^*(X,p)$ and the truncated Deligne
complex with supports in $\mathcal{Z}^p$ is denoted by $\tau
\mmD_{\log,\mathcal{Z}^p}^*(X,p)=\tau_{\leq 2p}
\mmD_{\log,\mathcal{Z}^p}^*(X,p)$.

Note that, since the truncation is not an exact functor, it is not
true that $\tau
\mmD_{\log,\mathcal{Z}^p}^*(X,p)$ is the simple complex of the map $\tau
\mmD_{\log}^*(X,p)\rightarrow \tau
\mmD_{\log}^*(X\setminus \mathcal{Z}^p,p)$.

\section{Differential forms and higher Chow groups}

In this section we construct a complex of differential forms
which is quasi-isomorphic to the complex $Z^p(X,*)_0\otimes \R$. This
last complex
computes the higher
algebraic Chow groups introduced by Bloch in \cite{Bloch1} with real
coefficients. The key point of this construction is the set of
isomorphisms given in \eqref{semipurity}.

This complex is very similar to the complex introduced by Bloch in
\cite{Bloch4} in order to construct the cycle map for the higher
Chow groups. In both constructions one considers a $2$-iterated
complex of differential forms on a cubical or simplicial scheme. Since
this leads to a second quadrant spectral sequence, to avoid
convergence problems, one has to truncate the complexes involved.
The main difference between both constructions is the direction of the
truncation. We
truncate the $2$-iterated complex at the degree given by the
differential forms, while in \emph{loc. cit.} the complex is truncated
at the degree
given by the simplicial scheme.

\subsection{The cubical Bloch complex}\label{higherchow}\label{cubicalp1}
We recall here the definition and main properties of the
\emph{higher Chow groups} defined by Bloch in \cite{Bloch1}.
Initially, they were defined using the chain complex associated to
a simplicial abelian group. However, since we are interested in the
product structure,
it is more convenient to use the
cubical presentation, as given by Levine in \cite{Levine1}.

Fix a base field $k$ and let $\P^1$ be the projective line over
$k$. Let $ \square = \P^1\setminus \{1\}(\cong \A^1).$
The cartesian
product $(\P^1)^{\cdot}$ has a cocubical scheme structure.
For $i=1,\dots,n$, we denote by $t_{i}\in (k\cup\{\infty\})\setminus
\{1\}$  the absolute coordinate of the $i$-th factor. Then the
coface and codegeneracy maps
are defined as
\begin{eqnarray*}
  \delta_0^i(t_1,\dots,t_n) &=& (t_1,\dots,t_{i-1},0,t_{i},\dots,t_n), \\
\delta_1^i(t_1,\dots,t_n) &=& (t_1,\dots,t_{i-1},\infty,t_{i},\dots,t_n), \\
\sigma^i(t_1,\dots,t_n) &=& (t_1,\dots,t_{i-1},t_{i+1},\dots,t_n).
\end{eqnarray*}
Then, $\square^{\cdot}$ inherits a cocubical scheme structure from that of $(\P^1)^{\cdot}$.
An $r$-dimensional \emph{face} of $\square^n$ is any subscheme of the form
$\delta^{i_1}_{j_1}\cdots \delta^{i_r}_{j_r}(\square^{n-r})$.

We have chosen to represent $\A^{1}$ as $\P^{1}\setminus \{1\}$ so
that the face maps are represented by the inclusion at zero and the
inclusion at infinity. In this way the cubical structure of
$\square^{\cdot}$ is compatible with the cubical structure of
$(\P^{1})^{\cdot}$ in \cite{Burgos1}. In the literature it is often
used the usual representation $\A^{1}=\P^{1}\setminus \{\infty\}$. We
will translate from one definition to another by using the involution
\begin{equation}
  \label{eq:5}
  x\longmapsto \frac{x}{x-1}.
\end{equation}
This involution has the fixed points $\{0,2\}$ and interchanges the
points $1$ and $\infty$.

Let $X$ be an equidimensional quasi-projective algebraic scheme of
dimension $d$ over the field $k$. Let $Z^{p}(X,n)$ be the free
abelian group generated by the codimension $p$ closed irreducible
subvarieties of $X\times \square^{n}$, which intersect properly
all the faces of $\square^n$.
The pull-back by the coface and codegeneracy maps of $\square^{\cdot}$
endow $Z^{p}(X,\cdot)$ with a cubical abelian group structure. Let
$(Z^{p}(X,*),\delta)$ be the associated chain complex and
consider the \emph{normalized chain complex} associated to $Z^p(X,*)$,
$$Z^{p}(X,n)_{0}:= NZ^p(X,n)=\bigcap_{i=1}^{n} \ker \delta^1_i.$$

\begin{df} Let $X$ be a quasi-projective equidimensional algebraic scheme  over a field
$k$. The \emph{higher Chow groups} defined by Bloch are
$$CH^p(X,n):=H_n(Z^{p}(X,*)_{0}).$$
\end{df}

Let $N_0$ be the refined normalized
complex of Definition \eqref{normalized2}. Let $Z^p(X,*)_{00}$ be
the complex with
$$Z^p(X,n)_{00} := N_0Z^p(X,n)=\bigcap_{i=1}^{n} \ker \delta^1_i\cap
\bigcap_{i=2}^{n} \ker \delta^0_i.$$

Fix $n\geq 0$. For every $j=1,\dots,n$, we define a map
\begin{eqnarray}\label{mapsh}
 \square^{n+1} & \xrightarrow{h^j} & \square^n  \\
(t_1,\dots,t_{n+1}) &\mapsto & (t_1,\dots,t_{j-1},1-(t_j-1)(t_{j+1}-1),t_{j+2},\dots,t_{n+1}). \nonumber
\end{eqnarray}
The refined normalized complex of \cite{Bloch3} \S 4.4 is given by considering the elements in the kernel of all faces but $\delta^1_1$, instead of 
$\delta^0_1$ like here. Taking this into account, together with the involution \eqref{eq:5}, the map $h^{j}$ agrees with the map denoted $h^{n-j}$ in
\cite{Bloch3} \S 4.4.  Therefore, the maps $h_{j}$ are smooth, hence
flat, so they induce pull-back maps
\begin{equation}
  \label{eq:4}
  h_{j}: Z^p(X,n) \longrightarrow Z^p(X,n+1),\  j=1,\dots,n+1,
\end{equation}
that satisfy the conditions of Proposition
\ref{normalized3}. Therefore the
 inclusion
$$Z^p(X,n)_{00} := N_0 Z^p(X,n)\rightarrow Z^p(X,n)_0 $$
is a homotopy equivalence (see \cite{Bloch3} \S 4.4).

\subsection{Functoriality.}
\label{functchow}
It follows easily from the definition that the complex $Z^{p}(X,*)_0$
is covariant with respect to proper maps (with a shift in the
grading) and contravariant for flat maps.

Let $f:X\rightarrow Y$ be an arbitrary map between two smooth
schemes $X,Y$. Let $Z_{f}^{p}(Y,n)_0\subset Z^{p}(Y,n)_0$ be the
subgroup generated by the  codimension $p$ irreducible
subvarieties $Z\subset Y\times \square^{n}$, intersecting properly
the faces of $\square^n$ and such that the pull-back $X\times Z$
intersects properly the graph of $f$, $\Gamma_f$. Then,
$Z_{f}^{p}(Y,*)_0$ is a chain complex and  the inclusion of
complexes $Z^{p}_f(Y,*)_{0} \subseteq Z^{p}(Y,*)_0$ is a
quasi-isomorphism. Moreover, the pull-back by $f$ is defined for
algebraic cycles in $Z_{f}^{p}(Y,*)_0$ and hence there is a
well-defined pull-back morphism
$$CH^{p}(Y,n) \xrightarrow{f^*} CH^p(X,n). $$
A proof of this fact can be found in \cite{Levine2}, $\S$3.5. See
also \cite{AmalenduLevine}.

\subsection{Product structure.}\label{chowprod} Let $X$ and $Y$ be
quasi-projective algebraic schemes over $k$. Then, there is a
chain morphism
$$s(Z^{p}(X,*)_0\otimes Z^{q}(Y,*)_0)\xrightarrow{\cup} Z^{p+q}(X\times Y,*)_0$$
inducing exterior products
$$CH^p(X,n)\otimes CH^q(Y,m)\xrightarrow{\cup} CH^{p+q}(X\times Y,n+m).$$

More concretely, let $Z$ be a  codimension $p$ irreducible subvariety
of $X\times \square^n$, intersecting properly the faces of
$\square^n$ and let $W$ be a codimension $q$ irreducible
subvariety of $Y\times \square^m$, intersecting properly the faces
of $\square^m$. Then, the codimension $p+q$ subvariety
$$Z \times W \subseteq X\times \square^n\times Y \times \square^m
\cong X\times Y \times \square^n\times \square^m
\cong X \times Y \times \square^{n+m}, $$ intersects properly the
faces of $\square^{n+m}$. By linearity, we obtain a morphism
$$Z^p(X,n)\otimes Z^q(Y,m)  \xrightarrow{\cup} Z^{p+q}(X\times Y,n+m). $$
It induces a chain morphism on the normalized complexes
$$ s(Z^p(X,*)_{0}\otimes Z^q(Y,*)_{0})  \xrightarrow{\cup}
Z^{p+q}(X\times Y,*)_{0}, $$ and hence there is an external product
\begin{equation}
 \cup: CH^{p}(X,n)\otimes CH^q(Y,m)\rightarrow CH^{p+q}(X\times Y,n+m),
\end{equation}
for all $p,q,n,m$.

If $X$ is smooth, then the pull-back by the diagonal map
$\Delta:X\rightarrow X\times X$ is defined on the higher Chow
groups, $CH^{p}(X\times X,*) \xrightarrow{\Delta^*} CH^p(X,*).$
Therefore, for all $p,q,n,m$, we obtain an internal product
\begin{equation}\label{chowprod1}
\cup: CH^p(X,n)\otimes CH^q(X,m)\rightarrow CH^{p+q}(X\times
X,n+m)\xrightarrow{\Delta^*} CH^{p+q}(X,n+m). \end{equation} In
the derived category of chain complexes,  the internal product is
given by the morphism {\small
$$\xymatrix{s(Z^p(X,*)_{0}\otimes Z^q(X,*)_{0} )
\ar[r]^(0.57){\cup} & Z^{p+q}(X\times X,*)_0 \\ &
Z^{p+q}_{\Delta}(X\times X,*)_{0} \ar[u]_{\sim}
 \ar[r]^(0.55){\Delta^*}& Z^{p+q}(X,*)_{0}.}$$}

 \begin{prop} Let $X$ be a quasi-projective algebraic scheme over
   $k$. The pairing  \eqref{chowprod1}
 defines an associative product on
 $CH^*(X,*)=\bigoplus_{p,n}CH^p(X,n)$. This product
 is graded commutative with respect to the degree
 given by $n$.
 \end{prop}
 \begin{proof}
See \cite{Levine1}, Theorem 5.2.
 \end{proof}

\subsection{Differential forms and affine lines}\label{diffformsaffine}
For every $n,p\geq 0$, let $\tau\mathcal{D}_{\log}^*(X\times
\square^n,p)$ be the truncated Deligne complex of differential forms in
$X\times \square^n$, with logarithmic singularities at infinity. The structural maps of the cocubical
scheme $\square^{\cdot}$ induce a cubical structure
on $\tau\mathcal{D}_{\log}^r(X\times \square^*,p)$ for every $r$ and $p$.

Consider the $2$-iterated cochain complex
$$\mathcal{D}_{\A}^{r,-n}(X,p)=\tau\mathcal{D}_{\log}^r(X\times \square^n,p),$$
with differential $(d_{\mmD},\delta=\sum_{i=1}^n
(-1)^i(\delta_i^0-\delta_i^1))$.
Let
$$\mmD_{\A}^*(X,p)=s(\mathcal{D}_{\A}^{*,*}(X,p))$$
be the associated simple complex. Hence its differential $d_s$ in
$\mmD_{\A}^*(X,p)$ is given, for every $\alpha\in
\mathcal{D}_{\A}^{r,-n}(X,p)$, by $d_s(\alpha) = d_{\mmD}(\alpha)
+(-1)^r \delta (\alpha)$.
Since we are using cubical structures, this complex does not compute the
right cohomology and we have to normalize it.

For every $r,n$, we write
 $$\mmD_{\A}^{r,-n}(X,p)_0=\tau\mmD_{\log}^r(X\times \square^n,p)_0:=
N\tau\mmD_{\log}^r(X\times \square^n,p).$$
Hence $\mmD_{\A}^{\ast,\ast}(X,p)_0$ is the normalized $2$-iterated
complex and we denote by $\mmD^*_{\A}(X,p)_0$ the associated simple
complex.

\begin{prop}\label{affine3}
The natural morphism of complexes
$$\tau\mmD^*_{\log}(X,p)= \mmD_{\A}^{*,0}(X,p)_0 \rightarrow
\mmD_{\A}^{*}(X,p)_0$$
is a quasi-isomorphism.
\end{prop}
\begin{proof}
 Consider the second quadrant spectral sequence with $E_1$ term given by
$$E_1^{r,-n}= H^r(\mmD_{\A}^{*,-n}(X,p)_0).$$
Since
\begin{displaymath}
  \mmD_{\A}^{r,-n}(X,p)_0=0,\text{ for }r<0 \text{ or }r>2p,
\end{displaymath}
this spectral sequence converges to the cohomology groups
$H^*(\mmD_{\A}^{*}(X,p)_0)$. This is the main reason why we use
the truncated complexes.

If we see that, for all $n>0$, the cohomology of the complex
$\mmD_{\A}^{*,-n}(X,p)_0$ is zero, the spectral sequence degenerates
and the proposition is proven. By the homotopy invariance of
Deligne-Beilinson cohomology, there is an isomorphism
$$\delta_1^1 \circ \cdots \circ
\delta_1^1:H^*(\tau\mmD^{*}_{\log}(X\times \square^n,p)) \rightarrow
H^*(\tau\mmD^{*}_{\log}(X,p)).$$ By definition, the image of
$H^*(\tau\mmD^{*}_{\log}(X\times \square^n,p)_0)$ by this isomorphism
is zero. Since $H^*(\tau\mmD^{*}_{\log}(X\times \square^n,p)_0)$ is a
direct summand of $H^*(\tau\mmD^{*}_{\log}(X\times \square^n,p))$, it
vanishes for all $n>0$.
\end{proof}

We define the complex $\mmD_{\A}^*(X,p)_{00}$ to be the simple
complex associated to the $2$-iterated complex with
$$\mmD_{\A}^{r,-n}(X,p)_{00}=
N_0\tau\mmD_{\log}^r(X\times \square^n,p).$$

\begin{cor}\label{affine4}
The natural morphism of complexes
$$\tau\mmD^*_{\log}(X,p)= \mmD_{\A}^{*,0}(X,p)_{00} \rightarrow \mmD_{\A}^{*}(X,p)_{00}$$
is a quasi-isomorphism.
\end{cor}
\begin{proof}
It follows from Proposition \ref{affine3}, Proposition
\ref{normalized3} (using as maps $\{h_j\}$ the ones induced by the
maps $h^{j}$ defined
in \ref{mapsh}) and Proposition \ref{cubchain}.
\end{proof}

\subsection{A complex with differential forms for the higher Chow
  groups}\label{diffformschow}
Let $\mathcal{Z}^p_{n,X}$ be the set of all codimension $p$
closed
subvarieties of $X\times \square^n$ intersecting properly the
faces of $\square^n$. We consider it as an ordered set by
inclusion. When there is no source of confusion, we
simply write $\mmZ^p_n$ or even $\mmZ^p$. Consider the cubical
abelian group
\begin{equation}
\mcH^p(X,*):=H^{2p}_{\mmD,\mathcal{Z}^p_*}(X\times
\square^*,\R(p)),
\end{equation}
with faces and degeneracies induced by those of $\square^{\cdot}$. Let
$\mcH^p(X,*)_0$ be the
associated normalized complex.

\begin{lema}\label{blochdif1} Let $X$ be a  complex algebraic
  manifold. For every $p\geq 0$, there is an isomorphism of chain
  complexes
$$
f_1: Z^p(X,*)_{0}\otimes \R  \xrightarrow{\cong}  \mcH^p(X,*)_0,
$$
sending $z$ to $cl(z)$.
\end{lema}
\begin{proof}
It follows from the isomorphism \eqref{semipurity}.
\end{proof}

\begin{obs}
Observe that the complex $\mcH^{p}(X,*)_0$ has the same functorial
properties as $Z^p(X,*)_{0}\otimes \R$.
\end{obs}

Let $\mmD^{*,*}_{\A,\mathcal{Z}^p}(X,p)_0$ be the $2$-iterated
cochain complex, whose component of bidegree $(r,-n)$ is
$$\tau\mmD_{\log,\mathcal{Z}_n^p}^{r}(X\times \square^n,p)_0
=N\tau\mmD_{\log,\mathcal{Z}_n^p}^{r}(X\times \square^n,p)
=N\tau_{\le 2p}\mmD_{\log,\mathcal{Z}_n^p}^{r}(X\times
\square^n,p), $$
and whose differentials are
 $(d_{\mmD},\delta)$. As usual,
we denote by $\mmD^{*}_{\A,\mathcal{Z}^p}(X,p)_0$ the associated
simple complex and by $d_s$ its differential.

Let $\mmD^{2p-*}_{\A,\mathcal{Z}^p}(X,p)_0$ be the chain complex
whose $n$-graded piece is $\mmD^{2p-n}_{\A,\mathcal{Z}^p}(X,p)_0$.

\begin{prop}\label{difcubs}
For every $p\geq 0$, the family of morphisms
\begin{eqnarray*}
\mmD^{2p-n}_{\A,\mathcal{Z}^p}(X,p)_0 & \xrightarrow{g_1}&
\mcH^{p}(X,n)_0 \\
((\omega_n,g_n),\dots,(\omega_0,g_0))  & \mapsto &
[(\omega_n,g_n)]
\end{eqnarray*}
defines a quasi-isomorphism of chain complexes.
\end{prop}
\begin{proof}
Consider the second quadrant spectral sequence with $E_1$-term
$$E_1^{r,-n}= H^r(\tau\mmD^{*}_{\log,\mathcal{Z}^p}(X\times \square^n,p)_0).$$
By construction, $E_1^{r,-n}=0$ for all $r>2p$.
Moreover, for all $r<2p$ and for all $n$, the semipurity property
of Deligne-Beilinson cohomology implies that
\begin{equation}\label{purity1}
H^r(\tau\mmD^*_{\log,\mathcal{Z}^p}(X\times \square^n,p))=0.
\end{equation}
Hence, by Proposition \ref{cubchain},
$$H^r(\tau\mmD^*_{\log,\mathcal{Z}^p}(X\times \square^n,p)_0)=0,\qquad r<2p. $$
Therefore, the $E_1$-term of the spectral sequence is
$$ E_1^{r,-n} = \left\{ \begin{array}{ll}
0 & \textrm{ if }r\neq 2p, \\
H^{2p}(\tau\mmD^*_{\log,\mathcal{Z}^p}(X\times \square^n,p)_0) &
\textrm{ if }r= 2p.
\end{array}\right.$$
Finally, from Proposition \ref{cubchain2}, it follows that the
natural map
$$H^{2p}(\tau\mmD^*_{\log,\mathcal{Z}^p}(X\times
\square^n,p)_0) \rightarrow \mcH^{p}(X,n)_0$$ is an isomorphism,
and the proposition is proved.
\end{proof}

We denote
$$CH^{p}(X,n)_{\R} = CH^{p}(X,n)\otimes {\R}. $$

\begin{cor}\label{blochdif2}
Let $z\in CH^{p}(X,n)_{\R}$ be the class of an algebraic cycle $z$
in $X\times \square^n$. By the isomorphisms of Lemma
\ref{blochdif1} and Proposition \ref{difcubs}, the algebraic cycle
$z$ is represented, in $H^{2p-n}(\mmD_{\A,\mathcal{Z}^p}(X,p)_0)$,
by any cycle $$((\omega _n,g_n), \dots,(\omega _0,g_0))\in
 \mmD_{\A,\mathcal{Z}^p}^{2p-n}(X,p)_0$$ such that
 $$cl(z)=[(\omega _n,g_n)].$$
\end{cor}

\begin{obs}\label{cycleclassbloch} Our construction differs from the construction given by
Bloch, in \cite{Bloch4}, in two points:
\begin{itemize}
\item He considered the $2$-iterated complex of differential forms
on the simplicial scheme $\A^n$, instead of the differential forms
on the cubical scheme $\square^n$. \item In order to ensure the
convergence of the spectral sequence in the proof of last
proposition, he truncated the $2$-iterated complex in the
direction given by the affine schemes.
\end{itemize}
\end{obs}

\subsection{Functoriality of $\mmD_{\A,\mathcal{Z}^p}^*(X,p)_0$}

In many aspects, the complex $\mmD_{\A,\mathcal{Z}^p}^*(X,p)_0$
behaves  like the complex $Z^{*}(X,*)_0$.

\begin{lema}\label{funct1}
Let $f:X\rightarrow Y$ be a flat map between two equidimensional
complex algebraic manifolds. Then there is a pull-back map
$$f^*: \mmD_{\A,\mathcal{Z}^p}^*(Y,p)_0\rightarrow \mmD_{\A,\mathcal{Z}^p}^*(X,p)_0.$$
\end{lema}
\begin{proof}
We will see that in fact there is a map of iterated complexes
$$ f^*: \mmD_{\A,\mathcal{Z}^p}^{r,-n}(Y,p)\rightarrow \mmD_{\A,\mathcal{Z}^p}^{r,-n}(X,p).$$
Let  $Z$ be a codimension $p$ subvariety of $Y\times \square^n$
intersecting properly the faces of $\square^n$. Since $f$ is flat,
there is a well-defined cycle $f^*(Z)$. It is a codimension $p$ cycle of
$X\times \square^n$ intersecting properly the faces of
$\square^n$, and whose support is $f^{-1}(Z)$. Then, by
\cite{GilletSouleIHES} 1.3.3, the pull-back of differential forms
gives a morphism
$$\tau\mmD_{\log}^*(Y\times \square^n\setminus
Z,p)\xrightarrow{f^*}\tau\mmD_{\log}^*(X\times \square^n\setminus
f^{-1}(Z),p).$$
Hence, there is an induced morphism
  $$\tau\mmD_{\log}^*(Y\times \square^n\setminus
  \mmZ^p_{Y},p)\xrightarrow{f^*}\lim_{\substack{\rightarrow
  \\ Z\in \mmZ^p_{Y}}}\tau\mmD_{\log}^*(X\times \square^n\setminus f^{-1}(Z),p)\rightarrow \tau\mmD_{\log}^*(X\times
  \square^n\setminus \mmZ^p_{X},p),$$
and thus, there is a pull-back morphism
$$f^*: \mmD_{\A,\mathcal{Z}^p}^{*,-n}(Y,p)\rightarrow \mmD_{\A,\mathcal{Z}^p}^{*,-n}(X,p) $$
compatible with the differential $\delta$.
\end{proof}
\begin{obs}
The pull-back defined here agrees with the pull-back defined by
Bloch under the isomorphisms of Lemma \ref{blochdif1} and
Proposition \ref{difcubs}. Indeed, let $f:X\rightarrow Y$ be a
flat map. Then, if $Z$ is an irreducible subvariety of $Y$ and
$(\omega,g)$ a couple representing the class of $[Z]$ in the
Deligne-Beilinson cohomology with support, then the couple
$(f^*\omega,f^*g)$ represents the class of $[f^{*}(Z)]$ (see
\cite{GilletSouleIHES}, Theorem 3.6.1).
\end{obs}

\begin{prop}\label{funct2}
Let $f:X\rightarrow Y$ be a morphism of equidimensional complex
algebraic manifolds. Let $\mmZ^p_f$ be the subset consisting of
the subvarieties $Z$ of $Y\times \square^n$ intersecting properly
the faces of $\square^n$ and such that $X\times Z \times
\square^n$ intersects properly the graph of $f$, $\Gamma_f$. Then,
\begin{enumerate}[(i)] \item
The complex $\mmD_{\A,\mathcal{Z}_f^p}^*(Y,p)_0$ is
quasi-isomorphic to $ \mmD_{\A,\mathcal{Z}^p}^*(Y,p)_0$. \item
There is a well-defined pull-back
  $$f^*:\mmD_{\A,\mathcal{Z}_f^p}^*(Y,p)_0\rightarrow
  \mmD_{\A,\mathcal{Z}^p}^*(X,p)_0.
  $$
  \end{enumerate}
\end{prop}

\begin{proof} Arguing as in the proof of the previous proposition,
there is a pull-back map
$$f^*:\tau\mmD_{\log}^*(Y\times \square^n\setminus \mmZ^p_{f},p)
\xrightarrow{f^*} \tau\mmD_{\log}^*(X\times \square^n\setminus
\mmZ^p,p),$$
 inducing a morphism
$$f^*: \mmD_{\A,\mathcal{Z}_f^p}^{*}(Y,p)\rightarrow \mmD_{\A,\mathcal{Z}^p}^{*}(X,p),
$$ and hence a morphism
$$f^*: \mmD_{\A,\mathcal{Z}_f^p}^{*}(Y,p)_0 \rightarrow
\mmD_{\A,\mathcal{Z}^p}^{*}(X,p)_0.
$$
All that remains to be shown is that the inclusion
$$\mmD_{\A,\mathcal{Z}_f^p}^{*}(Y,p)_0 \xrightarrow{i} \mmD_{\A,\mathcal{Z}^p}^{*}(Y,p)_0$$ is a quasi-isomorphism.
By the quasi-isomorphism mentioned in paragraph \ref{functchow}
and the quasi-isomorphism of Proposition \ref{difcubs},
there is a commutative diagram
$$ \xymatrix{
Z_f^p(Y,*)_0 \otimes \R \ar[r]\ar[d]_{\sim} & \mmD_{\A,\mathcal{Z}_f^p}^{*}(Y,p)_0 \ar[d]^{i}  \\
Z^p(Y,*)_0 \otimes \R \ar[r]^{\sim} &
\mmD_{\A,\mathcal{Z}^p}^{*}(Y,p)_0. }$$ The proof that the upper
horizontal arrow is a quasi-isomorphism is analogous to the proof
of Proposition \ref{difcubs}. Thus, we deduce   that $i$ is a
quasi-isomorphism.
\end{proof}

\section{Algebraic cycles and the Beilinson regulator}\label{regulator2}

In this section we define a chain morphism, in the derived
category of chain complexes, that induces in homology the
Beilinson regulator.

The construction is analogous to the definition of the cycle class
map given by Bloch in \cite{Bloch4}, with the minor modifications
mentioned in \ref{cycleclassbloch}. However, in \cite{Bloch4} there is
no proof of the fact that the composition of the isomorphism
$K_n(X)_{\Q}\cong \bigoplus_{p\geq 0}CH^p(X,n)_{\Q}$ with the
cycle class map agrees with the Beilinson regulator.

\subsection{Definition of the regulator}\label{regulator3}
Consider the map of iterated cochain complexes defined by the
projection onto the first factor
\begin{eqnarray*}
\mmD_{\A,\mathcal{Z}^p}^{r,-n}(X,p)=\tau_{\le 2p}s(\mmD_{\log}^{\ast}(X\times
\square^n,p)\to \mathcal{D}_{\log}^{\ast}(X\times
\square^n\setminus \mmZ^p,p))^r & \xrightarrow{\rho} &
\tau\mmD_{\log}^r(X\times \square^n,p)
\\ (\omega ,g) & \mapsto & \omega .
\end{eqnarray*}
It induces a cochain morphism
\begin{eqnarray*}
\mmD_{\A,\mathcal{Z}^p}^{*}(X,p)_0 & \xrightarrow{\rho} &
\mmD_{\A}^*(X,p)_0,
\end{eqnarray*}
and hence a chain morphism
\begin{equation}\label{regulator4}
 \mmD_{\A,\mathcal{Z}^p}^{2p-*}(X,p)_0 \xrightarrow{\rho}
\mmD_{\A}^{2p-*}(X,p)_0.\end{equation} The morphism induced by
$\rho$ in homology, together with the isomorphisms of Propositions
  \ref{affine3}, \ref{blochdif1} and \ref{difcubs}, induce a morphism
\begin{equation}\label{regulator}\rho: CH^p(X,n) \rightarrow CH^p(X,n)_{\R} \rightarrow
H_{\mmD}^{2p-n}(X,\R(p)).\end{equation} By abuse of notation, it will
also be denoted by $\rho$.

By corollary \ref{blochdif2}, we deduce that, if $z\in
Z^p(X,n)_0$, then
$$\rho(z) =  (\omega _n,\dots,\omega _0),$$
for any cycle $((\omega _n,g_n),\dots,(\omega _0,g_0))\in
\mmD_{\A,\mathcal{Z}^p}^{2p-n}(X,p)_0$  such that
$[(\omega _n,g_n)]=cl(z)$.

\begin{prop}
\begin{enumerate}[(i)]
\item The morphism $\rho:\mmD_{\A,\mathcal{Z}^p}^{2p-*}(X,p)_0
\rightarrow \mmD_{\A}^{2p-*}(X,p)_0$ is contravariant for flat
maps. \item The induced morphism $\rho:CH^p(X,n)\rightarrow
H_{\mmD}^{2p-n}(X,\R(p))$ is contravariant for arbitrary maps.
\end{enumerate}
\end{prop}

\begin{proof} Both assertions are obvious. Let $z=((\omega _n,g_n),
  \dots,(\omega _0,g_0))\in
 \mmD_{\A,\mathcal{Z}^p}^{2p-n}(X,p)_0$  be a cycle such that its
 inverse image  by $f$ is defined. This is the case when $f$ is flat
 or when $z$ belongs to $\mmD_{\A,\mathcal{Z}_{f}^p}^{2p-*}(X,p)_0$.
 In both cases
$$f^*((\omega _n,g_n),\dots,(\omega _0,g_0))=((f^*\omega _n,f^*g_n),
\dots,(f^*\omega _0,f^*g_0))$$ and
the claim follows.
\end{proof}

\begin{obs}\label{chowchern}
Let $X$ be an equidimensional compact complex algebraic manifold.
Observe that, by definition, the morphism
$$\rho: CH^p(X,0)=CH^p(X) \rightarrow H_{\mmD}^{2p}(X,\R(p))$$
agrees with the cycle class map $cl$.

Now let $E$ be a vector bundle of rank $n$ over $X$. For every
$p=1,\dots,n$,  there exists a characteristic class
$C_p^{CH}(E)\in CH^p(X)$ (see \cite{GrothChern}) and a
characteristic class $C_p^{\mmD}(E)\in H_{\mmD}^{2p}(X,\R(p))$,
called the $p$-th Chern class of the vector bundle $E$. By
definition, $cl(C_p^{CH}(E))=C_p^{\mmD}(E)$.
Hence,
 $$\rho(C_p^{CH}(E))=C_p^{\mmD}(E),  $$
 for all $p=1,\dots,n$.
\end{obs}

\subsection{Comparison with the Beilinson regulator}
We prove here that the regulator defined in
\eqref{regulator} agrees with the Beilinson regulator.

The comparison is based on the following facts:
\begin{itemize}
\item The morphism $\rho$ is compatible with inverse images.
\item The morphism $\rho$ is defined for quasi-projective schemes.
\end{itemize}
In view of these properties, it is enough to prove that the two
regulators agree when $X$ is a Grassmanian manifold, which in turn
follows from Remark \ref{chowchern}.

\begin{theo}\label{beichow} Let $X$ be an equidimensional complex algebraic scheme.
Let $\rho'$ be the composition of $\rho$ with the isomorphism
given by the Chern character
$$\rho': K_n(X)_{\Q} \xrightarrow{\cong} \bigoplus_{p\geq 0} CH^{p}(X,n)_{\Q} \xrightarrow{\rho}
 \bigoplus_{p\geq 0}H^{2p-n}_{\mmD}(X,\R(p)).$$
Then, the morphism $\rho'$ agrees with the Beilinson regulator.
\end{theo}
\begin{proof}
The outline of the proof is as follows. We first recall the
description of the Beilinson regulator in terms of homotopy theory
of simplicial sheaves as in \cite{GilletSouleFiltrations}. Then,
we recall the construction of the Chern character given by Bloch.
We proceed reducing the comparison of the two maps to the case
$n=0$ and for $X$ a Grassmanian scheme. We finally prove that at
this stage both maps agree. Our site will always be the small
Zariski site over $X$.

Consider $X$ as a smooth quasi-projective scheme over $\C$. Let
$B_{\cdot}GL_{N}$ be the simplicial version of the classifying space
of the group $GL_{N}(\mathbb{C})$ viewed as a simplicial complex
manifold.
Recall that all the face morphisms are flat. Let
$B_{{\cdot}}GL_{N,X}$ be the simplicial sheaf over $X$ given by
the sheafification of the presheaf
$$U \mapsto B_{\cdot}GL_N(\Gamma(U,\mathcal{O}_U))$$
for every Zariski open $U\subseteq X$. This is the same as the
simplicial sheaf given by
$$U\mapsto \underline{\Hom} (U,B_{\cdot}GL_{N}),$$
where $\underline{\Hom}$ means the simplicial function complex.

Consider the inclusion morphisms $BGL_{N,X}\rightarrow
BGL_{N+1,X}$, for all $N\geq 1$, and let
$$B_{\cdot}GL_X=\lim_{\rightarrow} B_{\cdot}GL_{N,X}.$$
Let $\Z_{\infty}BGL_{N,X}$ and $\Z_{\infty}BGL_X$ be the sheaves
associated to the respective Bousfield-Kan completions. Finally,
let $\Z$ be the constant simplicial sheaf on $\Z$ and consider the
following sheaves on $X$
\begin{eqnarray*}
\K_X &=& \Z \times \Z_{\infty}B_{\cdot}GL_X, \\
\K^N_X &=& \Z \times \Z_{\infty}B_{\cdot}GL_{N,X}. \end{eqnarray*}
By \cite{GilletSouleFiltrations}, Proposition 5, there is a
natural isomorphism
$$K_m(X)\cong H^{-m}(X,\K_X)=\lim_{\substack{\rightarrow\\ N} }H^{-m}(X,\K^N_X). $$
Here $H^{-*}(\cdot,*)$ denotes the generalized cohomology with
coefficients in $\K_X$ and $\K^{N}_X$, as described in
\cite{GilletSouleFiltrations}.

\vspace{0.3cm}

\emph{The Beilinson regulator} is the Chern character taking values in
Deligne-Beilinson cohomology. The regulator can be described   in
terms of homotopy theory of sheaves as follows.

Consider the Dold-Puppe functor $\mathcal{K}_{\cdot}(\cdot)$ (see
\cite{DoldPuppe}), which associates to every cochain complex of
abelian groups concentrated in non-positive degrees, $G^*$, a
simplicial abelian group $\mmK_{\cdot}(G)$, pointed by zero. It
satisfies the property that
$\pi_i(\mmK_{\cdot}(G),0)=H^{-i}(G^*)$.

In \cite{Gillet}, Gillet constructs Chern classes
$$C_p^{\mmD}\in H^{2p}(B_{\cdot}GL_N,\R(p)), \qquad N \gg 0,$$
which induce morphisms
$$c_{p,X}^{\mmD}:
\K_{X,\cdot}^N \rightarrow
\mmK_{\cdot}(\mmD^*_X({\cdot},p)[2p]),\qquad N\gg 0.
$$
These morphisms are compatible with the morphisms
$\K^N_{X,{\cdot}}\rightarrow \K^{N+1}_{X,{\cdot}}$. Therefore, we
obtain a morphism
$$ K_m(X)=\lim_{\substack{\rightarrow \\ N} }
H^{-m}(X,\K^N_X)\xrightarrow{C_{p,X}^{\mmD}}
H_{\mmD}^{2p-m}(X,\R(p)).$$
Using the standard formula for the Chern character in terms of the
Chern classes, we obtain a morphism
$$ K_m(X) \xrightarrow{ch^{\mmD}} H_{\mmD}^{2p-m}(X,\R(p)),$$
which is the Beilinson regulator.

\par{\emph{The Chern character for higher Chow
groups.}} The description of the isomorphism $K_n(X)_{\Q}
\xrightarrow{\cong} \bigoplus_{p\geq 0} CH^{p}(X,n)_{\Q}$ given by
Bloch follows the same pattern as the description of the Beilinson
regulator. However, since the complexes that define the higher
Chow groups are not sheaves on the big Zariski site, a few
modifications are necessary. We give here a sketch of the
construction. For details see \cite{Bloch1}.

If $Y_{\cdot}$ is a simplicial scheme whose face maps are flat,
then there is a well-defined $2$-iterated cochain complex
$Z^p(Y_{\cdot},*)_0,$ whose $(n,m)$-bigraded group is
$$Z^p(Y_{-n},m)_0,$$ and induced differentials.
The higher algebraic Chow groups of $Y_{\cdot}$ are then defined
as
$$CH^p(Y_{\cdot},n)=H^n(Z^p(Y_{\cdot},*)_0).$$

Since the face maps of the simplicial scheme $B_{\cdot}GL_N$ are
flat, the group $CH^p(B_{\cdot}GL_N,n)$ is well defined for every
$p$ and $n$.

First, Bloch constructs universal Chern classes
$$C_p^{CH}\in CH^p(B_{\cdot}GL_{N},0),$$ following the ideas of
Gillet. These classes are represented by elements
$$C_p^{CH,i}\in Z^{p}(B_{i}GL_{N},i)_0.$$

Because at the level of complexes the pull-back morphism is not
defined for arbitrary maps, one cannot consider the pull-back of
these classes $C_p^{CH,i}$ to $X$, as was the case for the
Beilinson regulator. However, by \cite{Bloch1} $\S$7, there
exists a purely transcendental extension $L$ of $\C$, and classes
$C_p^{CH,i}$  defined over $L$, such that the pull-back
$f^*C_p^{CH,i}$ is defined for every $\C$-morphism $f:V\rightarrow
B_iGL_N$.

Then, there is a map of simplicial Zariski sheaves on $X$
$$B_{\cdot}GL_{N,X} \rightarrow \mmK_X(g_*Z^p_{X_L}(-,*)_0), $$
where $g:X_L\rightarrow X$ is the natural map obtained by
extension to $L$.

There is a specialization process described in \cite{Bloch1},
which, in the homotopy category of sheaves over $X$, gives a
well-defined map
$$\mmK_X(g_*Z^p_{X_L}(-,*)_0)\rightarrow \mmK_X(Z^p_{X}(-,*)_0).$$

Therefore, there are maps $C_{p,X}^{CH}\in
[B_{\cdot}GL_{N,X},\mmK_X(Z^*_{X}(\cdot,p))]$. Proceeding as
above,
 we obtain the Chern character morphism
$$K_m(X) \rightarrow \bigoplus_{p\geq 0} CH^p(X,m)_{\Q}.$$
For $m=0$, this is the usual Chern character.

\par {\emph{End of the proof.}} Since, at the level of complexes,
$\rho$ is functorial for flat maps, there is a sheaf map
$$\rho: \mmK_X(Z^*_{X}(\cdot,p)) \rightarrow \mmK_{\cdot}(\mmD_{\log}(X,p))$$
in the small Zariski site of $X$.

It follows that the composition $\rho \circ C_p^{CH}$ is obtained
by the same procedure as the Beilinson regulator, but
starting with the characteristic classes $\rho (C_p^{CH})\in
H^{2p}_{\mmD}(X,\R(p))$ instead of the classes $C_p^{\mmD}$. Therefore, it
remains to see that
\begin{equation}
\rho (C_p^{CH})= C_p^{\mmD}.
\end{equation}

For integers $N,k\geq 0$ let $Gr(N,k)$ be the complex Grassmanian scheme of $k$-planes in
$\C^N$. It is a smooth complex projective scheme.
Let $E_{N,k}$ be
the rank $N$ universal bundle of $Gr(N,k)$ and
$U_k=(U_{k,\alpha})_{\alpha}$ its standard trivialization. Let
$N_{\cdot}U_k$ denote the nerve of this cover. It is a hypercover
of $Gr(N,k)$, $N_{\cdot}U_k\xrightarrow{\pi}Gr(N,k)$. Consider the
classifying map of the vector bundle $E_{N,k}$, $\varphi_k:
N_{\cdot}U_k\rightarrow B_{\cdot}GL_N$, which satisfies
$\pi^*(E_{N,k})=\varphi_k^*(E^N_{\cdot})$, for $E^N_{{\cdot}}$ the universal vector bundle over $B_{\cdot}GL_{N}$. Observe that all the
faces and degeneracy maps of the simplicial scheme $N_{\cdot}U_k$
are flat, as well as the inclusion maps $N_lU_k\rightarrow
Gr(N,k)$. Therefore, $CH^p(N_{\cdot}U_k,m)$ is defined and there
is a pull-back map $CH^p(Gr(N,k),m)\xrightarrow{\pi^*}
CH^p(N_{\cdot}U_k,m)$.

Since $\rho$ is defined on $N_{\cdot}U_k$ and is a functorial map,
we obtain the following commutative diagram
$$
\xymatrix{ & CH^p(B_{\cdot}GL_N,0) \ar[r]^{\rho}
\ar[d]_{\varphi_k^*}& H^{2p}_{\mmD}(B_{\cdot}GL_N,\R(p))
\ar[d]^{\varphi_k^*} \\ &
CH^p(N_{\cdot}U_k,0) \ar[r]^{\rho} & H^{2p}_{\mmD}(N_{\cdot}U_k,\R(p)) \\
K_0(Gr(N,k)) \ar[r]^{C_p^{CH}} \ar@/_1.8pc/[rr]_{C_p^{\mmD}} &
CH^p(Gr(N,k),0) \ar[r]^{\rho} \ar[u]^{\pi^*} &
H^{2p}_{\mmD}(Gr(N,k),\R(p)) \ar[u]_{\pi^*} }
$$
By construction, $C_p^{CH}(E_{N,k})$ is the standard $p$-th Chern
class in the classical Chow group of $Gr(N,k)$, and
$C_p^{\mmD}(E_{N,k})$ is the $p$-th Chern class in
Deligne-Beilinson cohomology. It then follows from Remark
\ref{chowchern} that
\begin{equation}\label{chow3}
\rho(C_p^{CH}(E_{N,k}))=C_p^{\mmD}(E_{N,k}).
\end{equation}

The vector bundle $E_{N,k}\in
K_0(Gr(N,k))=\lim_{\substack{\rightarrow
\\ M}}[Gr(N,k),\K^M_{\cdot}]$ is represented by the map of sheaves on $Gr(N,k)$ induced by
$$Gr(N,k) \xleftarrow{\pi} N_{\cdot}U_k \xrightarrow{\varphi_k} B_{\cdot}GL_N.$$
Here, since $N_{\cdot}U_k$ is a hypercover of $Gr(N,k)$, the map
$\pi$ is a weak equivalence of sheaves. This means that
\begin{equation}\label{chow1}
\varphi^*_k(C_p^{CH}(E^N_{\cdot})) = \pi^*(C_p^{CH}(E_{N,k})).
\end{equation}

For each $m_0$, there exists $k_0$ such that if $m\leq m_0$ and
$k\geq k_0$, $\varphi_k^*$ is an isomorphism on the $m$-th
cohomology group. Moreover, $\pi$ also induces  an isomorphism in
Deligne-Belinson cohomology. Under these isomorphisms, we obtain
the equality
\begin{equation}\label{chow2}
C_p^{\mmD}(E_{N,k})=
(\pi^*)^{-1}\varphi^*_k(C_p^{\mmD}(E^N_{\cdot})).
\end{equation}
Hence,
\begin{eqnarray*}\rho (C_p^{CH}(E^N_{\cdot}))= C_p^{\mmD}(E^N_{\cdot})
&\Leftrightarrow &  \varphi_k^* \rho (C_p^{CH}(E^N_{\cdot}))=
\varphi^*_k C_p^{\mmD}(E^N_{\cdot}) \\ & \Leftrightarrow &\rho
\varphi^*_k(C_p^{CH}(E^N_{\cdot}))= \varphi^*_k
C_p^{\mmD}(E^N_{\cdot}).
\end{eqnarray*} The last equality follows directly from
\eqref{chow3}, \eqref{chow1} and \eqref{chow2}. Therefore, the
theorem is proved.
\end{proof}
\section{Higher arithmetic Chow groups}
Let $X$ be an arithmetic variety over a field. Using the
description of the Beilinson regulator given in section
\ref{regulator2}, we define the \emph{higher arithmetic Chow
groups}, $\widehat{CH}^n(X,p)$. The definition is analogous to the
definition given by Goncharov, in \cite{Goncharov}, but using
differential forms instead of currents.

We need to restrict ourselves to arithmetic varieties
over a field, because the theory of higher algebraic Chow groups
by Bloch is only well established for schemes over a field. That is, we
can define the higher arithmetic Chow groups for arbitrary
arithmetic varieties, but since the functoriality properties and
the product structure of the higher algebraic Chow groups are
described only for schemes over a field, we cannot give a product
structure or define functoriality for the higher arithmetic Chow
groups of arithmetic varieties over a ring. Note however that, using
work by Levine \cite{Levine3}, it should be possible to extend the
constructions here to smooth varieties over a Dedekind domain, at least after
tensoring with $\mathbb{Q}$. In fact, when extending the definition to
arithmetic varieties over a ring, it might be better to use the point
of view of
motivic homology \`a la Voevodsky  or any of its more recent variants.

\subsection{Higher arithmetic Chow groups}
Following \cite{GilletSouleIHES}, an arithmetic field is a triple
$(K,\Sigma ,F_{\infty})$, where $K$ is a field, $\Sigma $ is a
nonempty set of complex immersions $K\hookrightarrow \mathbb{C}$ and
$F_{\infty}$ is a conjugate-linear $\mathbb{C}$-algebra automorphism
of $\mathbb{C}^{\Sigma }$ that leaves invariant the image of $K$ under
the diagonal immersion.
By an \emph{arithmetic variety} $X$ over the arithmetic ring $K$ we mean a
regular quasi-projective $K$-scheme $X$.

To the arithmetic variety $X$ we associate a complex variety
$X_{\mathbb{C}}=\coprod_{\iota \in \Sigma }X_{\iota }$, and a real
variety $X_{\R}=(X_{\C},F_{\infty})$.
The Deligne complex of differential forms on $X$ is defined from the
real variety $X_{\R}$
as
$$\mmD^n_{\log}(X,p):=
\mmD^n_{\log}(X_{\C},p)^{\sigma=id},$$
where $\sigma $ is the involution as in paragraph \ref{realvarieties}.
We define
analogously the chain complexes
$$\mmD_{\A}^{2p-*}(X,p)_0,\quad \mmD^{2p-*}_{\A}(X,p)_{00},\quad
\mmD^{2p-*}_{\A,\mathcal{Z}^p}(X,p)_0,\quad
\textrm{and}\quad \mmD^{2p-*}_{\A,\mathcal{Z}^p}(X,p)_{00}.$$

Let $f_1$ be the composition
$$f_1:Z^p(X,n)_0\xrightarrow{\otimes \R} Z^{p}(X,n)_{0}\otimes \R \xrightarrow{\times_F \R }
Z^{p}(X_{\R},n)_0\otimes \R \cong  \mcH^p(X,n)_0.$$
We consider the diagram of complexes of the type of \eqref{ourdiagrambei3}
{\small
  \begin{equation}
    \label{chowcomplex1}
\widehat{\mmZ}^p(X,*)_0=\left(\begin{array}{c}\xymatrix@C=0.5pt{
& \mcH^p(X,*)_0 & & \mmD_{\A}^{2p-*}(X,p)_0 \\
Z^{p}(X,*)_{0} \ar[ur]^{f_1} & & \ar[ul]_{g_1}^{\sim}
\mmD^{2p-*}_{\A,\mathcal{Z}^p}(X,p)_0 \ar[ur]^{\rho} &&  Z\mmD^{2p}(X,p)_* \ar[ul]_{i} }
\end{array}\right)
  \end{equation}
}
where $Z\mmD^{2p}(X,p)_*$ is the chain complex which is zero in all
degrees except in degree zero, where it consists of the vector subspace
of cycles in $\mmD^{2p}(X,p)$. Note that it agrees with
$ZE^{p,p}_{\R}(X)(p)$, the subspace of $E^{p,p}_{\R}(X)(p)$ consisting
of differential forms that are real up to a product by $(2\pi i)^{p}$,
of type $(p,p)$ and that
vanish under $\partial$ and $\overline \partial$. The morphism $i$ is the
inclusion of chain complexes.

\begin{df}
The \emph{higher arithmetic Chow complex} is the simple complex associated to
the diagram $\widehat{\mmZ}^p(X,*)_0$, as defined in
\eqref{diagrambei7}:
$$\widehat{Z}^p(X,*)_0:=s(\widehat{\mmZ}^p(X,*)_0).$$
\end{df}

Recall that, by definition, $\widehat{Z}^p(X,n)_0$ consists of
5-tuples
$$(Z,\alpha_0,\alpha_1,\alpha_2,\alpha_3)\in Z^{p}(X,n)_{0}\oplus
\mmD^{2p-n}_{\A,\mathcal{Z}^p}(X,p)_0\oplus  Z\mmD^{2p}(X,p)_n \oplus
\mcH^p(X,n+1)_0\oplus \mmD_{\A}^{2p-n-1}(X,p)_0,$$ and
the differential is given by
\begin{eqnarray*}
\widehat{Z}^p(X,n)_0 & \xrightarrow{d} & \widehat{Z}^p(X,n-1)_0 \\
(Z,\alpha_0,\alpha_1,\alpha_2,\alpha_3) &\mapsto &
(\delta(Z),d_s(\alpha_0),0,
f_1(Z)-g_1(\alpha_0)-\delta(\alpha_2),\rho(\alpha_0)-\alpha_1-d_s(\alpha_3)).
\end{eqnarray*}
Note that $\alpha_1$ will be zero unless $n=0$. Its differential, however, is always zero.

\begin{df}\label{definitionchow} Let $X$ be an arithmetic variety over
  an arithmetic field.
The $(p,n)$-th \emph{higher arithmetic Chow group} of $X$ is
defined by
$$\widehat{CH}^{p}(X,n) := H_{n}(\widehat{Z}^p(X,*)_0),\qquad p,n\geq 0. $$
\end{df}

By its definition as the cohomology of a simple of a diagram of
complexes it comes equipped with the following morphisms
\begin{alignat*}{2}
\zeta&:\widehat{CH}^{p}(X,n)\longrightarrow CH^{p}(X,n),&&\qquad
\zeta[(Z,\alpha _{0},\dots,\alpha _{3})]=[Z], \\
\amap&:H_{\mmD}^{2p-n}(X,\R(p))\longrightarrow \widehat{CH}^{p}
(X,n),&&\qquad\amap([a])=[(0,0,0,0,-a)], \\
\amap&:\mmD^{2p-1}_{\log}(X,p)\longrightarrow\widehat{CH}^{p}
(X,0),&&\qquad\amap(\widetilde{a})=[(0,0,-d_{\mmD}a,0,-a)], \\
\omega&:\widehat{CH}^{p}(X,0)\longrightarrow{\rm
  Z}\mmD^{2p}_{\log}(X,p),
&&\qquad\omega([(Z,\alpha _{0},\dots,\alpha _{3})])=\alpha _{1}.
\end{alignat*}

\begin{prop}\label{zeta}
There is a long exact sequence
\begin{multline}
  \label{chowarithsequence}
  \cdots \rightarrow \widehat{CH}^{p}(X,n) \xrightarrow{\zeta}
  CH^{p}(X,n) \xrightarrow{\rho}  H_{\mmD}^{2p-n}(X,\R(p))
 \xrightarrow{\amap} \widehat{CH}^{p}(X,n-1)  \rightarrow \cdots\\
\rightarrow CH^p(X,1)  \xrightarrow{\rho}
\mmD^{2p-1}_{\log}(X,p) / \im d_{\mmD} \xrightarrow{\amap}
\widehat{CH}^p(X,0) \xrightarrow{\zeta} CH^p(X,0)  \rightarrow
0,
\end{multline}
where $\rho$ is the Beilinson regulator.
\end{prop}
\begin{proof}
It follows from Theorem \ref{beichow}, Lemma \ref{sucllargachow2} and
the fact that the homology groups of the complex $$s(Z\mmD^{2p}(X,p)_* \xrightarrow{i}  \mmD_{\A}^{2p-*}(X,p)_0)$$ are
$H_{\mmD}^{2p-n}(X,\R(p))$ in degree $n\neq 0$ and $ \mmD^{2p-1}_{\log}(X,p) / \im d_{\mmD}$ in degree $0$.
\end{proof}

\begin{obs}
Let $\widehat{\mmD}_{\A}^{*,*}(X,p)_0$ be the $2$-iterated cochain complex given by the quotient
$\mmD_{\A}^{*,*}(X,p)_0/\mmD^{2p,0}(X,p)$. That is, for all $r,n$, $$\widehat{\mmD}_{\A}^{r,-n}(X,p)_0=\left\{\begin{array}{ll}
0 & \textrm{if } r=2p\textrm{ and } n=0, \\
\mmD_{\A}^{r,-n}(X,p)_0 & \textrm{otherwise}.
\end{array}\right.
$$
Let $\widehat{\mmD}_{\A}^{*}(X,p)_0$  denote the simple complex
associated to $\widehat{\mmD}_{\A}^{*,*}(X,p)_0$. Consider the composition of $\rho$ with the projection map
$$\rho: \mmD^{2p-*}_{\A,\mathcal{Z}^p}(X,p)_0\xrightarrow{\rho} \mmD_{\A}^{2p-*}(X,p)_0 \rightarrow \widehat{\mmD}_{\A}^{2p-*}(X,p)_0.$$
Then, there is a diagram of chain complexes of the type of \eqref{ourdiagrambei2}
  {\small
    \begin{equation}
      \label{chowcomplex2}
    \left(\begin{array}{c}\xymatrix@C=0.5pt{
& \mcH^p(X,*)_0 & & \widehat{\mmD}_{\A}^{2p-*}(X,p)_0 \\
Z^{p}(X,*)_{0} \ar[ur]^{f_1} & & \ar[ul]_{g_1}^{\sim}
\mmD^{2p-*}_{\A,\mathcal{Z}^p}(X,p)_0 \ar[ur]^{\rho} }
\end{array}\right).
    \end{equation}
}
By Proposition \ref{quasiisodiagrams}, the simple complex associated
to the diagram (\ref{chowcomplex2})  is quasi-isomorphic
to the complex $\widehat{Z}^p(X,*)_0$ and hence, its homology groups
are isomorphic $\widehat{CH}^{p}(X,*)$.
Nevertheless, in order to define a product structure in
$\widehat{CH}^{*}(X,*)$ it is better to work with the diagram
(\ref{chowcomplex1}).
\end{obs}

\subsection{Agreement with the arithmetic Chow groups} Let $X$ be an
arithmetic variety and let $\widehat{CH}^p(X)$ denote the $p$-th arithmetic Chow group of $X$ as defined by Burgos in \cite{Burgos2}. We recall here its definition.

For every $p$, let $Z^p(X)=Z^p(X,0)$ and let $
Z\mmD^{2p}_{\log}(X,p)$ denote the subgroup of cycles of
$\mmD^{2p}_{\log}(X,p)$.
Let {\small
$$ \widehat{Z}^p(X)= \left\{ (Z,(\omega,\tilde{g})) \in
Z^p(X)\oplus Z\mmD^{2p}_{\log}(X,p) \oplus
\frac{\mmD^{2p-1}_{\log}(X\setminus \mathcal{Z}^p,p)}{\im d_{\mmD}}\left|
\begin{array}{c}
\omega = d_{\mmD}\tilde{g},\\
cl(Z)=[(\omega,g)]
\end{array}\right.\right\}.
$$ }
If $Z\in Z^p(X)$, a \emph{Green form} for $Z$ is a couple
$(\omega,\tilde{g})$ as before such that
$cl(Z)=[(\omega,g)],$ where $g$ is any representative of $\tilde{g}$.

Let $Y$ be a codimension $p-1$ subvariety of $X$ and let $f\in
k^*(Y)$. As shown in \cite{Burgos2}, $\S$7, there is a canonical
Green form attached to $\Div f$. It is denoted by
$\mathfrak{g}(f)$ and it is of the form $(0,\widetilde{g}(f))$ for
some class $\widetilde{g}(f)$.

One defines the following subgroup of $\widehat{Z}^p(X)$:
$$\widehat{\mathrm{Rat}}^p(X)=\{ (\Div f,\mathfrak{g}(f))|\ f\in k^*(Y),\ Y\subset X \textrm{ a codimension
}p-1 \textrm{ subvariety}\}.$$

For every $p\geq 0$, the \emph{arithmetic Chow group} of $X$ is defined by
 $$\widehat{CH}^p(X) = \widehat{Z}^p(X) / \widehat{\mathrm{Rat}}^p(X). $$

It is proved in \cite{GilletSouleIHES}, Theorem 3.3.5 and
\cite{Burgos2}, Theorem 7.3,  that these groups fit into exact
sequences
$$  CH^{p-1,p}(X)  \xrightarrow{\rho}
\mmD^{2p-1}_{\log}(X,p) / \im d_{\mmD} \xrightarrow{\amap}
\widehat{CH}^p(X) \xrightarrow{\zeta}  CH^p(X) \rightarrow 0 $$
where:
\begin{itemize}
\item $CH^{p-1,p}(X)$ is the
term $E_2^{p-1,-p}$ in the Quillen spectral sequence (see
\cite{Quillen}, $\S$7). \item The map $\rho$ is the Beilinson regulator.
\item The map $\zeta$ is the projection on the first component.
\item The map $\amap$ sends $\alpha$ to $(0,(-d_{\mmD}\alpha,-\alpha))$.
\end{itemize}

\begin{theo}\label{chowagreement}
The morphism
\begin{eqnarray*}
\widehat{CH}^p(X)  & \xrightarrow{\Phi} & \widehat{CH}^p(X,0)  \\
\ [(Z,(\omega,\tilde{g}))] & \mapsto & [(Z,(\omega,g),0,0,0)],
\end{eqnarray*}
where $g$ is any representative of $\tilde{g}\in
\mmD^{2p-1}_{\log}(X,p) / \im d_{\mmD}$, is an isomorphism.
\end{theo}

\begin{proof} We first   prove that $\Phi$ is well defined. Afterwards, we will prove that
the diagram {\small $$ \xymatrix{
 CH^{p-1,p}(X)  \ar[r]^(0.43){\rho} \ar[d]^{\cong} &
\mmD^{2p-1}_{\log}(X,p) / \im d_{\mmD} \ar[r]^(0.6){a} \ar[d]^{=}
& \widehat{CH}^p(X)
\ar[r]^{\zeta} \ar[d]^{\Phi} & CH^p(X)  \ar[r] \ar[d]^{\cong} &  0 \\
 CH^p(X,1)  \ar[r]^(0.43){\rho} &
\mmD^{2p-1}_{\log}(X,p) / \im d_{\mmD} \ar[r]^(0.6){a} &
\widehat{CH}^p(X,0) \ar[r]^{\zeta} &  CH^p(X,0)  \ar[r] &  0 } $$
} is commutative. The statement then follows from the five lemma.

The proof is a consequence of Lemmas \ref{agreement1},
\ref{agreement2} and \ref{agreement3} below.

\begin{lema}\label{agreement1}
The map $\Phi$ is well defined.
\end{lema}
\begin{proof}
We have to prove that:
\begin{enumerate}[(i)]
\item The elements in the image of $\Phi$  are indeed cycles in
$\whZ^p(X,0)_0$.
 \item The map $\Phi$ does not depend on the choice of a representative of
$g$. \item The map $\Phi$ is zero on
$\widehat{\mathrm{Rat}}^p(X)$.
\end{enumerate}
Let $[(Z,(\omega,\tilde{g}))]\in \widehat{CH}^p(X)$. The claim
$(i)$ follows from the equality
$cl(Z)=[(\omega,\tilde{g})]=[(\omega,g)]$. Indeed, since
$d_s(\omega,g)=0$,
$$d(Z,(\omega,g),0,0,0)=(0,0,0,cl(Z)-cl(\omega,g),0)=0.$$

To see $(ii)$, assume that $g_1,g_2\in \mmD_{\log}^{2p-1}(X,p)$
are representatives of $\tilde{g}$, i.e. there exists $h\in
\mmD_{\log}^{2p-2}(X,p)$ such that $ d_{\mmD}h=g_1-g_2.$ Then
$$d(0,(0,h),0,0,0)=(0,(0,g_1-g_2),0,0,0)=(Z,(\omega,g_1),0,0,0)-(Z,(\omega,g_2),0,0,0) $$
and therefore we have $[(Z,(\omega,g_1),0,0,0)]=[(Z,(\omega,g_2),0,0,0)]. $

Finally, to prove $(iii)$, we have to see that, if $Y$ is a
codimension $p-1$ subvariety and $f\in k^{\ast}(Y)$, then
$$\Phi(\Div f,\mathfrak{g}(f))=0\in \widehat{CH}^p(X,0),$$
i.e.  that
$$ [( \Div f,(0,g(f)),0,0,0)]=0,$$ for any fixed representative $g(f)$
of $\widetilde{g}(f)$.

Let $\hat{f}$ be the function of $Y\times \square^1$ given by
$(y,(t_1:t_2)) \mapsto \frac{t_1-t_2f(y)}{t_1-t_2}. $
Its divisor defines a codimension $p$ subvariety of $X\times
\square^1$. Moreover, it intersects properly $X\times (0:1)$ and
$X\times (1:0)$. Fix
$g(\hat{f})$ to be any representative of
$\widetilde{g}(\hat{f})$. Since $
\delta(\widetilde{g}(\hat{f}))=\widetilde{g}(f),$ there exists
$h\in \mmD_{\log}^{2p-1}(X\setminus \Div f,p)$
with
$ d_{\mmD}h= \delta(g(\hat{f})) - g(f). $
Then,
\begin{eqnarray*}
d(\Div \hat{f}, (0,g(\hat{f}),(0,h)),0,0,0) &=& (\Div
f,(0,g(f)),0,0,0)
\end{eqnarray*}
as desired.
\end{proof}

\begin{lema}\label{agreement2}
There are isomorphisms
\begin{eqnarray*}
CH^p(X) & \xrightarrow{\varphi_1} & CH^p(X,0), \\ CH^{p-1,p}(X) &
\xrightarrow{\varphi_2} & CH^p(X,1),
\end{eqnarray*}
making the following diagrams commutative
$$ \xymatrix{
 CH^{p-1,p}(X)  \ar[r]^(0.43){\rho} \ar[d]_{\varphi_2} &
\mmD^{2p-1}_{\log}(X,p) / \im d_{\mmD} \ar[d]^{=}\\
 CH^p(X,1)  \ar[r]^(0.43){\rho} &
\mmD^{2p-1}_{\log}(X,p) / \im d_{\mmD} } \qquad  \xymatrix{
 \widehat{CH}^p(X)
\ar[r]^{\zeta} \ar[d]_{\Phi} & CH^p(X)   \ar[d]^{\varphi_1} \\
 \widehat{CH}^p(X,0)
\ar[r]^{\zeta} &  CH^p(X,0).  }
 $$
\end{lema}
\begin{proof} Both isomorphisms are well known. The morphism $\varphi_1$ is the isomorphism between
the classical Chow group $CH^p(X)$ and the Bloch Chow group
$CH^p(X,0)$. The diagram is obviously commutative, since
$\varphi_1([Z])=[Z]$.

The isomorphism $\varphi_2$ is defined as follows.  Let $f\in
CH^{p-1,p}(X)$. It can be represented by a linear combination
$\sum_{i} [f_i]$, where $f_i\in k^*(W_i)$, $W_i$ is a codimension p-1
subvariety of $X$ and $\sum \Div f_i=0$.  Let $\Gamma_{f_i}$ be the
restriction of the graph of $f_i$ in $\subset X\times \P^1$, to
$X\times \square^1$. That is, $\Gamma_{f_i}$ is the codimension $p$
subvariety of $X\times \square^1$ given by
$$\{(y,f_i(y)) | \ y\in W_i,\ f_i(y)\neq 1\}.$$
Then $\varphi_{2}(f)$ is represented by the image in
\begin{displaymath}
  Z^{p}(X,1)/DZ^{p}(X,1)\cong Z^{p}(X,1)_{0}
\end{displaymath}
of $\sum \Gamma _{f_{i}}$, where $DZ^{p}(X,1)$ are the degenerate
elements.

We want to see that $\rho\varphi_2 = \rho$, i.e., $\rho(\sum_i
\Gamma_{f_i} ) = \rho(\sum [f_i])$. See \cite{Burgos2} or
\cite{BurgosKuhnKramer} for more details on the definition of $\rho$ on the
right hand side.

Let $f=\sum_{i} [f_i] \in CH^{p-1,p}(X)$ be as above. For every $i$, we can choose:
\begin{itemize}
\item a rational function $\tilde{f}_i\in k^*(X)$ whose restriction to $W_i$ is $f_i$,
\item a Green form for $W_i$, $\mathfrak{g}(W_i)=(\omega_i,g_i).$
\end{itemize}
The form
$$\mathfrak{g}(\tilde{f}_i):=(0,-\frac{1}{2}\log \tilde{f}_i\overline{\tilde{f}}_i)$$
is a Green form for the divisor $\Div \tilde{f}_i$ on $X$.

Let $\star$ denote the $\star$-product of Green forms as described by
Burgos in \cite{Burgos2}. Then, we write
$$(\omega_{\rho},\widetilde {g}_{\rho})=\sum \mathfrak{g}(\tilde{f}_i)
\star \mathfrak{g}(W_i).$$
 Since the first component of $\mathfrak{g}(\tilde{f}_i)$ is zero, we
 have that $\omega_{\rho}=0$ as well. Moreover, since $(0,\widetilde {g}_{\rho})$
 is a Green form for $\sum_i \Div \tilde{f}_i \cap W_i = \sum_i \Div
 f_i=0$,  we can obtain a representative $g_{\rho}$ of $\widetilde{g}_{\rho }$
 that is a closed smooth form. Then $g_{\rho}$ is a representative of
 $\rho(\sum [f_i])$.

 Let us show now that $g_{\rho}$ is a representative of
 $\rho(\varphi_2(f))$ as well. By the results of the previous
 sections, the form $\rho(\sum_i \Gamma_{f_i})$ is obtained as
 follows. Let $Z\in Z^{p}(X,1)_{0}$ be a cycle in the normalized group
 that differs from $\sum \Gamma _{f_{i}}$ by a degenerate element. We
 consider a representative $(\omega_{Z},g_{Z})\in
 \tau \mmD^{2p}_{\mathcal{Z}^{p}}(X\times \square _{1},p)_{0}$ of $Z$.
 Since
$$ \beta = \delta_1^0(\omega_{Z},g_{Z})
-\delta_1^1 (\omega_{Z},g_{Z})$$
represents the class of $\sum_i \Div f_i=0$, the class of $\beta$ is
zero and hence there exists $(\omega,g)$ such that
$d_{\mmD}(\omega,g)= \beta$. Moreover, since
$d_{\mmD}\omega_{Z}=0$ and the complex $\tau \mmD^{\ast}(X\times
\square^{1},p)_{0}$ is acyclic (see the proof of Proposition
\ref{affine3}),  there exists $\gamma\in
\mmD^{2p-1}_{\log}(X\times \square^1,p)_{0}$ such that
$d_{\mmD}(\gamma)=\omega_{Z}$.  Then, $\rho(\sum_i
\Gamma_{f_i})$ is represented by $\omega+\delta(\gamma)$.

Therefore, we start by constructing the cycle $Z$ and suitable forms
$(\omega_{Z},g_{Z})$ representing the class of $Z$.
Consider the rational function $h_i \in k^*(X\times \square^1)$ given by
$$(y,(t_1:t_2)) \mapsto \frac{t_1-t_2\tilde{f_i}(y)}{t_1-t_2}.$$
If we write $\Div f_i = (\Div f_i)^0 - (\Div f_i)^{\infty}$ where
$(\Div f_i)^0$ is the divisor of
zeroes and $(\Div f_i)^\infty$ is the divisor of poles, the intersection of the divisor
of $h_i$ with $W_i$, $\Div h_i \cap W_i$, is exactly $\Gamma_{f_i} -
(\Div f_i)^{\infty}$.
Observe that $(\Div f_i)^{\infty}$ is a codimension $p$ degenerate
cycle. Moreover $\Div h_i \cap W_i$ belongs to $Z^p(X,1)_0$.
 Hence
$$Z = \sum \Div h_i \cap W_i$$
is the cycle we need.
Let $\mathfrak{g}(h_i)=(0,-\frac{1}{2}\log h_i\overline{h_i})$ be the
canonical Green form for $ \Div h_i$.
Then, as above, a Green form for $Z$ is given by
$$  \sum \mathfrak{g}(h_i) \star \mathfrak{g}(W_i)=(0,\widetilde g_Z).$$
Now, observe that
\begin{eqnarray*}
\delta(0,\widetilde g_Z) &=&
\sum_i \delta_1^0(\mathfrak{g}(h_i)) \star \mathfrak{g}(W_i) = \sum_i\mathfrak{g}(\tilde{f}_i) \star \mathfrak{g}(W_i)
= (0,\widetilde g_{\rho}).
 \end{eqnarray*}
Since we can assume that $g_{\rho}$ is a smooth representative of
$\widetilde g_{\rho }$, we have that $d_s(g_{\rho},0)= (0,g_{\rho}),$
and hence by the above description of $\rho$ we see that
$$\rho(\sum_i \Gamma_{f_i})=g_{\rho}.$$
This finishes the proof of the lemma.
\end{proof}

\begin{lema}\label{agreement3}
The following diagram is commutative:
$$ \xymatrix@R=2pt{ & \widehat{CH}^p(X) \ar[dd]^{\Phi} \\ \mmD^{2p-1}_{\log}(X,p) / \im d_{\mmD}
\ar[ur]^(0.6){\amap}\ar[dr]_(0.6){\amap} & \\
&\widehat{CH}^p(X,0) } $$
\end{lema}
\begin{proof} Let $\tilde{\alpha}\in \mmD^{2p-1}_{\log}(X,p) / \im
d_{\mmD}$. Then, the lemma follows from the equality
$$d(0,(\alpha,0),0,0,0)=(0,(d_{\mmD}\alpha,\alpha),0,0,0)+(0,0,0,0,\alpha)$$
 in $\widehat{CH}^p(X,0)$.
\end{proof}
This finishes the proof of Theorem \ref{chowagreement}.
\end{proof}

\subsection{Functoriality of the higher arithmetic Chow groups}

\begin{prop}[Pull-back] \label{chowpullbackth} Let $f:X\rightarrow Y$ be a morphism between two arithmetic varieties.
 Then, for all $p\geq 0$, there exists a chain complex,
$\whZ^p_f(Y,*)_0$ such that:
\begin{enumerate}[(i)]
\item  There is a quasi-isomorphism
$$\whZ^p_f(Y,*)_0 \xrightarrow{\sim} \whZ^p(Y,*)_0.$$
\item There is a pull-back
morphism
$$f^*: \whZ^p_f(Y,*)_0 \rightarrow \whZ^p(X,*)_0, $$
inducing a pull-back morphism of higher arithmetic Chow groups
$$\widehat{CH}^p(Y,n) \xrightarrow{f^*} \widehat{CH}^p(X,n), $$ for every
$p,n\geq 0$.
\item The pull-back is compatible with the morphisms
$\amap$ and $\zeta$. That is, there are commutative diagrams
\begin{equation}\label{functdiagram}\xymatrix{
\cdots \ar[r] & H_{\mmD}^{2p-n-1}(Y,\R(p)) \ar[r]^(0.57){\amap}
\ar[d]_{f^*} & \widehat{CH}^p(Y,n)  \ar[r]^{\zeta} \ar[d]_{f^*} &
CH^p(Y,n) \ar[r] \ar[d]_{f^*} & \cdots \\ \cdots \ar[r] &
H_{\mmD}^{2p-n-1}(X,\R(p)) \ar[r]_(0.57){\amap}   &
\widehat{CH}^p(X,n) \ar[r]_{\zeta} & CH^p(X,n)  \ar[r]  & \cdots }
\end{equation}
\end{enumerate}
\end{prop}
\begin{proof}
Recall that there are inclusions of complexes
\begin{eqnarray*}
Z^p_f(Y,*)_{0} & \subseteq & Z^p(Y,*)_0, \\
\mcH^p_f(Y,*)_{0} &\subseteq &
\mcH^p(Y,*)_0, \\
\mmD_{\A,\mmZ_f^p}^*(Y,p)_{0} &\subseteq &
\mmD_{\A,\mmZ^p}^*(Y,p)_0,
\end{eqnarray*}
which are quasi-isomorphisms. The pull-back by $f$ is defined for
any $\alpha$ in $Z^p_f(Y,*)_{0},$ in $\mcH^p_f(Y,*)_{0}$ or in
$\mmD_{\A,\mmZ_f^p}^*(Y,p)_{0} $. Moreover, by construction, there
is a commutative diagram
$$\xymatrix{
Z^p_{f}(Y,*)_{0} \ar[d]_{f^*} \ar[r]^{f_1} & \mcH^p_f(Y,*)_{0}
\ar[d]_{f^*} & \mmD_{\A,\mmZ_f^p}^*(Y,p)_{0} \ar[d]_{f^*}
 \ar[l]_{g_1}^{\sim} \ar[r]^(0.55){\rho}& \mmD_{\A}^*(Y,p)_0 \ar[d]_{f^*} & \ar[l]_{i} Z\mmD^{2p}(X,p)_* \ar[d]^{f^*}\\
Z^p (X,*)_0 \ar[r]_{f_1} & \mcH^p(X,*)_0 &
\mmD_{\A,\mmZ^p}^*(X,p)_0 \ar[l]^{g_1}_{\sim} \ar[r]_(0.55){\rho}
& \mmD_{\A}^*(X,p)_0 & \ar[l]_{i} Z\mmD^{2p}(Y,p)_*.}
$$
Let $\whZ^p_f(Y,*)_0$ denote the simple associated to the first
row diagram. Then, there is a pull-back morphism
$$f^*:\whZ^p_f(Y,*)_0 \rightarrow \whZ^p(X,*)_0.  $$
Moreover, as noticed in $\S$\ref{diagrambei5}, the natural
map
$$\whZ^p_f(Y,*)_0 \rightarrow \whZ^p(Y,*)_0 $$
is a quasi-isomorphism. Therefore, $(i)$ and $(ii)$ are proved.
Statement $(iii)$ follows from the construction.
\end{proof}

\begin{obs}
If the map is flat, then the pull-back is already defined at the
level of the chain complexes $\whZ^p(Y,*)_0$ and $\whZ^p(X,*)_0$.
\end{obs}

\begin{prop}[Functoriality of pull-back]
Let $f:X\rightarrow Y$ and $g:Y\rightarrow Z$ be two morphisms of
arithmetic varieties. Then,
$$ f^* \circ g^* =(g\circ f)^*:\widehat{CH}^p(Z,n) \rightarrow \widehat{CH}^p(X,n). $$
\end{prop}
\begin{proof}
Let $\whZ^p_{gf\cup g}(Z,n)_0$ be the subgroup of $\whZ^p(Z,n)_0$
obtained considering, at each of the complexes of the diagram
$\widehat{\mmZ}^p(Z,*)_0$, the subvarieties $W$ of $Z\times
\square^n$ intersecting properly the faces of $\square^n$ and such
that
\begin{itemize}
\item $X\times W \times \square^n$ intersects properly the graph
of $g\circ f$,
\item $Y\times W \times \square^n$ intersects
properly the graph of $g$.
\end{itemize}
That is,
$$\whZ^p_{gf\cup g}(Z,n)_0=\whZ^p_{gf}(Z,n)_0\cap \whZ^p_{g}(Z,n)_0.$$
Then, the proposition follows from the commutative diagram
$$ \xymatrix@R=5pt{
& \whZ^p(X,*)_0 & \\ \whZ^p_{gf\cup g}(Z,*)_0 \ar[ur]^{(g\circ
f)^*} \ar[dr]_{g^*} &
\\ & \whZ^p_f(Y,*)_0. \ar[uu]_{f^*}
}$$
\end{proof}

\begin{cor}[Homotopy invariance] \label{homotopyinvariancecor}
Let $\pi: X\times \A^m \rightarrow X$ be the projection on $X$.
Then, the pull-back map
$$\pi^*: \widehat{CH}^p(X,n)  \rightarrow \widehat{CH}^p(X\times \A^m,n)   $$
is an isomorphism for all $n\geq 1$.
\end{cor}
\begin{proof}
It follows from the five lemma in the diagram
\eqref{functdiagram}, using the fact that both the higher Chow
groups and the Deligne-Beilinson cohomology groups are homotopy
invariant.
\end{proof}

\section{Product structure}\label{defiprod}
Let $X,Y$ be arithmetic varieties over an arithmetic field $K$. In this
section, we define an external product,
$ \widehat{CH}^*(X,*)\otimes \widehat{CH}^*(Y,*)\rightarrow \widehat{CH}^*(X\times Y,*),$
and an internal product
$ \widehat{CH}^*(X,*)\otimes \widehat{CH}^*(X,*)\rightarrow \widehat{CH}^*(X,*),$
for the higher arithmetic Chow groups. The internal product endows
$\widehat{CH}^*(X,*)$ with a ring structure. It will be shown that
this product is commutative and associative. There are two main
technical difficulties. The first one is that we are representing
a cohomology class with support in a cycle by a pair of forms, the
first one smooth on the whole variety and the second one with
singularities along the cycle. The product of two singular forms has
singularities along the union of the singular locus. Therefore, in
order to define a cohomology class with support on the intersection of
two cycles we need a little bit of homological algebra. To this end we
adapt the technique used in \cite{Burgos2}. The second difficulty is
that the external product in higher Chow groups is not graded
commutative at the level of complexes, but only graded commutative up
to homotopy. To have explicit homotopies we will adapt the techniques
of \cite{Levine1}.

Recall that the higher arithmetic Chow groups are the homology
groups of the simple complex associated to a diagram of complexes.
Therefore, in order to define a product, we use the general
procedure developed by Beilinson, as recalled in
$\S$\ref{productdiagram}. To this end, we need to define a product for each of the
complexes in the diagram $\whmZ^p(X,*)_0$  (\ref{chowcomplex1}), commuting with the
morphisms $f_1$, $g_1$, $\rho$ and $i$.
The pattern for the external product construction is analogous to
the pattern followed to define the external product for the cubical higher
Chow groups, described in $\S$\ref{chowprod}.

For the complex $Z^p(X,*)_0$ we already have an external product
recalled in \S \ref{chowprod}.
Since the complex $\mcH^p(X,*)_0$ is isomorphic to
$Z^p_{\R}(X_{\R},*)_0$,
  the external product on the complex $\mcH^*(X,*)_0$ can
be defined by means of this isomorphism. We will now construct the
product for the remaining complexes.

\subsection{Product structure on the complexes $\mmD^*_{\A}(X,p)$ and $Z\mmD^{2p}(X,p)_*$}
We start by defining a product structure on $\mmD^*_{\A}(X,p)$. Let
$$ X\times Y \times \square^n\times \square^m  \xrightarrow{p_{13}} X \times \square^n, \qquad
   X\times Y \times \square^n\times \square^m  \xrightarrow{p_{24}}
Y \times \square^m $$
be the projections indicated by
the subindices. For every $\omega_1\in \tau\mmD^{r}_{\log}(X\times
\square^n,p)$ and $\omega_2\in \tau\mmD^{s}_{\log}(Y\times
\square^m,q)$, we define
$$\omega_1\bullet_{\A} \omega_2 := (-1)^{ns} p_{13}^*\omega_1\bullet p_{24}^*\omega_2\in
\tau\mmD^{r+s}_{\log}(X\times Y\times \square^{n+m},p+q). $$ This
gives a map
\begin{eqnarray*}
\mmD^{r_1}_{\A}(X,p) \otimes \mmD^{r_2}_{\A}(Y,q) &
\xrightarrow{\bullet_{\A}} & \mmD^{r_1+r_2}_{\A}(X\times Y,p+q) \\
(\omega_1,\omega_2) & \mapsto & \omega_1\bullet_{\A} \omega_2,
\end{eqnarray*}
where $\bullet$ in the right hand side is the product in the Deligne
complex (see \S \ref{dbcohomology}).

\begin{lema}\label{prod}
The map $\bullet_{\A}$ satisfies the Leibniz rule. Therefore,
there is a cochain morphism
$$s(\mmD^{*}_{\A}(X,p) \otimes \mmD^{*}_{\A}(Y,q))
\xrightarrow{\bullet_{\A}}  \mmD^{*}_{\A}(X\times Y,p+q) . $$
\end{lema}
\begin{proof} Let
 $\omega_1\in \tau\mmD^r_{\log}(X,n)$ and $\omega_2\in
 \tau\mmD^s_{\log}(Y,m)$. By definition of $\delta$, the following
 equality
  holds
 $$ \delta(p_{13}^*\omega_1\bullet p_{24}^*\omega_2)=p_{13}^*(\delta \omega_1)\bullet p_{24}^*\omega_2 +(-1)^{n}
p_{13}^*\omega_1\bullet p_{24}^*(\delta \omega_2).$$ Then,
\begin{eqnarray*}
d_s(\omega_1\bullet_{\A} \omega_2) &=&
(-1)^{ns}d_s(p_{13}^*\omega_1\bullet p_{24}^*\omega_2)
\\ &=& (-1)^{ns}d_{\mmD}(p_{13}^*\omega_1\bullet
p_{24}^*\omega_2)+(-1)^{r+s+ns}\delta(p_{13}^*\omega_1\bullet p_{24}^*\omega_2)\\
&=& (-1)^{ns}d_{\mmD}(p_{13}^*\omega_1)\bullet p_{24}^*\omega_2
+(-1)^{r+ns} p_{13}^*\omega_1\bullet d_{\mmD}(p_{24}^*\omega_2) +
\\ && + (-1)^{r+s+ns}p_{13}^*(\delta \omega_1)\bullet p_{24}^*\omega_2 +(-1)^{r+s+n+ns}
p_{13}^*\omega_1\bullet p_{24}^*(\delta \omega_2) \\
&=&d_{\mmD}\omega _{1}\bullet_{\A} \omega_2+(-1)^{r+n} \omega_1\bullet_{\A}
d_{\mmD}(\omega_2)+\\
&& + (-1)^{r}\delta \omega _{1}\bullet_{\A} \omega_2+(-1)^{r+n+s} \omega_1\bullet_{\A}
\delta (\omega_2)\\
&=&
d_s(\omega_1)\bullet_{\A} \omega_2+(-1)^{r+n} \omega_1\bullet_{\A}
d_s(\omega_2),
\end{eqnarray*}
as desired.
\end{proof}

\begin{df}\label{AA}
Let $\tau\mmD_{\log}^*(X\times Y\times \square^*\times \square^*,p)_0$ be the
$3$-iterated cochain complex whose $(r,-n,-m)$-th graded piece is
the group $\tau\mmD_{\log}^r(X\times Y\times \square^n\times \square^m,p)_0$
and whose differentials are $(d_{\mmD},\delta,\delta)$. Let
\begin{equation}\label{doble}
\mmD^*_{\A\times \A}(X\times Y,p)_0:= s \big(
\tau\mmD_{\log}^*(X\times Y\times
\square^*\times \square^*,p)_0\big) \end{equation} be the associated
simple complex.
\end{df}

\begin{obs} \label{RA}
Observe that there is a cochain morphism
$$\mmD^*_{\A\times \A}(X\times Y,p)_0 \xrightarrow{\kappa}
\mmD^*_{\A}(X\times Y,p)_0 $$
sending $\alpha\in\tau\mmD_{\log}^r(X\times Y\times \square^n\times
\square^m,p) $ to $\alpha \in \tau\mmD_{\log}^r(X\times Y\times
\square^{n+m},p)$ under the identification
\begin{eqnarray*}
\square^{n+m}  & \xrightarrow{\cong} & \square^n\times \square^m \\
(x_1,\dots,x_{n+m})& \mapsto &
((x_1,\dots,x_n),(x_{n+1},\dots,x_{n+m})) .
\end{eqnarray*}
Moreover, the product $\bullet _{A}$ that we have defined previously,
factors through
the morphism $\kappa $ and a
product, also denoted by $\bullet_{A}$,
\begin{displaymath}
  \mmD^{\ast}_{\A}(X,p) \otimes \mmD^{\ast}_{\A}(Y,q)
\xrightarrow{\bullet_{\A}} \mmD^{\ast}_{\A\times \A}(X\times Y,p+q).
\end{displaymath}
\end{obs}

In order to define the product on the complex $
Z\mmD^{2p}(X,p)_{\ast}$,
recall that we have an isomorphism (see \cite{Burgos2})
\begin{displaymath}
  Z\mmD^{2p}(X,p)\cong ZE^{p,p}_{\R}(X)(p)
\end{displaymath}
and that the restriction of the product $\bullet$ to this subspace is
given by the product $\wedge$.

The inclusion $i$ is compatible with the product $\bullet_{\A}$ and
the product $\wedge$. That is, consider the projections
$p_X:X\times Y\rightarrow X$ and $p_Y:X\times Y\rightarrow Y$. Then,
if $\alpha\in ZE^{p,p}_{\R}(X)(p)$ and $\beta\in ZE^{q,q}_{\R}(Y)(q)$,
we put
$$\alpha\wedge \beta = p_X^*(\alpha) \wedge p_Y^*(\beta) \in
ZE^{p+q,p+q}_{\R}(X\times Y)(p+q).$$
We have a commutative diagram
$$
\xymatrix{
ZE^{p,p}(X)(p)\otimes ZE^{q,q}(Y)(q) \ar[r]^(.5){\wedge}
\ar[d]_{i\otimes i} & ZE^{p+q,p+q}(X\times Y)(p+q) \ar[d]^{i}\\
s(\mmD^{*}_{\A}(X,p)_0  \otimes
\mmD^{*}_{\A}(Y,q)_0) \ar[r]_(.53){\bullet_{\A}} &
\mmD^{*}_{\A}(X\times Y,p+q)_0. }
$$

\subsection{Product structure on the complex $\mmD_{\A,\mathcal{Z}^p}^*(X,p)$}
We define here a product on the complex
$\mmD_{\A,\mathcal{Z}^p}^*(X,p)$. It will be compatible with the
product on $\mmD_{\A}^*(X,p)$, under the morphism $\rho$, and with
the product on $\mcH^p(X,*)_0$ under $g_1$.

Let $X,Y$ be two real varieties. For every $p$, let
$\mmZ_{X,n}^p$ be the subset of codimension $p$ subvarieties of
$X\times \square^n$ intersecting properly the faces of
$\square^n$. Let
$$ \mathcal{Z}_{X,Y,n,m}^{p,q}\subseteq \mmZ^{p+q}_{X\times Y,n+m}$$
be the subset
of the set of codimension $p+q$ subvarieties of $X\times Y\times
\square^{n+m}$, intersecting properly the faces of
$\square^{n+m}$, which are obtained as the cartesian product
$Z\times W$ with $Z\in \mmZ^p_{X,n}$ and $W\in \mmZ^q_{Y,m}$.

For shorthand, we make the following identifications:
\begin{eqnarray*}
\mmZ_{Y,m}^q =\{X\times Z\mid Z\in \mmZ_{Y,m}^q\}& \subseteq &
\mmZ^{q}_{X\times Y,n+m},\\
\mmZ^p_{X,n}=\{W\times Y\mid W\in \mmZ_{X,m}^p\}& \subseteq & \mmZ_{X\times
Y,n+m}^{p}.\end{eqnarray*}

To ease the notation, we write temporarily
$$  \square^{n,m}_{X,Y}:=X\times Y \times \square^{n}\times \square^{m}.$$

For every $n,m,p,q$, let $j_{X,Y}^{p,q}(n,m)$ be the morphism
{\small $$\mmD_{\log}^*(\square^{n,m}_{X,Y}\setminus
\mathcal{Z}_{X,n}^p,p+q)\oplus
\mmD_{\log}^*(\square^{n,m}_{X,Y}\setminus\mmZ_{Y,m}^q,p+q)\xrightarrow{j_{X,Y}^{p,q}(n,m)}
\mmD_{\log}^*(\square^{n,m}_{X,Y} \setminus\mathcal{Z}_{X,n}^p\cup
\mmZ_{Y,m}^q,p+q)
$$}
induced on the limit complexes by the morphism $j$ in Lemma \ref{delignesupports}.

\begin{lema}
There is a short exact sequence {\small $$0\rightarrow
\mmD_{\log}^*(\square^{n,m}_{X,Y}\setminus
\mathcal{Z}_{X,Y,n,m}^{p,q},p+q) \rightarrow
\mmD_{\log}^*(\square^{n,m}_{X,Y}\setminus \mathcal{Z}_{X,n}^p
,p+q)\oplus
\mmD_{\log}^*(\square^{n,m}_{X,Y}\setminus\mmZ_{Y,m}^q,p+q)$$
$$\xrightarrow{j_{X,Y}^{p,q}(n,m)} \mmD_{\log}^*(\square^{n,m}_{X,Y}
\setminus\mathcal{Z}_{X,n}^p\cup \mmZ_{Y,m}^q,p+q)\rightarrow 0.
$$}
\end{lema}
\begin{proof}
It follows from Lemma \ref{delignesupports}.
\end{proof}

By the quasi-isomorphism between the simple complex and the kernel
of an epimorphism (see \eqref{simplekernel}), for every
$n,m$, there is a
quasi-isomorphism
\begin{eqnarray*}
\mmD_{\log}^*(\square^{n,m}_{X,Y}\setminus
\mathcal{Z}_{X,Y,n,m}^{p,q},p+q) & \xrightarrow{\sim}
 & s(-j_{X,Y}^{p,q}(n,m))^* \\
\omega & \mapsto & (\omega,\omega,0).\end{eqnarray*}
It induces a quasi-isomorphism
\begin{equation}\label{quasiiso1}
\mmD_{\log,\mathcal{Z}_{X,Y,n,m}^{p,q}}^*(\square^{n,m}_{X,Y},p+q)\xrightarrow{\sim}
s\left(\mmD^*_{\log}(\square^{n,m}_{X,Y},p+q)^*
\xrightarrow{i_{X,Y}^{p,q}(n,m)} s(-j_{X,Y}^{p,q}(n,m))\right)^*,
\end{equation}
where $i_{X,Y}^{p,q}(n,m)$ is defined by
\begin{eqnarray*}
\mmD^*_{\log}(\square^{n,m}_{X,Y},p+q) &
\xrightarrow{i_{X,Y}^{p,q}(n,m)}&
  s(-j_{X,Y}^{p,q}(n,m))^* \\
\omega \qquad& \mapsto & (\omega,\omega,0).
\end{eqnarray*}

\begin{obs}
\label{si} Observe that there is an induced bicubical cochain complex
structure on
$s(i_{X,Y}^{p,q}({\cdot},{\cdot}))^*$. For every $r$, let
$s(i_{X,Y}^{p,q}(*,*))^r_0$ denote the $2$-iterated complex obtained
by taking the normalized complex functor to both cubical
structures. Consider the $3$-iterated complex
$s(i_{X,Y}^{p,q}(*,*))_0^*$ whose piece of degree $(r,-n,-m)$ is the
group $\tau _{r\leq 2p+2q}s(i_{X,Y}^{p,q}(n,m))_0^r$, and whose
differential is
$(d_s,\delta,\delta)$. Denote by $s(i_{X,Y}^{p,q})_0^*$ the
associated simple complex. Observe that the differential of
$\alpha=(\alpha_0,(\alpha_1,\alpha_2),\alpha_3)\in
s(i_{X,Y}^{p,q})_0^r$ is given by
$$d_s'(\alpha_0,(\alpha_1,\alpha_2),\alpha_3) =
(d_{\mmD}\alpha_0,(\alpha_0-d_{\mmD}\alpha_1,\alpha_0-
d_{\mmD}\alpha_2),-\alpha_1+\alpha_2+d_{\mmD}\alpha_3).$$
\end{obs}

\begin{df} Let $\bullet_{\A}$ be the map
\begin{eqnarray*}
\mmD^{r}_{\log,\mathcal{Z}^p}(X\times \square^n,p)_0  \otimes
\mmD^{s}_{\log,\mathcal{Z}^q}(Y\times \square^m,q)_0  &
\xrightarrow{\bullet_{\A}} & s(i_{X,Y}^{p,q}(n,m))_0^{r+s}
\end{eqnarray*} defined by sending $ (\omega,g)\otimes (\omega',g')$  to
$$  (-1)^{ns}(\omega\bullet
\omega',(g\bullet  \omega',(-1)^r\omega\bullet g'),(-1)^{r-1}
g\bullet g').
$$
\end{df}

\begin{lema}
The map $\bullet_{\A}$ defines a pairing of complexes
$$s\left(\mmD^{*}_{\A,\mathcal{Z}^p}(X,p)_0  \otimes
\mmD^{*}_{\A,\mathcal{Z}^q}(Y,q)_0\right)  \xrightarrow{\bullet_{\A}}
s(i_{X,Y}^{p,q})^{*}_0. $$
\end{lema}
\begin{proof} Let $(\omega,g)\in \mmD^{r}_{\log,\mathcal{Z}^p}(X\times
\square^n,p)_0 $ and $(\omega',g')\in
\mmD^{s}_{\log,\mathcal{Z}^q}(Y\times \square^m,q)_0$. Then, we
have to see that
$$d_s'((\omega,g)\bullet_{\A} (\omega',g'))=
d_s'(\omega,g)\bullet_{\A}
(\omega',g')+(-1)^{r-n}(\omega,g)\bullet_{\A} d_s'(\omega',g').
$$
That is, we have to show that the following two equalities hold:
\begin{eqnarray*}
d_{s} ((\omega,g)\bullet_{\A} (\omega',g')) &=&
d_s(\omega,g)\bullet_{\A}
(\omega',g')+(-1)^{r-n}(\omega,g)\bullet_{\A} d_s(\omega',g') \\
\delta ((\omega,g)\bullet_{\A} (\omega',g')) &=& (-1)^s
\delta(\omega,g)\bullet_{\A}
(\omega',g')+(-1)^{n}(\omega,g)\bullet_{\A} \delta(\omega',g').
\end{eqnarray*}
The proof of the second equality is analogous to the proof
of Lemma \ref{prod}. The first equality is a direct computation.
\end{proof}

 We define a
complex $\mmD_{\A\times \A,\mathcal{Z}_{X,Y}^{p,q}}^{*}(X\times
Y,p+q)_0$ that is analogous to the complex $\mmD^*_{\A\times
  \A}(X,p)_0$ of Definition \ref{AA}.

\begin{df} Let
  $\mmD_{\A\times
\A,\mathcal{Z}_{X,Y}^{p,q}}^{*}(X\times Y,p+q)_0$ be the simple
complex
associated to the $3$-iterated complex whose $(r,-n,-m)$ graded
piece is $\tau\mmD_{\log,\mathcal{Z}_{X,Y,n,m}^{p,q}}^{r}(X\times
Y\times \square^n\times\square^m,p+q)_0$.
\end{df}

As in Remark \ref{RA}, we will denote by $\kappa $ the morphisms obtained
by identifying
$\square^{n}\times \square^{m}$ with $\square^{n+m}$.
\begin{displaymath}
  \mmD_{\A\times \A,\mathcal{Z}^{p+q}}^{*}(X\times
Y,p+q)_0 \xrightarrow{\kappa}
\mmD_{\A,\mathcal{Z}^{p+q}}^{*}(X\times Y,p+q)_0.
\end{displaymath}
We will denote by $\rho $ the morphisms obtained by forgetting the
support
\begin{eqnarray*}
\mmD_{\A\times \A,\mathcal{Z}_{X,Y}^{p,q}}^{*}(X\times Y,p+q)_0
&\xrightarrow{\rho} & \mmD^*_{\A\times \A}(X\times Y,p)_0,\\
s(i_{X,Y}^{p,q})_0^* &\xrightarrow{\rho} & \mmD^*_{\A\times
\A}(X\times Y,p)_0.
\end{eqnarray*}
There are also natural morphisms, whose definitions are obvious,
\begin{eqnarray*}
\mmD_{\A\times \A,\mathcal{Z}_{X,Y}^{p,q}}^{*}(X\times Y,p+q)_0
&\rightarrow & \mmD_{\A\times \A,\mathcal{Z}^{p+q}}^{*}(X\times
Y,p+q)_0,\\
\mmD_{\A\times \A,\mathcal{Z}_{X,Y}^{p,q}}^{*}(X\times Y,p+q)_0
&\rightarrow& s(i_{X,Y}^{p,q})_0^*.
\end{eqnarray*}

\begin{lema}
The natural map
\begin{equation}\label{eq:6}
\mmD_{\A\times \A,\mathcal{Z}_{X,Y}^{p,q}}^{*}(X\times Y,p+q)_0
\rightarrow s(i_{X,Y}^{p,q})_0^*
\end{equation}
is a quasi-isomorphism. Moreover, it commutes with  $\rho$.
\end{lema}
\begin{proof} It follows from the quasi-isomorphism \eqref{quasiiso1}.
\end{proof}

The external product on
$\mmD^{*}_{\A,\mathcal{Z}^*}({\cdot},*)_0$ is given, in the
derived category of complexes, by {\small
$$ \xymatrix@C=15pt{\mmD^{r}_{\A,\mathcal{Z}^p_X}(X,p)_0 \otimes
\mmD^{s}_{\A,\mathcal{Z}_Y^q}(Y,q)_0  \ar[r]^(0.65){\bullet_{\A}} &
s(i_{X,Y}^{p,q})_0^{r+s} & \\ & \mmD^{r+s}_{\A\times
\A,\mathcal{Z}^{p,q}_{X,Y}}(X\times Y,p+q)_0 \ar[u]^{\sim}
\ar[r]^(0.52){\kappa} & \mmD^{r+s}_{\A,\mathcal{Z}_{X\times Y}^{p+q}}(X\times
Y,p+q)_0.} $$ }
The fact that to define the product in this complex we need to invert
a quasi-isomorphism is the main reason of the complexity of the
definition of the product on the higher arithmetic Chow groups.

By definition, it is clear that this morphism commutes with the
morphism defined on the complex $\mmD^*_{\A}(X,p)$. It
remains to be seen
that the product on $\mmD_{\A,\mmZ^p}^{2p-n}(X,p)_0$ is compatible
with the product on $\mcH^p(X,n)_0$, under the quasi-isomorphism
$g_1$.

Let $\omega\in s(i_{X\times Y}^{p,q})^{2p+2q-l}_0$ and let
$$(\omega_{l}^0,\dots,\omega_l^l)\in \bigoplus_{j=0}^l
\tau_{\leq 2p+2q}s(i_{X,Y}^{p,q}(j,l-j))_0^{2p+2q}$$ be the components of $\omega$
corresponding to the degree $(2p+2q,-j,j-l)$. These are the components
that have maximal degree as differential forms and, by the definition
of the truncated complex they satisfy
$d_{s}\omega_{l}^j=0$. Thus, the form
$\omega ^{j}_{l}$ defines a
cohomology class $[\omega_{l}^j]$ in the complex
$s(i_{X,Y}^{p,q}(j,l-j))_0^{*}$.
Since there is a quasi-isomorphism
$$\mmD^{*}_{\log,\mmZ^{p,q}_{X,Y}}(X\times Y\times \square^{l},p+q)_0
\xrightarrow{\sim} s(i_{X,Y}^{p,q}(j,l-j))_0^{*}, $$ we obtain a
cohomology class  in
$H^*(\mmD^{*}_{\log,\mmZ^{p,q}_{X,Y}}(X\times Y\times
\square^{l},p+q)_0)$. Hence, a cohomology class
$[\omega_l^j]\in \mcH^{p+q}(X\times Y,l)_{0}$. This procedure defines a chain
morphism, denoted $g_{1}$,
\begin{eqnarray*}
s(i_{X,Y}^{p,q})_0^{2p+2q-l}&\xrightarrow{g_{1}}& \mcH^{p+q}(X\times
Y,l)_0\\
\omega &\longmapsto & \sum_{j}[\omega^{j}_{l} ].
\end{eqnarray*}
By composition, we can define a morphism, also denoted $g_{1}$,
\begin{displaymath}
  \mmD^{*}_{\A\times A,\mmZ^{p,q}_{X,Y}}(X\times Y,p+q)_0
  \xrightarrow{g_{1}}
  \mcH^{p+q}(X\times Y,\ast)_0.
\end{displaymath}
Moreover there is a commutative diagram
\begin{displaymath}
  \xymatrix{ \mmD^{*}_{\A\times \A,\mmZ^{p,q}_{X,Y}}(X\times Y,p+q)_0
    \ar_{\kappa }[dd] \ar^{g_{1}}[dr]&\\
  & \mcH^{p+q}(X\times Y,\ast)_0\\
   \mmD^{*}_{\A,\mmZ^{p+q}}(X\times Y,p+q)_0 \ar_{g_{1}}[ur]&
}
\end{displaymath}

\begin{prop} Let $Z \in \mmZ^p_{X,n}$ and $T\in \mmZ^q_{Y,m}$.
Let $[(\omega_Z,g_Z)]\in \mcH^{p}(X,n)_0$ represent the class of a
cycle $z\in Z^{p}(X,n)_0$ with support on $Z$ and $[(\omega_T,g_T)]\in
\mcH^{q}(Y,m)_0$ represent the class of a cycle $t \in	
Z^{q}(Y,m)_0$ with support on $T$. Then,
$$ [(\omega_Z,g_Z)\bullet_{\A} (\omega_T,g_T)]\in \mcH^{p+q}(X\times Y,n+m)_0$$
represents the class of the cycle $z\times t$ in $Z^{p+q}(X\times
Y,n+m)_0$.
\end{prop}
\begin{proof}
It follows from \cite{GilletSouleIHES}, Theorem 4.2.3 and
\cite{Burgos2}, Theorem 7.7.
\end{proof}

\begin{cor}
For every $p,q,n,m$, the following diagram is commutative:
$$\xymatrix{
\mmD^{2p-n}_{\A,\mathcal{Z}^p}(X,p)_0\otimes
\mmD^{2q-m}_{\A,\mathcal{Z}^q}(Y,q)_0 \ar[r]^(.55){g_1} \ar[d]_{\bullet_{\A}}
& \mcH^{p}(X,n)_0 \otimes \mcH^{q}(Y,m)_0 \ar[d]^{\times} \\
s(i_{X,Y}^{p,q})^{2p+2q-n-m}_0 \ar[r]_{g_1} &\mcH^{p+q}(X\times Y,n+m)_0
}$$
\ \hfill $\square$
\end{cor}

\subsection{Product structure on the higher arithmetic Chow
  groups}\label{diagramgros}

Once we have defined a compatible product on each of the complexes
involved, the
product on the higher arithmetic Chow groups is given by the following
diagram.
{\footnotesize $$
\xymatrix@C=-35pt{ & \mcH^{p}(X,n)_0 \otimes \mcH^{q}(Y,m)_0
\ar[dd]^{\times} & & \mmD_{\A}^{2p-n}(X,p)_0\otimes
\mmD_{\A}^{2q-m}(Y,q)_0 \ar[dd]^{\bullet_{\A}} \\
Z^{p}(X,n)_0\otimes Z^q(Y,m)_0 \ar[dd]^{\times} \ar[ur]^{f_1} & &
\ar[ul]_{g_1}^{\sim} \mmD^{2p-n}_{\A,\mathcal{Z}^p}(X,p)_0\otimes
\mmD^{2q-m}_{\A,\mathcal{Z}^q}(Y,q)_0 \ar[ur]^{\rho}
\ar[dd]^{\bullet_{\A}} & & ZD^{2p}(X,p)_n\otimes ZD^{2q}(Y,q)_m
\ar[ul]_{i} \ar[dd]^{\wedge}
 \\
 & \mcH^{p+q}(X\times Y,n+m)_0 & & \mmD_{\A\times \A}^{2(p+q)-n-m}
 (X\times Y,p+q)_0  \\
Z^{p+q}(X\times Y,n+m)_0  \ar[ur]^{f_1} & & \ar[ul]_{g_1}
s(i_{X,Y}^{p,q})^{2p+2q-n-m}_0 \ar[ur]^{\rho} & & \ar[ul]_{i}
ZD^{2(p+q)}(X\times Y,p+q)_{n+m}   \ar@{=}[dd]
 \\
 & \mcH^{p+q}(X\times Y,n+m)_0 \ar@{=}[dd]
 \ar@{=}[uu] & & \mmD_{\A\times \A}^{2(p+q)-n-m}(X\times Y,p+q)_0
 \ar[dd]^{\kappa}  \ar@{=}[uu]\\
Z^{p+q}(X\times Y,n+m)_0 \ar@{=}[dd] \ar@{=}[uu] \ar[ur]^{f_1} & &
\ar[ul]_{g_1} \mmD^{2(p+q)-n-m}_{\A\times
\A,\mathcal{Z}^{p,q}_{X,Y}}(X\times Y,p+q)_0 \ar[ur]^{\rho}
\ar[dd]^{\kappa} \ar@{->}[uu]_{\sim}  & & \ar[ul]_{i}
ZD^{2(p+q)}(X\times Y,p+q)_{n+m}   \ar@{=}[dd]
 \\
 & \mcH^{p+q}(X\times Y,n+m)_0 & & \mmD_{\A}^{2(p+q)-n-m}(X\times Y,p+q)_0 \\
Z^{p+q}(X\times Y,n+m)_0 \ar[ur]^{f_1} & & \ar[ul]_{g_1}^{\sim}
\mmD^{2(p+q)-n-m}_{\A,\mathcal{Z}^{p+q}}(X\times Y,p+q)_0
\ar[ur]^{\rho}
& & \ar[ul]_{i} ZD^{2(p+q)}(X\times Y,p+q)_{n+m}
}
$$ }
Observe that, in the first set of vertical arrows is where the product
is defined, in the second set of vertical arrows we are just inverting
the quasi-isomorphism \eqref{eq:6}, finally in the last set of
vertical arrows we are applying the morphism $\kappa $.

The above diagram induces a morphism in the
derived category of chain complexes
$$s\big(\whmZ^p(X,*)_0\otimes \whmZ^q(Y,*)_0\big)\xrightarrow{\cup}
s\big(\whmZ^{p+q}(X\times Y,*)_0\big)=\whZ^{p+q}(X\times Y,*)_0.$$
Recall here the notation we are using, the symbol $\whmZ^p(X,*)_0$
denotes the diagram where the symbol $\whZ^p(X,*)_0$ denotes the
associated simple complex.

By $\S$\ref{productdiagram}, for any  $\beta \in \Z$ there is
a morphism $\star_{\beta}$
$$\whZ^p(X,*)_0\otimes \whZ^q(Y,*)_0 \xrightarrow{\star_{\beta}} s\big(\whmZ^p(X,*)_0\otimes \whmZ^q(Y,*)_0\big). $$
The composition of $\star_{\beta}$ with $\cup$ induces a product
$$ \widehat{CH}^p(X,n) \otimes \widehat{CH}^q(Y,m) \xrightarrow{\cup} \widehat{CH}^{p+q}(X\times Y,n+m),$$
independent of $\beta$.

Finally the pull-back by the diagonal map $X\xrightarrow{\Delta}
X\times X $ gives an internal product on $\widehat{CH}^p(X,*)$:
$$ \widehat{CH}^p(X,n) \otimes \widehat{CH}^q(X,m) \xrightarrow{\cup }
\widehat{CH}^{p+q}(X\times X,n+m)\xrightarrow{\Delta^*}
\widehat{CH}^{p+q}(X,n+m).$$

Thus, in the derived category of complexes, the product is given by the
composition {\small $$ \xymatrix{ \whZ^p(X,n)_0\otimes
\whZ^q(X,m)_0 \ar[d]^{\star_{\beta}}
  \\ s\big(\whmZ^p(X,n)_0\otimes \whmZ^q(X,m)_0\big)  \ar[r]^{\cup} &
  \whZ^{p+q}(X\times X,n+m)_0
   \\ &   \whZ^{p+q}_{\Delta}(X\times X,n+m)_0
   \ar[u]_{\sim} \ar[r]^(0.55){\Delta^*} & \whZ^{p+q}(X,n+m)_0.} $$ }

\begin{obs} It follows from the definition that, for $n=0$, the product $\cup$ agrees with the product on the arithmetic
Chow group $\widehat{CH}^p(X)$ defined in
\cite{Burgos2}.
\end{obs}

\subsection{Commutativity of the product}\label{commut12} Let $X,Y$
be arithmetic varieties over a field $K$. We
prove here that the pairing defined in the previous subsection on the
higher arithmetic Chow groups is commutative, in the sense
detailed below.

We first introduce some notation:
\begin{itemize}
\item If $B_*,C_*$ are chain complexes, let
$$\sigma: s(B_*\otimes C_*) \rightarrow s(C_*\otimes B_*) $$ be
the map sending $b\otimes c \in B_n\otimes C_m$ to $(-1)^{nm}c\otimes
b \in C_m\otimes B_n$.
\item Let $\sigma_{X,Y}$ be the morphism
$$ \sigma_{X,Y}: Y\times X \rightarrow X\times Y$$ interchanging
$X$ with $Y$.
\end{itemize}
We will prove that there is a commutative diagram
$$\xymatrix{ \widehat{CH}^p(X,n) \otimes \widehat{CH}^q(Y,m)
  \ar[r]^{\cup } \ar[d]_{\sigma}
  & \widehat{CH}^{p+q}(X\times Y,n+m) \ar[d]^{\sigma^*_{X,Y}}\\
  \widehat{CH}^q(Y,m) \otimes \widehat{CH}^p(X,n) \ar[r]^{\cup } &
  \widehat{CH}^{p+q}(Y\times
  X,n+m) }$$
In particular, the internal product on the higher
arithmetic Chow groups will be graded commutative with respect to
the degree $n$.
That is, if $W\in \widehat{CH}^p(X,n)$ and $Z\in
\widehat{CH}^q(X,m)$, then
$$W\cup  Z = (-1)^{nm} Z \cup  W. $$

Recall that, by definition, the product factorizes as {\small
$$\widehat{CH}^p(X,n) \otimes \widehat{CH}^q(Y,m)
\xrightarrow{\star_{\beta}} H_{n+m}(s(\whmZ^p(X,*)_0\otimes
\whmZ^q(Y,*)_0)) \xrightarrow{\cup} \widehat{CH}^{p+q}(X\times
Y,n+m).$$}

By Lemma \ref{star}, this factorization is independent on the integer
$\beta $. Moreover, there is a commutative diagram
\begin{displaymath}
\xymatrix{ \widehat{CH}^p(X,n) \otimes \widehat{CH}^q(Y,m)
  \ar[r]^(.44){\star_{\beta}} \ar[d]_{\sigma}
  & H_{n+m}(s(\whmZ^p(X,*)_0\otimes
\whmZ^q(Y,*)_0)) \ar[d]^{\sigma}\\
  \widehat{CH}^q(Y,m) \otimes \widehat{CH}^p(X,n) \ar[r]^(.43){\star_{1-\beta}} &
  H_{n+m}(s(\whmZ^p(Y,*)_0\otimes
\whmZ^q(X,*)_0))}
\end{displaymath}
Therefore,
all that remains is to check the commutativity for
\begin{equation}\label{commut10}
s\big(\whmZ^p(X,*)_0\otimes \whmZ^q(Y,*)_0\big)
\stackrel{\cup}{\dashrightarrow} \whZ^{p+q}(X\times Y,*)_0.
\end{equation}
Hence, we want to see that, in the derived category of chain complexes,
there is a commutative diagram
$$ \xymatrix{ s\big(\whmZ^p(X,*)_0\otimes \whmZ^q(Y,*)_0\big)
  \ar[d]_{\sigma} \ar[r]^(.57){\cup} &
 \whZ^{p+q}(X\times
Y,*)_0 \ar[d]^{\sigma_{X,Y}^*} \\
  s\big(\whmZ^q(Y,*)_0\otimes \whmZ^p(X,*)_0\big) \ar[r]_(.57){\cup}  &
  \whZ^{p+q}(Y\times X,*)_0. }$$

The obstruction to strict commutativity comes from the change of
coordinates
\begin{eqnarray}
\square^{n+m}=\square^{m}\times \square^n &
\xrightarrow{\sigma_{n,m}} & \square^n\times
\square^m=\square^{n+m}
\label{sigma2} \\
(y_1,\dots,y_m,x_1,\dots,x_n) &\mapsto &
(x_1,\dots,x_n,y_1,\dots,y_m). \nonumber
\end{eqnarray}

Recall that the product is described by the big diagram in \S
\ref{diagramgros}. In order to prove the commutativity, we change
the second and third row diagrams of this big diagram, by more
suitable diagrams. These changes do not modify the definition of
the product, but ease the study of the commutativity.

We define a complex $Z^{p}_{\A\times\A}(X,n)_0$ analogously to the
definition complex $\mmD_{\A\times \A}(X,p)_0$ (see $\S$\ref{AA}). Let
$$Z^p(X,n,m)_{0}:=Z^p(X,n+m)_{0},$$ and let
$\delta'=\sum_{i=1}^n (-1)^i\delta_i^0$ and
$\delta^{''}=\sum_{i=n+1}^{n+m} (-1)^{i-n}\delta_i^0.$
Then, $(Z^p(X,*,*)_{0},\delta',\delta^{''})$
is a $2$-iterated chain complex. For the sake of simplicity, we
denote both $\delta'$ and $\delta^{''}$ by $\delta$.

Denote by $Z^p_{\A\times \A}(X,*)_0$ the associated simple
complex. The complex $\mcH^{p}_{\A\times \A}(X,*)_0$ is defined
analogously.

Let $\whmZ^{p,q}_{\A\times \A}(X\times Y,*)_0$ be the
 diagram
  {\small $$ \xymatrix@C=-10pt{ &
\mcH^{p+q}_{\A\times \A}(X\times Y,*)_0 &&
\mmD_{\A\times \A}^{2(p+q)-*}
 (X\times Y,p+q)_0. \\ Z^{p+q}_{\A\times \A}(X\times Y,*)_0  \ar[ur]^{f_1}
 &&
 s(i_{X,Y}^{p,q})^{2(p+q)-*}_0 \ar[ur]^{\rho} \ar[ul]_{g_1}^{\sim} & &
 ZD^{2p+q}(X\times Y,p+q)_* \ar[ul]_{i}
  }$$ }
This diagram will fit in the second row of the new big diagram.
Denote by $\whZ^{p,q}_{\A\times \A}(X\times Y,*)_0$ the simple complex
associated to this diagram.

The third row of the new big diagram corresponds to a diagram
whose complexes are obtained from the refined normalized complex
of Definition \ref{normalized2}. The fact that, in these
complexes, most of the face maps vanish is the key point to
construct explicit homotopies for the commutativity of the
product. So, consider the following complexes:
\begin{itemize}
\item  Let $Z^q(X,*,*)_{00}$ be the $2$-iterated chain complex
with
$$Z^q(X,n,m)_{00}:=\bigcap_{i\neq 0,n+1} \ker \delta_i^0 \subset
Z^q(X,n+m)_{0}, $$
and with differentials $(\delta',\delta{''})=( -\delta_1^0,
-\delta_{n+1}^0)$. Denote by $Z^q_{\A\times \A}(X,*)_{00}$ the
associated simple complex.
\item
Let $\tau\mmD^*_{\log}(X\times \square^*\times
\square^*,p)_{00}$ be the $3$-iterated complex whose
$(r,-n,-m)$-graded piece is
$$\tau\mmD^r_{\log}(X\times \square^{n}\times \square^{m},p)_{00}=\bigcap_{i\neq 0,n+1} \ker \delta_i^0 \subset
\tau\mmD^r_{\log}(X\times \square^{n+m},p)_{0},$$
and with differentials $(d_{\mmD},-\delta_1^0,-\delta_{n+1}^0)$.
Let $\mmD^{*}_{\A\times\A}(X,p)_{00}$ be the associated simple
complex.
\item Let
$\tau\mmD^{*}_{\mathcal{Z}^{p,q}_{X,Y,*,*}}(X\times Y\times
\square^*\times \square^*,p+q)_{00}$ be the $3$-iterated complex
with
$$\tau\mmD^{r}_{\mathcal{Z}^{p,q}_{X,Y,n,m}}(X\times Y\times \square^n\times
\square^m,p+q)_{00}= \bigcap_{i\neq 0,n+1} \ker \delta_i^0$$
as a subset of $\tau\mmD^{r}_{\mathcal{Z}^{p,q}_{X,Y,n,m}}(X\times Y\times
\square^{n+m},p+q)_{0}$.
The differentials are given by $(d_{\mmD},-\delta_1^0,-\delta_{n+1}^0)$.
Let $\mmD^{*}_{\A\times\A,\mathcal{Z}^{p,q}_{X,Y}}(X\times
Y,p)_{00}$ be the associated simple complex.
\end{itemize}

\begin{obs}
Observe that there are induced morphisms
\begin{eqnarray*}
Z^{p+q}_{\A\times \A}(X\times Y,*)_{00} & \xrightarrow{f_1} &
\mcH^{p+q}_{\A\times \A}(X\times Y,*)_{00}, \\
\mmD^{2(p+q)-*}_{\A\times \A,\mathcal{Z}^{p,q}_{X,Y}}(X\times
Y,p+q)_{00} & \xrightarrow{g_1} & \mcH^{p+q}_{\A\times \A}(X\times
Y,*)_{00}, \\ \mmD^{2(p+q)-*}_{\A\times
\A,\mathcal{Z}^{p,q}_{X,Y}}(X\times Y,p+q)_{00} &
\xrightarrow{\rho} & \mmD_{\A\times \A}^{2(p+q)-*}
 (X\times Y,p+q)_{00}.
\end{eqnarray*}
\end{obs}

Let $\whmZ^{p,q}_{\A\times \A}(X\times Y,*)_{00}$ be the
 diagram
  {\footnotesize $$ \xymatrix@C=-15pt{ &
\mcH^{p+q}_{\A\times \A}(X\times Y,*)_{00} &&
\mmD_{\A\times \A}^{2(p+q)-*}
 (X\times Y,p+q)_{00}. \\ Z^{p+q}_{\A\times \A}(X\times Y,*)_{00}  \ar[ur]^{f_1}
 &&
\mmD^{2(p+q)-*}_{\A\times \A,\mathcal{Z}^{p,q}_{X,Y}}(X\times
Y,p+q)_{00} \ar[ur]^{\rho} \ar[ul]_{g_1}^{\sim} &
  & ZD^{2p+q}(X\times Y,p+q)_* \ar[ul]_{i}  } $$ }
This is the diagram fitting in the third row of the new diagram. Let $\whZ^{p,q}_{\A\times
\A}(X\times Y,*)_{00}$ be the simple complex associated to this
diagram.

\begin{lema}\label{hom} Let $X$ be an arithmetic variety over a
field.
\begin{enumerate}[(i)]
\item The natural chain morphisms
\begin{eqnarray}
 Z^q_{\A\times \A}(X,*)_{00} & \xrightarrow{i} & Z^q_{\A\times \A}(X,*)_{0},\label{kappa3}
\\
 Z^q_{\A\times
\A}(X,*)_{0} & \xrightarrow{\kappa} & Z^q(X,*)_{0}, \label{kappa1}
\end{eqnarray}
are quasi-isomorphisms. \item The natural cochain morphisms
\begin{eqnarray}
\mmD^{*}_{\A\times\A}(X,p)_{00} & \xrightarrow{i} &
\mmD^{*}_{\A\times\A}(X,p)_{0}, \\
\mmD^{*}_{\A\times \A,\mathcal{Z}^{p,q}_{X,Y}}(X\times Y,p+q)_{00}
& \xrightarrow{i} & \mmD^{*}_{\A\times
\A,\mathcal{Z}^{p,q}_{X,Y}}(X\times Y,p+q)_{0},
\\ \mmD^{*}_{\A\times\A}(X,p)_{0} & \xrightarrow{\kappa} &
\mmD^{*}_{\A}(X,p)_{0}, \label{kappa2}
\end{eqnarray}
are quasi-isomorphisms.
\end{enumerate}
\end{lema}
\begin{proof} The proofs of the facts that the morphisms $i$ are
quasi-isomorphisms are analogous for the three cases. For every
$n,m$, let $B(n,m)$ denote either $Z^p(X,n,m)$,
$\tau\mmD_{\log}^r(X\times \square^n\times \square^m,p)$ or
$\mmD_{\tau\log,\mmZ^{p,q}_{X,Y,n,m}}^r(X\times Y\times
\square^n\times \square^m,p+q)$, for some $r$. The groups
$B(n,m)_{0}$ and $B(n,m)_{00}$ are defined analogously.

Observe that for every $n,m$, $B({\cdot},m)$ and $B(n,{\cdot})$
are cubical abelian groups. We want to see that there is a
quasi-isomorphism
\begin{equation}\label{n2}
s(N_0^2N_0^1B(*,*))\xrightarrow{i} s(N^2N^1B(*,*)),
\end{equation}
 where superindex $1$ refers to the cubical structure given by the first
index $n$ and superindex $2$ to the cubical structure given by the
second index $m$.
An spectral sequence argument together with Lemma \ref{normalized3}
and  Proposition \ref{cubchain}
show that there is a quasi-isomorphism
$
s(N^2N_0^1B(*,*))\xrightarrow{\sim} s(N^2N^1B(*,*)).
$
By Lemma \ref{normalized3} and an spectral sequence argument again, we
obtain that
there is a quasi-isomorphism
$
s(N^2_0N_0^1B(*,*))\xrightarrow{i} s(N^2N^1_0B(*,*)).
$ Therefore, \eqref{n2} is a quasi-isomorphism.

The proofs of the facts that the morphisms in \eqref{kappa1} and
\eqref{kappa2} are quasi-isomorphisms are
analogous to each other. Therefore, we just prove the statement for
the morphism
\eqref{kappa1}. Consider the composition morphism
$$j: Z^q(X,m)_{0} \rightarrow Z^q(X,0,m)_{0} \rightarrow
Z^q_{\A\times \A}(X,m)_{0}.$$
The composition  of morphisms $Z^q(X,m)_{0} \xrightarrow{j}
Z^q_{\A\times \A}(X,m)_{0}\xrightarrow{\kappa} Z^q(X,m)_{0}$ is
the identity. Hence, it is enough to see that $j$ is a
quasi-isomorphism. Consider the 1st quadrant spectral sequence
with
$$E^1_{n,m}= H_m(Z^q(X,n,*)_0). $$
We will see that if $n\geq 1$, $E^1_{n,m}=0$.  By the homotopy
invariance of higher Chow groups, the map
$$f: Z^q(X\times \square^n,*)_0 \xrightarrow{\delta_1^1\cdots
  \delta_1^1}Z^q(X,*)_0  $$
is a quasi-isomorphism. By Proposition \ref{cubchain}, it induces
a quasi-isomorphism
$$f: Z^q(X\times \square^n,*)_0= NZ^q(X\times \square^n,*)_0
\rightarrow NZ^q(X,*)_0  $$ where the cubical structure on
$Z^q(X,*)_0$ is the trivial one. Since for a trivial cubical
abelian group $NZ^q(X,*)_0 =0$, we see that
$$ H_m(Z^q(X,n,*)_0)=0,\qquad n>0,$$
and hence
$$E^1_{n,m}=\left\{ \begin{array}{ll} 0 & \textrm{ if }n>0, \\
    CH^q(X,m) & \textrm{ if }n=0. \end{array}
\right.$$
 \end{proof}

It follows from the lemma that the product on the higher
arithmetic Chow groups is also represented by the following diagram of
complexes
 {\footnotesize $$\xymatrix@C=-35pt@R=25pt{ &
\mcH^{p}(X,n)_0 \otimes \mcH^{q}(Y,m)_0 \ar[dd]^{\times} & &
\mmD_{\A}^{2p-n}(X,p)_0\otimes
\mmD_{\A}^{2q-m}(Y,q)_0 \ar[dd]^{\bullet_{\A}} \\
Z^{p}(X,n)_0\otimes Z^q(Y,m)_0 \ar[dd]^{\times} \ar[ur]^{f_1} & &
\ar[ul]_{g_1}^{\sim} \mmD^{2p-n}_{\A,\mathcal{Z}^p}(X,p)_0\otimes
\mmD^{2q-m}_{\A,\mathcal{Z}^q}(Y,q)_0 \ar[ur]^{\rho}
\ar[dd]^{\bullet_{p,q}} & & ZD^{2p}(X,p)_n\otimes ZD^{2q}(X,q)_m
\ar[ul]_{i} \ar[dd]^{\wedge}
 \\
 & \mcH^{p+q}_{\A\times \A}(X\times Y,n+m)_0 & & \mmD_{\A\times \A}^{2(p+q)-n-m}
 (X\times Y,p+q)_0  \\
Z^{p+q}_{\A\times \A}(X\times Y,n+m)_0  \ar[ur]^{f_1} & &
\ar[ul]_{g_1} s(i_{X,Y}^{p,q})^{2p+2q-n-m}_0 \ar[ur]^{\rho} & &
ZD^{2(p+q)}(X\times Y,p+q)_{n+m}\ar[ul]_{i} \ar@{=}[dd]
 \\
 & \mcH^{p+q}_{\A\times \A}(X\times Y,n+m)_{00} \ar@{=}[dd]
 \ar[uu]_{i}^{\sim} & & \mmD_{\A\times \A}^{2(p+q)-n-m}(X\times Y,p+q)_{00}
 \ar@{=}[dd]  \ar[uu]_{i}^{\sim}\\
Z^{p+q}_{\A\times \A}(X\times Y,n+m)_{00} \ar@{=}[dd]
\ar[uu]_{i}^{\sim} \ar[ur]^{f_1} & & \ar[ul]_{g_1}
\mmD^{2(p+q)-n-m}_{\A\times \A,\mathcal{Z}^{p,q}_{X,Y}}(X\times
Y,p+q)_{00} \ar[ur]^{\rho} \ar[dd]^{\eta}
\ar@{->}[uu]_{i}^{\sim} & & ZD^{2(p+q)}(X\times
Y,p+q)_{n+m}\ar[ul]_{i} \ar@{=}[dd]
 \\
 &\mcH^{p+q}_{\A\times \A}(X\times Y,n+m)_{00} \ar[dd]^{\kappa}
    & & \mmD_{\A\times \A}^{2(p+q)-n-m}(X\times Y,p+q)_{00}  \ar[dd]^{\kappa}  \\
Z^{p+q}_{\A\times \A}(X\times Y,n+m)_{00} \ar[dd]^{\kappa}
  \ar[ur]^{f_1} & & \ar[ul]_{g_1}^{\sim}
\mmD^{2(p+q)-n-m}_{\A\times \A,\mathcal{Z}^{p+q}}(X\times
Y,p+q)_{00} \ar[ur]^{\rho} \ar[dd]^{\kappa} & & ZD^{2(p+q)}(X\times
Y,p+q)_{n+m}\ar[ul]_{i} \ar@{=}[dd]
 \\
 & \mcH^{p+q}(X\times Y,n+m)_0 & & \mmD_{\A}^{2(p+q)-n-m}(X\times Y,p+q)_0 \\
Z^{p+q}(X\times Y,n+m)_0 \ar[ur]^{f_1} & & \ar[ul]_{g_1}^{\sim}
\mmD^{2(p+q)-n-m}_{\A,\mathcal{Z}^{p+q}}(X\times Y,p+q)_0
\ar[ur]^{\rho} & & ZD^{2(p+q)}(X\times Y,p+q)_{n+m}\ar[ul]_{i} } $$}

In the first set of vertical arrows of this diagram is where the
product is defined. In the second set of vertical rows we invert the
quasi-isomorphisms that relate the normalized complex and the refined
normalized complex. Moreover, we also invert the  quasi-isomorphism
analogous to \eqref{eq:6}. In the third set of vertical arrows we just
consider the change of supports $\mathcal{Z}^{p,q}_{X,Y}\subset
\mathcal{Z}^{p+q}$. We will denote the map induced by this change of
support by $\eta$. Finally in the last set of vertical arrows we
apply the morphisms $\kappa$ induced by the identification
$\square^{n}\times \square^{m}=\square^{n+m}$.

Let $\whZ^{p+q}_{\A\times \A}(X\times Y,*)_{00}$ denote the simple
of the diagram of the fourth row. Hence, in the derived category
of complexes, this product is described by the composition
 $$\xymatrix{ s(\whmZ^p(X,*)_0\otimes \whmZ^q(Y,*)_0) \ar[r]^(0.54){\cup} &
 \whZ^{p,q}_{\A\times \A}(X\times Y,*)_{0} \\  &
\whZ^{p,q}_{\A\times \A}(X\times Y,*)_{00}\ar@{-->}[u]_{i}
\ar[r]^(0.5){\eta} & \whZ^{p+q}_{\A\times \A}(X\times Y,*)_{00}
\ar[d]^{\kappa}\\
  & & \whZ^{p+q}(X\times Y,*)_0. } $$
Note that the difference between the complexes  $\whZ^{p,q}_{\A\times \A}(X\times Y,*)_{00}$  and  $\whZ^{p+q}_{\A\times \A}(X\times Y,*)_{00}$ lies on the change of supports $\mathcal{Z}^{p,q}_{X,Y}\subset
\mathcal{Z}^{p+q}$. This is indicated by either two codimension superindices $p,q$ in the first one or a unique codimension superindex $p+q$ in the second.

We next use this description of the product in the higher arithmetic Chow
groups in order to prove
its commutativity.

Recall that the map $\sigma_{n,m}$ is defined by
\begin{eqnarray*}
\square^{n+m}=\square^{m}\times \square^n &
\xrightarrow{\sigma_{n,m}} & \square^n\times
\square^m=\square^{n+m}
 \\
(y_1,\dots,y_m,x_1,\dots,x_n) &\mapsto &
(x_1,\dots,x_n,y_1,\dots,y_m). \nonumber
\end{eqnarray*}
Let $$\sigma_{X,Y,n,m}: Y\times X \times \square^m\times \square^n
\rightarrow X \times Y \times \square^n\times \square^m
$$ be the map $\sigma_{X,Y}\times \sigma_{n,m}$.

 We define a morphism of diagrams $$
\whmZ^{p,q}_{\A\times \A}(X\times Y,*)_{0}
  \xrightarrow{\sigma_{X,Y,\square}^*}   \whmZ^{q,p}_{\A\times \A}(Y\times X,*)_0$$
as follows:
\begin{itemize}
\item Let $\sigma_{X,Y,\square}^*:Z^{p+q}_{\A\times \A}(X\times
Y,*)_{0}\rightarrow Z^{p+q}_{\A\times \A}(Y\times X,*)_{0}$ be the
map sending
$$Z\in Z^{p+q}(X\times Y,n,m)_{0}\ \text{to}\
(-1)^{nm}\sigma_{X,Y,n,m}^*(Z)\in Z^{p+q}(Y\times X,m,n)_{0}.$$
The morphism $\sigma_{X,Y,\square}^*:\mcH^{p+q}_{\A\times
\A}(X\times Y,*)_{0}\rightarrow \mcH^{p+q}_{\A\times \A}(Y\times
X,*)_{0}$ is defined analogously.
\item Let
$\sigma_{X,Y,\square}^*:  \mmD^{*}_{\A\times \A}(X\times Y,p+q)_0
\rightarrow \mmD^{*}_{\A\times \A}(Y\times X,p+q)_0$
 be the map that, at the $(*,-n,-m)$ component, is
{\small $$(-1)^{nm}\sigma_{X,Y,n,m}^*:\tau\mmD^{*}_{\log}(X\times
Y\times \square^{n}\times \square^{m},p+q)_0 \rightarrow
\tau\mmD^{*}_{\log}(Y\times X\times \square^{m}\times
\square^{n},p+q)_0. $$} Observe that it is a cochain morphism.
\item We define analogously the morphism
$\sigma_{X,Y,\square}^*: s(i_{X,Y}^{p,q})^{*}_0\rightarrow s(i_{Y,X}^{q,p})^{*}_0. $
\end{itemize}
These morphisms commute with the morphisms
$f_1,g_1$ and $\rho$. Hence, they induce a morphism of diagrams
and therefore a morphism on the associated simple complexes:
$$ \whZ^{p,q}_{\A\times \A}(X\times Y,*)_{0}
  \xrightarrow{\sigma_{X,Y,\square}^*}   \whZ^{q,p}_{\A\times \A}(Y\times X,*)_0. $$

Note that the morphism $\sigma_{X,Y,\square}^*$ restricts to $
\whZ^{p,q}_{\A\times \A}(X\times Y,*)_{00}$ and to
$\whZ^{p+q}_{\A\times \A}(X\times Y,*)_{00}$.

\begin{lema}\label{lemmacommut1}
The following diagram is  commutative:
$$\xymatrix{\whZ^{p,q}_{\A\times \A}(X\times
Y,*)_{0} \ar[d]_{\sigma_{X,Y,\square}^*} & \whZ^{p,q}_{\A\times
\A}(X\times Y,*)_{00} \ar[l]_{i} \ar[r]^{\eta}
\ar[d]_{\sigma_{X,Y,\square}^*}  & \whZ^{p+q}_{\A\times
\A}(X\times Y,*)_{00} \ar[d]_{\sigma_{X,Y,\square}^*}
 \\
 \whZ^{q,p}_{\A\times \A}(Y\times X,*)_0 &
 \whZ^{q,p}_{\A\times \A}(Y\times X,*)_{00} \ar[l]_{i} \ar[r]^{\eta} &
 \whZ^{p+q}_{\A\times \A}(Y\times X,*)_{00}. } $$
\end{lema}
\begin{proof}
The statement follows from the definitions.
\end{proof}

\begin{lema}\label{lemmacommut2} The following diagram is commutative
$$\xymatrix { s(\whmZ^p(X,*)_0\otimes \whmZ^q(Y,*)_0)
\ar[d]_{\sigma} \ar[r]^(0.55){\cup} & \whZ^{p,q}_{\A\times
\A}(X\times Y,*)_{0} \ar[d]^{\sigma_{X,Y,\square}^*}
 \\ s(\whmZ^q(Y,*)_0\otimes \whmZ^p(X,*)_0) \ar[r]^(0.55){\cup} &
 \whZ^{q,p}_{\A\times \A}(Y\times X,*)_0 . } $$
\end{lema}
\begin{proof}
It follows from the definition that the morphism
$\sigma_{X,Y,\square}^*$ commutes with the product $\times$ in
$Z^*(X,*)_0$ and in $\mcH^*(X,*)_0$. The fact that it commutes with
$\bullet_{\A}$ and $\bullet_{p,q}$ is an easy computation.
\end{proof}

By Lemmas \ref{lemmacommut1} and \ref{lemmacommut2}, we are left
to see that the diagram
\begin{equation}\label{diagram4}
\xymatrix{  \whZ^{p+q}_{\A\times \A}(X\times Y,*)_{00}
\ar[d]_{\sigma_{X,Y,\square}^*} \ar[r]^{\kappa} &
\whZ^{p+q}(X\times Y,*)_{0} \ar[d]^{\sigma_{X,Y}^*}
 \\
 \whZ^{p+q}_{\A\times \A}(Y\times X,*)_{00}\ar[r]^{\kappa} & \whZ^{p+q}(Y\times X,*)_{0} }
 \end{equation}
is commutative up to homotopy. We follow the ideas used by Levine,
in \cite{Levine1}, $\S 4$, in order to prove the commutativity of
the product on the higher algebraic Chow groups. We will end up
with an explicit homotopy for the commutativity of diagram
\ref{diagram4}.

\begin{obs}
For any scheme  $X$,  consider the morphism
$$  \whZ^{p}_{\A\times
\A}(X,*)_{00} \xrightarrow{\sigma_{\square}^*}
 \whZ^{p}_{\A\times \A}(X,*)_{00}$$
  induced  by
$(-1)^{nm}\sigma_{n,m}^*$ at each component. Then,
$\sigma_{X,Y,\square}^*=\sigma_{X,Y}^*\sigma_{\square}^*$ and
hence, the commutativity of the diagram \eqref{diagram4} will follow
from the commutativity (up to homotopy) of the diagram
$$  \xymatrix@R=4pt{ \whZ^{p}_{\A\times
\A}(X,*)_{00} \ar[dd]_{\sigma_{\square}^*} \ar[drr]^{\kappa} \\ &&
\whZ^{p}(X,*)_{0}.  \\
 \whZ^{p}_{\A\times
\A}(X,*)_{00} \ar[urr]_{\kappa} }  $$
\end{obs}

Let $W_n$ be the closed subvariety of $\square^{n+1}\times \P^1$
defined by the equation
\begin{equation}
t_1(1-x_1)(1-x_{n+1})=t_1-t_0,
\end{equation}
where $(t_0:t_1)$ are  the coordinates in $\P^1$ and
$(x_1,\dots,x_{n+1})$ are the coordinates in $\square^{n+1}$. Recall
that we have identified $\square^{1}$ with the subset $t_{0}\not =
t_{1}$ of $\P^{1}$, with coordinate $x=t_{0}/t_{1}$.
Then, there is an isomorphism $W_n \cong \square^{n}\times
\square^1$. The inverse of this isomorphism is  given by
\begin{eqnarray*}
\square^{n+1}& \xrightarrow{\varphi_n} & W_n \\
(x_1,\dots,x_{n+1}) & \mapsto &
(x_1,\dots,x_{n+1},x_1+x_{n+1}-x_1x_{n+1}).
\end{eqnarray*}
Consider the projection
$$\pi_n:W_n \rightarrow \square^n, \qquad (x_1,\dots,x_{n+1},t) \mapsto (x_2,\dots,x_{n},t). $$
Let $\tau$ be the permutation
$$
\square^{n}  \xrightarrow{\tau}  \square^n, \qquad (x_1,\dots,x_n)
\mapsto  (x_2,\dots,x_n,x_1).$$

\begin{obs}
Let $\sigma_{n,m}$ be the map defined in \eqref{sigma2}. Observe
that it is decomposed as $ \sigma_{n,m} = \tau \circ \stackrel{m}{\dots} \circ \tau.$
Therefore,
$\sigma_{n,m}^*=\tau^*\circ \stackrel{m}{\dots} \circ \tau^*. $
\end{obs}

It is easy to check that the following identities are satisfied:
\begin{eqnarray} \label{homotdelta}
\pi_n\varphi_n\delta^i_0 & = & \left\{ \begin{array}{ll} id & \textrm{if }i=1, \\
\delta^{i-1}_0\pi_{n-1}\varphi_{n-1} & \textrm{if }i=2,\dots,n, \\
\tau & \textrm{if }i=n+1.
\end{array}\right. \\
\pi_n\varphi_n\delta^i_1 & = & \left\{ \begin{array}{ll} \delta^n_1\sigma ^{n} & \textrm{if }i=1, \\
\delta^{i-1}_1\pi_{n-1}\varphi_{n-1} & \textrm{if }i=2,\dots,n, \\
\delta_1^n\sigma ^{n} \tau & \textrm{if }i=n+1.
\end{array}\right. \nonumber
\end{eqnarray}

Let $W_n^X$ be the pull-back of $W_n$ to $X\times \square^n$.
Then, the maps
$$\pi_n: W^X_n\rightarrow X\times \square^n,\quad \textrm{and}\quad
\varphi_n:X\times \square^{n+1}\rightarrow W_n^X$$ are defined
accordingly.

\begin{prop}\label{commut2} Let $X$ be a quasi-projective regular scheme over a field $k$.
\begin{enumerate}[(i)] \item The scheme $W_n$ is a flat regular scheme over $\square^n$.
\item There is a well-defined map
$$Z^p(X,n)  \xrightarrow{h_n}  Z^p(X,n+1), \qquad Y  \mapsto   \varphi_n^*\pi_n^*(Y).
$$
\end{enumerate}
\end{prop}

\begin{proof}
See  \cite{Levine1},  Lemma 4.1.
\end{proof}

For every $n\geq 1$, we define the morphisms
\begin{eqnarray*}
\mcH^{p}(X,n) &\xrightarrow{h_n} & \mcH^p(X,n+1), \\
\tau\mmD^*_{\log}(X\times \square^n,p) &\xrightarrow{h_n} &
\tau\mmD^*_{\log}(X\times \square^{n+1},p), \\
\tau\mmD^*_{\log,\mmZ^p}(X\times \square^n,p) &\xrightarrow{h_n} &
\tau\mmD^*_{\log,\mmZ^p}(X\times \square^{n+1},p),
\end{eqnarray*}
by $h_n=\varphi_n^*\pi_n^*. $ By Proposition \ref{commut2}, (ii),
these morphisms are well defined.

\begin{lema}
Let $\alpha$ be an element of $Z^q(X,n)_{0}$, $\mcH^{p}(X,n)_0$,
$\tau\mmD^*_{\log,\mmZ^p}(X\times \square^n,p)_0$ or $\tau\mmD^*_{\log}(X\times \square^n,p)_0$. Then, the following
equality is satisfied
$$  \delta h_{n}(\alpha)+ \sum_{i=1}^{n-1}(-1)^{i} h_{n-1}\delta_{i}^0(\alpha)=
-\alpha +(-1)^{n+1}\tau^*(\alpha). $$
\end{lema}
\begin{proof}  By hypothesis,
$\delta_i^1(\alpha)=0$ for all $i=1,\dots,n$. Then, by the
pull-back of the equalities \eqref{homotdelta}, we see that
$\delta_i^1\varphi_n^*\pi_n^*(\alpha)=0$. Therefore, using
\eqref{homotdelta},
\begin{eqnarray*}
\delta h_n(\alpha) &=& \sum_{i=1}^{n+1}\sum_{j=0,1}(-1)^{i+j}
\delta_i^j \varphi_n^*\pi_n^*(\alpha) =\sum_{i=1}^{n+1}(-1)^{i}
\delta_i^0 \varphi_n^*\pi_n^*(\alpha)
\\ &=& -\alpha + \sum_{i=2}^{n}(-1)^{i}  \varphi_n^*\pi_{n-1}^*\delta_{i-1}^0(\alpha)
+(-1)^{n-1}\tau^*(\alpha)\\
&=& -\alpha - \sum_{i=1}^{n-1}(-1)^{i} h_{n-1}\delta_{i}^0(\alpha)
+(-1)^{n+1}\tau^*(\alpha),
\end{eqnarray*}
as desired.
\end{proof}

\begin{prop}\label{homotopy4} Let $X$ be an arithmetic variety over a
  field. Then the following diagram is commutative up to homotopy.
{\small $$ \xymatrix@R=3pt{\widehat{Z}^p_{\A\times
\A}(X,n)_{00}
\ar[drr]^{\kappa} \ar[dd]_{\sigma_{\square}^*} \\ && \widehat{Z}^p(X,n)_{0}. \\
\widehat{Z}^p_{\A\times \A}(X,n)_{00} \ar[urr]_{\kappa} } $$ }
\end{prop}
\begin{proof}
We start by defining maps
\begin{eqnarray*}
Z^{p}(X,n,m)_{00} & \xrightarrow{H_{n,m}} & Z^p(X,n+m+1)_0,
\\ \mcH^p(X,n,m)_{00} & \xrightarrow{H_{n,m}} & \mcH^p(X,n+m+1)_0, \\ \tau\mmD^{*}_{\log}(X\times
\square^n\times \square^m,p)_{00} & \xrightarrow{H_{n,m}} &
\tau\mmD^{*}_{\log}(X\times \square^{n+m+1},p)_{0}, \\
\tau\mmD^*_{\log,\mmZ^p}(X\times \square^n\times \square^m,p)_{00} &
\xrightarrow{H_{n,m}} & \tau\mmD^{*}_{\log,\mmZ^p}(X\times
\square^{n+m+1},p)_{0}.
\end{eqnarray*}
By construction, these maps will commute with $f_1,g_1$ and
$\rho$. This will allow us to define the homotopy for the
commutativity of the diagram in the statement.

All the maps $H_{n,m}$ will be defined in the same way. Thus, let
$B(X,n,m)_{00}$ denote either $Z^{p}(X,n,m)_{00}$,
$\mcH^p(X,n,m)_{00}$, $\tau\mmD^{*}_{\log}(X\times \square^n\times
\square^m,p)_{00},$ or $\tau\mmD^*_{\log,\mmZ^p}(X\times
\square^n\times \square^m,p)_{00} .$ For the last two cases,
$B(X,n,m)_{00}$ is a cochain complex, while for the first two
cases, it is  a group. Analogously, denote by $B(X,n+m+1)_{0}$ the
groups/complexes that are the target of $H_{n,m}$. The map
$H_{n,m}$ will be a cochain complex for the last two cases.

Let $\alpha \in B(X,n,m)_{00}$. Then, let $H_{n,m}(\alpha)\in
B(X,n+m+1)_0$ be defined by
\begin{equation}
H_{n,m}(\alpha)=\left\{ \begin{array} {ll}
\sum_{i=0}^{n-1}(-1)^{(m+i)(n+m-1)}
h_{n+m+1}((\tau^*)^{m+i}(\alpha)), & n\neq 0, \\ 0 & n=0.
\end{array} \right.
\end{equation}

From the definition it follows that:
\begin{enumerate*}[$\rhd$]
\item If $B(X,n,m)_{00}$ is $\tau\mmD^{*}_{\log}(X\times
\square^n\times \square^m,p)_{00},$ or
$\tau\mmD^*_{\log,\mmZ^p}(X\times \square^n\times \square^m,p)_{00}$,
 then
$$ d_{\mmD}H_{n,m}(\alpha)=H_{n,m}d_{\mmD}(\alpha),$$ i.e. $H_{n,m}$ is a
cochain morphism. \item $f_1H_{n,m}=H_{n,m}f_1$,
$g_1H_{n,m}=H_{n,m}g_1$ and $\rho H_{n,m}=H_{n,m}\rho$.
\end{enumerate*}

Recall that in all these complexes,
\begin{eqnarray*}
\delta'(\alpha)& = & -\delta_1^0(\alpha)\in B(X,n-1,m)_{00},
\\ \delta''(\alpha)&=& -\delta_{n+1}^0(\alpha)\in B(X,n,m-1)_{00}. \end{eqnarray*}

\begin{lema}\label{H}
For every $\alpha\in B(X,n,m)_{00}$  we have
$$\delta H_{n,m}(\alpha)  - H_{n-1,m}\delta_1^0(\alpha)-(-1)^n H_{n,m-1}\delta_{n+1}^0(\alpha)=
\alpha-(-1)^{nm}\sigma_{n,m}^*(\alpha). $$
\end{lema}
\begin{proof}
If $n=0$, since $\alpha =\sigma_{0,m}(\alpha)$ and
$H_{0,m}(\alpha)=0$  the equality is satisfied. For simplicity,
for every $i=0,\dots,n-1$, we denote
$$
H_{n,m}^i(\alpha)=(-1)^{(m+i)(n+m-1)}
h_{n+m+1}((\tau^*)^{m+i}(\alpha))\in B(X,n+m+1)_{0}.
$$

An easy computation shows that
$$ \delta_j^0\tau^*(\alpha)= \left\{\begin{array}{ll}
\tau^* \delta_{j-1}^0(\alpha) & \textrm{ if }j\neq 1, \\
\delta_n^0(\alpha) & \textrm{ if }j=1,
\end{array} \right.$$
and hence,
$$ \delta_j^0(\tau^*)^{i}(\alpha)= \left\{\begin{array}{ll}
(\tau^*)^{i} \delta_{j-i}^0(\alpha) & \textrm{ if }j > i, \\
(\tau^*)^{i-1}\delta_n^0(\alpha)  & \textrm{ if }j=i, \\
(\tau^*)^{i-1}\delta_{n-i+j}^0(\alpha) & \textrm{ if }j<i.
\end{array} \right.$$
Therefore,
\begin{eqnarray*}
\delta H_{n,m}^i(\alpha) &=&
\sum_{j=1}^{n+m+1}(-1)^{j+(m+i)(n+m-1)}\delta_{j}^0h_{n+m+1}( (\tau^*)^{m+i}(\alpha))\\
&=&  (-1)^{1+(m+i)(n+m-1)}(\tau^*)^{m+i}(\alpha) +(-1)^{(m+i+1)(n+m-1)} (\tau^*)^{m+i+1}(\alpha)\\
&&+ \sum_{j=2}^{n+m}(-1)^{j+(m+i)(n+m-1)}
h_{n+m}(\delta_{j-1}^0(\tau^*)^{m+i}(\alpha)).
\end{eqnarray*}

Recall that the only non-zero faces of $\alpha$ are $\delta_1^0$
and $\delta_{n+1}^0$. Therefore,
from the equalities \eqref{homotdelta}, we see that the only
non-zero faces are the faces corresponding to the indices
$j=m+i+2$ and $j=i+2$. In these cases, they take the values
$(\tau^*)^{m+i}\delta_1^0$ and $(\tau^*)^{m+i-1}\delta_{n+1}^0$
respectively. Therefore, if $i\neq n-1$, we obtain
\begin{eqnarray*}
\delta H_{n,m}^i(\alpha) &=&
-(-1)^{(m+i)(n+m-1)}(\tau^*)^{m+i}(Z)\\ && +(-1)^{(m+i+1)(n+m-1)}
(\tau^*)^{m+i+1}(\alpha)
\\ &&+ (-1)^{(m+i)(n+m-2)} h_{n+m}((\tau^*)^{m+i}\delta_1^0(\alpha))\\ &&
+(-1)^{i+(m+i)(n+m-1)}h_{n+m}((\tau^*)^{m-1+i}\delta_{n+1}^0(\alpha)).
\end{eqnarray*}
Observe that $(-1)^{i+(m+i)(n+m-1)}=(-1)^{(m+i-1)(n+m)+n}$.
Therefore, the last summand in the previous equality is exactly
$$H_{n-1,m}^i(\delta_1^0(\alpha))+(-1)^n H_{n,m-1}^i(
\delta_{n+1}^0(\alpha)).$$ If $i=n-1$, then
$\delta_{j-1}^0(\tau^*)^{m+i}(\alpha)=0$, for $j=2,\dots,n-m$.
Therefore,
\begin{eqnarray*} \delta H_{n,m}^{n-1}(\alpha)& =&
(-1)^{1+(m+n-1)(n+m-1)}(\tau^*)^{m+n-1}(\alpha)\\ && +(-1)^{(m+n)(n+m-1)} (\tau^*)^{m+n}(\alpha)\\
&&
+(-1)^{n-1+(m+n-1)(n+m-1)}h_{n+m}((\tau^*)^{m-1+i}\delta_{n+1}^0(\alpha))
\\ &=& -(-1)^{(m+n-1)(n+m-1)}(\tau^*)^{m+n-1}(\alpha)+\alpha \\ &&
+(-1)^{n+(m+n-2)(n+m)}h_{n+m}((\tau^*)^{m-1+i}\delta_{n+1}^0(\alpha)).
\end{eqnarray*}

Finally, we have seen that
\begin{eqnarray*}
\delta H_{n,m}(\alpha) &=& -(-1)^{m(n+m-1)}(\tau^*)^{m}(\alpha) +
\sum_{i=0}^{n-2} H_{n-1,m}^i(\delta_1^0(\alpha))\\ && +
\sum_{i=0}^{n-1}(-1)^n H_{n,m-1}^i(\delta_{n+1}^0(\alpha))
+\alpha,
\end{eqnarray*}
and since $(-1)^{m(n+m-1)}=(-1)^{nm}$, we obtain the equality
 $$\delta H_{n,m}(\alpha)
 - H_{n-1,m}(\delta_1^0(\alpha))-(-1)^n H_{n,m-1}(\delta_{n+1}^0(\alpha))=\alpha- (-1)^{nm}\sigma_{n,m}^*(\alpha).$$
\end{proof}

Let
$$ Z^{p}_{\A\times \A}(X,*)_{00} \xrightarrow{H} Z^p(X,*+1)_0,
\qquad
 \mcH^p_{\A\times \A}(X,*)_{00}  \xrightarrow{H}  \mcH^p(X,*+1)_0, $$
be the maps which are $H_{n,m}$ on the $(n,m)$-component.
Let
$$\mmD^{2p-*}_{\A\times \A,\mmZ^p}(X,p)_{00}  \xrightarrow{H} \mmD^{2p-*-1}_{\A,\mmZ^p}(X,p)_{0},$$
be the maps which are $(-1)^r H_{n,m}$ on the $(r,-n,-m)$-component. Observe that now
$$d_{\mmD}H = -H d_{\mmD}. $$

Let
$$H: \whZ^p_{\A \times \A}(X,n)_{00} \rightarrow \whZ^p(X,n+1)_{0}$$
be defined  by
$$H(Z,\alpha_0,\alpha_1,\alpha_2,\alpha_3)=(H(Z),H(\alpha_0),\alpha_1,-H(\alpha_2),-H(\alpha_3)).$$
Let $x=(Z,\alpha_0,\alpha_1,\alpha_2,\alpha_3)\in \widehat{Z}^p_{\A\times
\A}(X,n)_{00}$. Then,
{\small
\begin{align*}
dH(x)& =(\delta H(Z),d_s H(\alpha_0),d_{\mmD}(\alpha_1),f_1H(Z) - g_1H(\alpha_0)+
\delta
H(\alpha_2),\rho H(\alpha_0)+d_{s}H(\alpha_3)-\alpha_1)\\
H d(x) & =
(H\delta(Z),Hd_s(\alpha_0),d_{\mmD}(\alpha_1),-Hf_1(Z)+Hg_1(\alpha_0)+H\delta(\alpha_2),\\ & \qquad -H\rho(\alpha_0)+Hd_{s}(\alpha_3)+H(\alpha_1)).
\end{align*}}
Observe that for $\alpha_0\in \tau\mmD^{r}_{\log,\mmZ^p}(X\times
\square^n\times \square^m,p)_{00}$, we have
\begin{eqnarray*}
H d_s(\alpha_0)& = &  Hd_{\mmD}(\alpha_0)+(-1)^r
H\delta(\alpha_0)= -d_{\mmD}H(\alpha_0) + (-1)^r
H\delta(\alpha_0), \\
d_s H(\alpha_0) &=& d_{\mmD}H(\alpha_0)+ (-1)^r \delta
H(\alpha_0).
\end{eqnarray*}
The same remark applies to $\alpha_3\in \tau\mmD^{r}_{\log}(X\times
\square^n\times \square^m,p)_{00}$.
Moreover, since $\alpha_1$ equals zero in all degrees but $0$ and $H$ is the identity in degree zero, we have, by Lemma \ref{H},
$$dH(x)+Hd(x) = x-\sigma_{\square}^*(x). $$
\end{proof}

\begin{cor} The following diagram is commutative up to homotopy
$$\xymatrix{  \whZ^{p+q}_{\A\times
\A}(X\times Y,*)_{00} \ar[d]_{\sigma_{X,Y,\square}^*}
\ar[r]^{\kappa} & \whZ^{p+q}(X\times Y,*)_{0}
\ar[d]^{\sigma_{X,Y}^*}
 \\
 \whZ^{p+q}_{\A\times \A}(Y\times X,*)_{00}\ar[r]^{\kappa} & \whZ^{p+q}(Y\times X,*)_{0} } $$
\end{cor}
\begin{proof}
It follows from Proposition \ref{homotopy4}.
\end{proof}

\begin{cor} Let $X,Y$ be arithmetic varieties.
\begin{enumerate}[(i)]
\item Under the canonical isomorphism $X\times Y\cong Y\times X$,
  the pairing
  $$ \widehat{CH}^p(X,n) \otimes \widehat{CH}^q(Y,m) \xrightarrow{\cup} \widehat{CH}^{p+q}(X\times Y,n+m),$$
  is graded commutative with respect to the degree $n$.
\item The internal pairing
  $$ \widehat{CH}^p(X,n) \otimes \widehat{CH}^q(X,m) \xrightarrow{\cup} \widehat{CH}^{p+q}(X,n+m),$$
  is graded commutative with respect to the degree $n$.
\end{enumerate}
\end{cor}

\subsection{Associativity} \label{assoc4} We prove here
that the product for the higher arithmetic Chow groups is
associative. First of all, observe that the product on
$Z^*(X,*)_0$ is strictly associative. Hence, all that remains is
 to study the associativity of the product in the complexes
with differential forms, except for $ZD^{2p}(X,p)_*$, where it is already associative. The key point will be Proposition
\ref{commutassoc}.

Denote by $h$ the homotopy for the associativity of the product in
the Deligne complex of differential forms of Proposition
\ref{commutassoc}. Let $X,Y,Z$ be complex algebraic manifolds.
Then, the external product $\bullet_{\A}$ is associative, in the
sense that there is a commutative diagram up to homotopy: {\small
\begin{equation}\label{assoc1} \xymatrix@C=-35pt{ & \mmD^{r}_{\A}(X,p)_0\otimes
\mmD^{s}_{\A}(Y,q)_0\otimes \mmD^{t}_{\A}(Z,l)_0
\ar[dl]_{\bullet_{\A}\otimes id }
\ar[dr]^{id\otimes \bullet_{\A}} & \\
\mmD^{r+s}_{\A}(X\times Y,p+q)_0 \otimes \mmD^{t}_{\A}(Z,l)_0
\ar[dr]_{\bullet_{\A}} & & \mmD^{r}_{\A}(X,p)_0 \otimes
\mmD^{s+t}_{\A}(Y\times
Z,q+l)_0 \ar[dl]^{\bullet_{\A}} \\
& \mmD^{r+s+t}_{\A}(X\times Y \times Z,p+q+l)_0 } \end{equation} }
This follows from the fact that the homotopy  $h$ is functorial
(see \cite{Burgos2}).
 
\begin{prop}\label{assoc2}
Let $X,Y,Z$ be complex algebraic manifolds. Then, there is a
commutative diagram, up to homotopy: {\small
$$ \xymatrix@C=-80pt{ & \mmD^{r}_{\A,\mmZ^p}(X,p)_0\otimes
\mmD^{s}_{\A,\mmZ^q}(Y,q)_0\otimes \mmD^{t}_{\A,\mmZ^l}(Z,l)_0
\ar[dl]_{\bullet_{\A}\otimes id }
\ar[dr]^{id\otimes \bullet_{\A}} & \\
\mmD^{r+s}_{\A,\mmZ^{p+q}}(X\times Y,p+q)_0 \otimes
\mmD^{t}_{\A,\mmZ^l}(Z,l)_0 \ar[dr]_{\bullet_{\A}} & &
\mmD^{r}_{\A,\mmZ^p}(X,p)_0 \otimes
\mmD^{s+t}_{\A,\mmZ^{q+l}}(Y\times
Z,q+l)_0 \ar[dl]^{\bullet_{\A}} \\
& \mmD^{r+s+t}_{\A,\mmZ^{p+q+l}}(X\times Y \times Z,p+q+l)_0 } $$}
\end{prop}
\begin{proof}
In order to prove the proposition, we need to introduce some new
complexes, which are analogous to $s(i_{X,Y}^{p,q})^*$, but with
the three varieties $X,Y,Z$. Due to the similarity, we will
leave the details to the reader.

We write $\square^{n,m,d}_{X,Y,Z}=X\times Y\times Z\times
\square^{n+m+d}$. Let {\small $$A^*=
\mmD_{\log}^*(\square^{n,m,d}_{X,Y,Z}\setminus \mmZ^{p}_{X,n},k)
\oplus \mmD_{\log}^*(\square^{n,m,d}_{X,Y,Z}\setminus
\mmZ^{q}_{Y,m},k) \oplus
\mmD_{\log}^*(\square^{n,m,d}_{X,Y,Z}\setminus \mmZ^{l}_{Z,d},k),
$$} and {\small $$B^*=\mmD_{\log}^*(\square^{n,m,d}_{X,Y,Z}
\setminus \mmZ^{p,q}_{X,Y,n,m},k)\oplus
\mmD_{\log}^*(\square^{n,m,d}_{X,Y,Z} \setminus
\mmZ^{p,l}_{X,Z,n,d},k)\oplus
\mmD_{\log}^*(\square^{n,m,d}_{X,Y,Z} \setminus
\mmZ^{q,l}_{Y,Z,m,d},k), $$ } and consider the sequence of
morphisms of complexes {\small $$ A^* \xrightarrow{i} B^*
\xrightarrow{j} \mmD_{\log}^*(\square^{n,m,d} \setminus
\mmZ^{p,q,l}_{X,Y,Z},k).
$$} By analogy with the definition of $s(-j^{p,q}_{X,Y}(n,m))^*$, denote by
$s(-j^{p,q,l}_{X,Y,Z}(n,m,d))^*$ the simple complex associated to
this sequence of morphisms. Consider
the morphism
\begin{eqnarray*}
\mmD_{\log}^*(\square^{n,m,d}_{X,Y,Z},k) &
\xrightarrow{i_{X,Y,Z}^{p,q,l}(n,m,d)} & s(-j^{p,q,l}_{X,Y,Z}(n,m,d))^* \\
\omega & \mapsto & (\omega,\omega,\omega,0,0,0,0).
\end{eqnarray*}
Observe that for every $n,m,d$, the simple of this morphism is a
cochain complex. Moreover, considering the normalized complex
associated to the cubical structure at every component of
$s(i_{X,Y,Z}^{p,q,l}({\cdot},{\cdot},{\cdot}))^*$, we obtain the
cochain complex  $s(i^{p,q,l}_{X,Y,Z})_0^*$ (analogous to the
construction of $s(i^{p,q}_{X,Y})_0^*$ in Remark \ref{si}).

Let $\mmD_{\A\times \A\times \A,\mmZ^{p,q,l}_{X,Y,Z}}^*(X\times
Y\times Z,p+q+l)_0$ be the complex analogous to $\mmD_{\A\times
\A,\mmZ^{p,q}_{X,Y}}^*(X\times Y,p+q)_0$, but with the cartesian
product of $3$ varieties. It is the simple complex associated to
the analogous $4$-iterated complex (see Remark \ref{si}).

Observe that there is a quasi-isomorphism
$$\mmD_{\A\times \A\times \A,\mmZ^{p,q,l}_{X,Y,Z}}^*(X\times Y\times
Z,p+q+l)_0\xrightarrow{\sim} s(i^{p,q,l}_{X,Y,Z})_0^*. $$

We define a pairing
 $$
s(i^{p,q}_{X,Y}(n,m))_0^r\otimes \mmD_{\A,\mmZ^l}^{s,d}(Z,l)_0
\xrightarrow{\bullet} s(i^{p,q,l}_{X,Y,Z}(n,m,d))_0^{r+s} $$ by
\begin{eqnarray*}
 (a,(b,c),d)\bullet (a',b') &=& (-1)^{(n+m)s}(a\bullet
 a',(b\bullet a',c\bullet a',(-1)^r a\bullet b'),\\ && \quad( d\bullet a',(-1)^{r-1} b\bullet b', (-1)^{r-1}c\bullet b'),
 (-1)^{r-2} d\bullet b').
\end{eqnarray*}
Define analogously a pairing
 $$
 \mmD^{r,n}_{\A,\mmZ^p}(X,p)_0 \otimes
s(i^{q,l}_{Y,Z}(m,d))_0^s\xrightarrow{\bullet}
s(i^{p,q,l}_{X,Y,Z}(n,m,d))_0^{r+s}
$$
by
\begin{eqnarray*}
 (a,b)\bullet (a',(b',c'),d') &=& (-1)^{ns}(a\bullet
 a',(b\bullet a',(-1)^r a\bullet b',(-1)^r a\bullet c'),\\ && \quad((-1)^{r-1} b\bullet b',(-1)^{r-1} b\bullet c',
 a\bullet d'),
   b\bullet d').
\end{eqnarray*}
It is easy to check that these two morphisms are chain morphisms.

\begin{lema}
The diagram {\small
\begin{equation}\label{commut11} \xymatrix@C=-35pt{ &
\mmD^{r}_{\A,\mmZ^p}(X,p)_0\otimes
\mmD^{s}_{\A,\mmZ^q}(Y,q)_0\otimes \mmD^{t}_{\A,\mmZ^l}(Z,l)_0
\ar[dl]_{\bullet_{p,q}\otimes id }
\ar[dr]^{id\otimes \bullet_{\A}} & \\
 s(i^{p,q}_{X,Y})_0^{r+s} \otimes \mmD^{t}_{\A,\mmZ^l}(Z,l)_0
\ar[dr]_{\bullet} & & \mmD^{r}_{\A,\mmZ^p}(X,p)_0 \otimes
s(i^{q,l}_{Y,Z})_0^{s+t} \ar[dl]^{\bullet } \\
& s(i^{p,q,l}_{X,Y,Z})_0^{r+s+t} }
\end{equation}}
is commutative up to homotopy.
\end{lema}
\begin{proof}
Let
$(\omega_1,g_1)
\in  \tau\mmD^{r}_{\log,\mmZ^p}(X\times \square^n,p)_0,(\omega_2,g_2)
\in  \tau\mmD^{s}_{\log,\mmZ^q}(Y\times \square^m,q)_0,$ and $
(\omega_3,g_3) \in  \tau\mmD^{t}_{\log,\mmZ^l}(Z\times
\square^d,l)_0.$
Then, the composition of the morphisms on the left side of the
diagram is
\begin{eqnarray*}
(-1)^{(n+m)t+ns}((\omega_1\bullet \omega_2)\bullet
\omega_3,((g_1\bullet \omega_2)\bullet \omega_3,(-1)^r
(\omega_1\bullet g_2)\bullet\omega_3, \\
(-1)^{r+s}(\omega_1\bullet \omega_2)\bullet g_3 ),
((-1)^{r-1}(g_1\bullet g_2)\bullet \omega_3,(-1)^{r+s-1}
(g_1\bullet \omega_2)\bullet g_3, \\ (-1)^{s-1} (\omega_1\bullet
g_2) \bullet g_3), (-1)^{s-1}(g_1\bullet g_2)\bullet g_3).
\end{eqnarray*}
The composition of the morphisms on the right side of the diagram
is
\begin{eqnarray*}
(-1)^{(n+m)t+ns}(\omega_1\bullet (\omega_2\bullet
\omega_3),(g_1\bullet( \omega_2\bullet \omega_3),(-1)^r
\omega_1\bullet ( g_2\bullet\omega_3), \\
(-1)^{r+s}\omega_1\bullet (\omega_2\bullet g_3) ),
((-1)^{r-1}g_1\bullet (g_2\bullet \omega_3),(-1)^{r+s-1}
g_1\bullet (\omega_2\bullet g_3), \\ (-1)^{s-1} \omega_1\bullet
(g_2 \bullet g_3)), (-1)^{s-1} g_1\bullet (g_2\bullet g_3)).
\end{eqnarray*}
Then, the homotopy for the commutativity of the diagram is given
by
\begin{eqnarray*}
H_{n,m,d}& = & (-1)^{(n+m)t+ns}((h(\omega_1\otimes
\omega_2\otimes \omega_3),h(g_1\otimes \omega_2\otimes\omega_3),\\
&& (-1)^r h(\omega_1\otimes g_2\otimes\omega_3),
(-1)^{r+s}h(\omega_1\otimes \omega_2\otimes g_3) ), \\ &&
((-1)^{r-1} h(g_1\otimes g_2\otimes \omega_3),(-1)^{r+s-1}
h(g_1\otimes
\omega_2\otimes g_3), \\
&&(-1)^{s-1} h(\omega_1\otimes g_2 \otimes g_3)), (-1)^{s-1}
h(g_1\otimes g_2\otimes g_3)).
\end{eqnarray*}
Observe that it gives indeed a homotopy, since $H$ and $\delta$
commute.
\end{proof}

Finally, the claim of Proposition \ref{assoc2} follows from the
 commutative diagram (all squares and triangles, apart from the one
 marked with $\#$ are strictly commutative),
{\small $$
\xymatrix@C=-85pt@R=12pt{ & \mmD^{r}_{\A,\mmZ^p}(X,p)_0\otimes
\mmD^{s}_{\A,\mmZ^q}(Y,q)_0\otimes \mmD^{t}_{\A,\mmZ^l}(Z,l)_0
\ar[dl]_(0.65){\bullet_{p,q}\otimes id }
\ar[dr]^(0.65){id\otimes \bullet_{q,l}} & \\
 s(i^{p,q}_{X,Y})_0^{r+s} \otimes \mmD^{t}_{\A,\mmZ^l}(Z,l)_0
\ar[dr]_{\bullet} & \# & \mmD^{r}_{\A,\mmZ^p}(X,p)_0 \otimes
s(i^{q,l}_{Y,Z})_0^{s+t} \ar[dl]^{\bullet } \\
& s(i^{p,q,l}_{X,Y,Z})_0^{r+s+t}  \ar@{=}[dd] \\
\mmD_{\A\times \A,\mmZ^{p,q}}^{r+s}(X\times Y,p+q)_0 \otimes
\mmD_{\A,\mmZ^{l}}^{t}(Z,l)_0 \ar[uu]^{\sim} \ar[dd] & &  \mmD_{
\A,\mmZ^{p}}^{r}(X,p)_0 \otimes \mmD_{\A\times
\A,\mmZ^{q,t}}^{s+l}(Y\times Z,q+l)_0 \ar[uu]_{\sim} \ar[dd] \\ &
s(i^{p,q,l}_{X,Y,Z})_0^{r+s+t} \ar[dr] \ar[dl]
\\
s(i^{p+q,l}_{X\times Y,Z})_0^{r+s+t}  && s(i^{p,q+l}_{X,Y\times
Z})_0^{r+s+t}  \\ & \mmD_{\A\times \A \times
\A,\mmZ^{p,q,l}}^{r+s+t}(X\times Y \times Z,p+q+l)_0 \ar[ur]
\ar[ul] \ar[uu]^{\sim} \ar[dd]^{\kappa} \ar[dr]^{\kappa} \ar[dl]_{\kappa} \\
\mmD_{\A\times \A,\mmZ^{p+q,l}}^{r+s+t}(X\times Y \times
Z,p+q+l)_0 \ar[dr]^{\kappa } \ar[uu]^{\sim} & & \mmD_{\A\times
\A,\mmZ^{p,q+l}}^{r+s+t}(X\times Y \times Z,p+q+l)_0
 \ar[dl]^{\kappa } \ar[uu]_{\sim}
\\ & \mmD_{\A,\mmZ^{p+q+l}}^{r+s+t}(X\times Y \times Z,p+q+l)_0.
 & } $$ }
\end{proof}

\begin{obs}\label{assoc3}
Observe that the homotopy constructed in the proof of Proposition
\ref{assoc2} has no component in maximal degree, that is, in
$\mmD^{2p+2q+2l}_{\A,\mmZ^{p+q+l}}(X\times Y\times Z,p+q+l)_0$.
\end{obs}

\begin{cor} Let $X,Y,Z$ be arithmetic varieties.
\begin{enumerate}[(i)]
\item Under the canonical isomorphism $(X\times Y)\times Z \cong
X\times (Y\times Z)$, the external pairing
$$ \widehat{CH}^p(*,n) \otimes \widehat{CH}^q(*,m) \otimes \xrightarrow{\cup} \widehat{CH}^{p+q}(*\times
*,n+m),$$ is associative. \item The internal pairing
$$ \widehat{CH}^p(X,n) \otimes \widehat{CH}^q(X,m) \xrightarrow{\cup} \widehat{CH}^{p+q}(X,n+m),$$
is associative.
\end{enumerate}
\end{cor}
\begin{proof}
It follows from \eqref{assoc1} and Proposition \ref{assoc2},
together with Remark \ref{assoc3} and the compatibility of the
homotopies in \eqref{assoc1} and Proposition \ref{assoc2}. For
$n=m=l=0$, the associativity follows from equality
\eqref{commutassoc2}.
\end{proof}

Finally, we have proved the following theorem.

\begin{theo}\label{chowproductth}
Let $X$ be an arithmetic variety  over an arithmetic field $K$.
Then,
$$\widehat{CH}^*(X,*):=\bigoplus_{p\geq 0,n\geq 0} \widehat{CH}^p(X,n) $$
is a commutative and associative ring with unity (graded commutative with
respect to the degree $n$ and commutative with respect to the
degree $p$). Moreover, the
morphism
$\widehat{CH}^*(X,*)  \xrightarrow{\zeta}  CH^*(X,*),
$
of Proposition \ref{zeta}, is a ring morphism.
\end{theo}

%\bibliographystyle{amsplain}
%\bibliography{bib}

\begin{thebibliography}{10}

\bibitem{Beilinson}
A.~A. Be{\u\i}linson, \emph{Notes on absolute {H}odge cohomology}, Applications
  of algebraic $K$-theory to algebraic geometry and number theory, Part I, II
  (Boulder, Colo., 1983), Contemp. Math., vol.~55, Amer. Math. Soc.,
  Providence, RI, 1986, pp.~35--68.

\bibitem{Bloch3}
S.~Bloch, \emph{Some notes on elementary properties of higher chow groups,
  including functoriality properties and cubical chow groups},
  http://www.math.uchicago.edu/~bloch/publications.html.

\bibitem{Bloch1}
\bysame, \emph{Algebraic cycles and higher {$K$}-theory}, Adv. in Math.
  \textbf{61} (1986), no.~3, 267--304.

\bibitem{Bloch4}
Spencer Bloch, \emph{Algebraic cycles and the {B}e\u\i linson conjectures}, The
  Lefschetz centennial conference, Part I (Mexico City, 1984), Contemp. Math.,
  vol.~58, Amer. Math. Soc., Providence, RI, 1986, pp.~65--79.

\bibitem{Burgos2}
J.~I. Burgos, \emph{Arithmetic {C}how rings and {D}eligne-{B}eilinson
  cohomology}, J. Alg. Geom. \textbf{6} (1997), 335--377.

\bibitem{BurgosKuhnKramer2}
J.~I. Burgos, J.~Kramer, and U.~K{\"u}hn, \emph{Arithmetic characteristic
  classes of automorphic vector bundles}, Doc. Math. \textbf{10} (2005),
  619--716 (electronic).

\bibitem{Burgos1}
J.~I. Burgos and S.~Wang, \emph{Higher {B}ott-{C}hern forms and {B}eilinson's
  regulator}, Invent. Math. \textbf{132} (1998), no.~2, 261--305.

\bibitem{BurgosKuhnKramer}
J.~I. Burgos~Gil, J.~Kramer, and U.~K{\"u}hn, \emph{Cohomological arithmetic
  {C}how rings}, J. Inst. Math. Jussieu \textbf{6} (2007), no.~1, 1--172.

\bibitem{DeligneHodgeII}
Pierre Deligne, \emph{Th\'eorie de {H}odge. {II}}, Inst. Hautes \'Etudes Sci.
  Publ. Math. (1971), no.~40, 5--57.

\bibitem{DoldPuppe}
A.~Dold and D.~Puppe, \emph{Homologie nicht-additiver {F}unktoren.
  {A}nwendungen}, Ann. Inst. Fourier Grenoble \textbf{11} (1961), 201--312.

\bibitem{Gillet}
H.~Gillet, \emph{Riemann-{R}och theorems for higher algebraic {$K$}-theory},
  Adv. in Math. \textbf{40} (1981), no.~3, 203--289.

\bibitem{GilletSouleIHES}
H.~Gillet and C.~Soul{\'e}, \emph{Arithmetic intersection theory}, Inst. Hautes
  \'Etudes Sci. Publ. Math. (1990), no.~72, 93--174 (1991).

\bibitem{GilletSouleFiltrations}
\bysame, \emph{Filtrations on higher algebraic {$K$}-theory}, Algebraic
  $K$-theory (Seattle, WA, 1997), Proc. Sympos. Pure Math., vol.~67, Amer.
  Math. Soc., Providence, RI, 1999, pp.~89--148.

\bibitem{Goncharov}
A.~B. Goncharov, \emph{Polylogarithms, regulators, and {A}rakelov motivic
  complexes}, J. Amer. Math. Soc. \textbf{18} (2005), no.~1, 1--60
  (electronic).

\bibitem{GrothChern}
Alexander Grothendieck, \emph{La th\'eorie des classes de {C}hern}, Bull. Soc.
  Math. France \textbf{86} (1958), 137--154.

\bibitem{AmalenduLevine}
Amalendu Krishna and Marc Levine, \emph{Additive higher {C}how groups of
  schemes}, J. Reine Angew. Math. \textbf{619} (2008), 75--140. \MR{MR2414948
  (2009d:14005)}

\bibitem{Levine1}
M.~Levine, \emph{Bloch's higher {C}how groups revisited}, Ast\'erisque (1994),
  no.~226, 10, 235--320, $K$-theory (Strasbourg, 1992).

\bibitem{Levine2}
\bysame, \emph{Mixed motives}, Mathematical Surveys and Monographs, vol.~57,
  American Mathematical Society, Providence, RI, 1998.

\bibitem{Levine3}
\bysame, \emph{Chow's moving lemma and the homotopy coniveau tower}, $K$-Theory
  \textbf{37} (2006), no.~1-2, 129--209. \MR{MR2274672 (2007m:19001)}

\bibitem{Quillen}
D.~Quillen, \emph{Higher algebraic {$K$}-theory. {I}}, Algebraic $K$-theory, I:
  Higher $K$-theories (Proc. Conf., Battelle Memorial Inst., Seattle, Wash.,
  1972), Springer, Berlin, 1973, pp.~85--147. Lecture Notes in Math., Vol. 341.

\end{thebibliography}

\providecommand{\bysame}{\leavevmode\hbox to3em{\hrulefill}\thinspace}
\providecommand{\MR}{\relax\ifhmode\unskip\space\fi MR }
\providecommand{\MRhref}[2]{%
  \href{http://www.ams.org/mathscinet-getitem?mr=#1}{#2}
}
\providecommand{\href}[2]{#2}

\end{document}